\documentclass[11pt, oneside]{article}

\usepackage[utf8]{inputenc}
\usepackage[T1]{fontenc}
\usepackage{xcolor}
\usepackage{color}

\usepackage{amsmath}
\usepackage{amssymb}
\usepackage{mathrsfs}
\usepackage{bm}

\usepackage{amsthm}
\theoremstyle{plain}
\newtheorem{theorem}{Theorem}[section]

\newtheorem{proposition}[theorem]{Proposition}
\newtheorem{lemma}[theorem]{Lemma}
\theoremstyle{definition}

\newtheorem{assumption}[theorem]{Assumption}

\newtheorem{example}[theorem]{Example}
\newtheorem{remark}[theorem]{Remark}

\usepackage[english]{babel}
\usepackage[round, authoryear]{natbib}
\usepackage{hyperref}
\hypersetup{
    pdfpagemode={UseOutlines},
    bookmarksopen,
    pdfstartview={FitH},
    colorlinks,
    linkcolor={blue},
    citecolor={blue},
    urlcolor={blue}
}
\usepackage[letterpaper,margin=1.25in]{geometry}
\usepackage{enumitem}
\def\iid{\overset{\textnormal{iid}}{\sim}}

\def\N{\mathbb{N}}\def\R{\mathbb{R}}\def\T{\mathbb{T}}\def\X{\mathbb{X}}\def\Z{\mathbb{Z}}

\def\Acal{\mathcal{A}}\def\Bcal{\mathcal{B}}\def\Ccal{\mathcal{C}}\def\Dcal{\mathcal{D}}\def\Ecal{\mathcal{E}}\def\Lcal{\mathcal{L}}\def\Ncal{\mathcal{N}}\def\Pcal{\mathcal{P}}\def\Vcal{\mathcal{V}}\def\Wcal{\mathcal{W}}\def\Xcal{\mathcal{X}}

\title{\bf Posterior contraction rates in Sobolev norms and Bayesian derivative estimation for infinite-dimensional exponential families}
\author{Emanuele Dolera$^*$, Stefano Favaro$^{\star,\dagger}$ and Matteo Giordano$^{\star,\dagger}$ \\ \\ 
$^*$University of Pavia, $^{\star}$University of Turin and $^{\dagger}$Collegio Carlo Alberto, Turin}
\date{} 

\begin{document}

\maketitle

\begin{abstract}
We study posterior contraction in positive-order Sobolev norms and Bayesian derivative estimation for infinite-dimensional exponential families. We embed the natural parameter in a Hilbert scale and model it via a Gaussian series prior expanded in the eigenbasis generating the scale. Under a two-sided link condition on the Fisher information and suitable local regularity assumptions, we show that smoothness-matching priors achieve minimax-optimal posterior contraction rates in any Hilbert scale norm up to the regularity of the ground truth. Our analysis builds on the novel approach to posterior contraction based on the Wasserstein distance recently introduced by \citet{dolera2024strong}. It combines Laplace-type approximations for infinite-dimensional integrals associated to the posterior kernels with a mixed-geometry estimate controlling their stability under fluctuations in the data, itself resting on a tailored Poincaré inequality for posterior distributions conditioned on neighbourhoods of the truth. We apply the general theory to density estimation under the logistic parametrisation, Poisson intensity estimation under the exponential link, and mildly ill-posed linear inverse problems observed in Gaussian white noise. The resulting minimax rates yield optimal recovery of density score functions, derivatives of Poisson intensities and, in a concrete elliptic inverse problem, the unknown source function.

\vspace{.25cm}

\noindent \textbf{MSC2020 subject classifications.} Primary 62G20, 62F15; secondary 62G05.

\noindent \textbf{Keywords.} Derivative estimation; frequentist analysis of posterior distributions; Gaussian priors; infinite-dimensional exponential families; Wasserstein distance.
\end{abstract}

\tableofcontents

\section{Introduction}
\label{Sec:Intro}

Nonparametric Bayesian methods have become a standard approach to inference in infinite-dimensional statistical models. A prior probability measure is placed on the unknown parameter and updated through the likelihood to form the posterior distribution, from which estimates, uncertainty quantification and predictions are obtained. Even when unavailable in closed form, the posterior can often be explored numerically using Markov chain Monte Carlo (MCMC) techniques. This unified treatment of the core tasks of statistical inference, together with the flexibility to encode structural information and the availability of advanced computational procedures, accounts for the widespread use of nonparametric Bayesian methods across the sciences. We refer to \citet{ghosal2017fundamentals}, and to \citet{hjort2010bayesian,muller2015bayesian,castillo2024bayesian}, for broad accounts of the subject.

The practical appeal of these methods makes a rigorous assessment of their performance particularly important. In infinite dimensions, genuinely subjective prior beliefs are typically difficult to elicit, and the prior is often chosen partly for its regularising properties, controlling features of the parameter that are only weakly informed by the likelihood. A natural question is therefore whether, from the frequentist perspective in which the data are generated by a fixed ground truth, a poorly calibrated prior can lead to suboptimal inference or even prevent recovery altogether. In regular finite-dimensional models, the Bernstein--von Mises theorem provides a broadly positive answer: as the sample size grows, the posterior contracts at the parametric rate, while credible sets are asymptotically valid confidence sets with optimal diameter \citep[Chapter 10]{vaart2000asymptotic}. In infinite-dimensional models, however, no analogous general principle holds, as demonstrated by classical counterexamples \citep{diaconis1986consistency,cox1993analysis,freedman1999bernstein}.

This issue has motivated an extensive literature on the frequentist analysis of nonparametric Bayesian methods. Early foundational consistency results were established by \citet{doob1949application} and \citet{schwartz1965bayes}, while their quantitative refinement through posterior contraction rates originated in the seminal works of \citet{ghosal2000convergence} and \citet{shen2001rates}. The resulting theory now encompasses a wide range of statistical models, priors, loss functions and analytical techniques \citep{ghosal2007convergence,walker2007rates,vaart2008rates,gine2011rates,rousseau2011asymptotic}; see also \citep[Chapter 7.3]{gine2016mathematical}, \citep[Chapters 8 and 9]{ghosal2017fundamentals} and \citet{castillo2024bayesian}.

A defining feature of this theory is that posterior contraction rates are typically established in statistical metrics determined by the testing geometry of the experiment, such as the Hellinger distance for independent and identically distributed (i.i.d.)~observations \citep[Chapter 8]{ghosal2017fundamentals}. To obtain recovery guarantees for the parameter itself, these metrics must then be related to the distance representing the inferential target. For several integrated losses, such as those of $L^1$- and $L^2$-type, this comparison is rate-preserving and yields minimax-optimal contraction \citep[Chapter 9]{ghosal2017fundamentals}. By contrast, for stronger metrics, such as the supremum norm or positive-order Sobolev distances, it generally requires inversion or interpolation inequalities and may lead to suboptimal results. Ad hoc analyses tailored to specific model--prior combinations can still yield optimal contraction rates \citep[e.g.,][]{castillo2014supremum,yoo2016supremum,shen2017posterior}. To the best of our knowledge, however, no comparably general theory is available for the frequentist analysis of posterior distributions in strong norms not directly controlled by the likelihood, for which uniformly exponentially consistent tests at the desired scale need not exist.

Here, we study posterior contraction in positive-order Sobolev norms, motivated by the task of inferring the derivatives of an unknown function. Such norms jointly control the distance between an estimate and its target and the discrepancies between their mixed partial derivatives, thereby guaranteeing simultaneous recovery of the parameter and its differential structure. Derivatives are themselves central objects in many applications: they describe marginal effects and local rates of change in regression \citep{hardle1989investigating}, determine modes and other shape features of probability density functions (p.d.f.'s) \citep{genovese2016non}, and underpin mean-shift methods for clustering, pattern recognition and image processing \citep{comaniciu2002mean}. Another prominent inferential target is the score function of an unknown p.d.f., whose estimation lies at the heart of score matching and modern score-based generative modelling, including diffusion models \citep{song2019generative}.

There is an extensive frequentist literature on derivative estimation, including foundational contributions by \citet{bhattacharya1967estimation,singh1977improvement,stone1982optimal}. More recently, minimax convergence rates under Sobolev norms were established by \citet{comte2020optimal} in density estimation and by \citet{fischer2020sobolev} in nonparametric regression; see also \citet{wibisono2024optimal} for score estimation. Results in Bayesian nonparametrics have emerged mainly over the past decade and remain largely model- and prior-specific, reflecting the difficulty of pursuing testing-based arguments in metrics stronger than the Hellinger distance. In multivariate nonparametric regression, \citet{yoo2016supremum} obtained minimax pointwise, $L^2$- and supremum-norm contraction rates for mixed partial derivatives using tensor-product B-spline priors. Adaptive rates were established by \citet{yoo2018adaptive} under wavelet spike-and-slab priors and, recently, by \citet{liu2026optimal} for plug-in differentiated Gaussian-process posteriors. In the Gaussian white-noise model, \citet{dolera2024besov} proved optimal contraction in Sobolev norms and derivative recovery using Besov--Laplace series priors; see also \citet{castillo2013nonparametric}. For density estimation, \citet{shen2017posterior} derived minimax $L^2$-contraction rates for density derivatives under B-spline series priors. Results for Gaussian sequence models and linear inverse problems \citep[e.g.,][]{belitser2003adaptive,knapik2011bayesian,gugushvili2020bayesian} can likewise accommodate derivative estimation by viewing differentiation as a mildly ill-posed inverse problem.

Recently, \citet{dolera2024strong} introduced an alternative, non-testing-based route to the frequentist analysis of posterior distributions in general metric spaces. Their approach is built on a characterisation of contraction rates through the Wasserstein distance \citep{ambrosio2008gradient} and its dynamic formulation \citep{benamou2000computational}, combined with Laplace-type integral approximations \citep{wong2001asymptotic}. However, when applied to study posterior concentration in Sobolev norms for infinite-dimensional exponential families, and for density estimation in particular, their analysis requires oracle Gaussian series priors whose covariance operators are diagonal in an eigenbasis of the Fisher information at the ground truth \citep[Section 3.1]{dolera2024strong}, and yields only suboptimal contraction rates \citep[Section 4.4]{dolera2024strong}.

%

\subsection{Contributions}

In the present paper, we build on the latter reference to study posterior contraction in Sobolev norms and derivative estimation for infinite-dimensional exponential families \citep{pistone1995infinite,canu2006kernel,fukumizu2009exponential}. These generalise classical parametric exponential families \citep{brown1986fundamentals} by allowing the natural parameter to range over an infinite-dimensional Hilbert space and the data to enter the likelihood through a sufficient statistic taking values in its topological dual; see Section \ref{Subsec:Model}. They retain the convenient analytic structure of their parametric counterparts: the first two derivatives of the convex log-partition function identify the mean map and Fisher information, while its associated Bregman divergence coincides with the Kullback--Leibler divergence. At the same time, they encompass important statistical models through a suitable specification of their natural parameter and sufficient statistic, including density estimation, Poisson intensity estimation and linear inverse problems observed in Gaussian white noise.

The asymptotic properties of frequentist estimators in infinite-dimensional exponential families have been investigated in several papers, notably in connection with score matching and kernel methods \citep{fukumizu2009exponential,gretton2012kernel,sriperumbudur2017density,fukumizu2024estimation}. The nonparametric Bayesian literature is comparatively sparse, with available contributions including the Bernstein--von Mises-type Gaussian approximations of \citet{shang2018gaussian} and the posterior contraction analysis of \citet{dolera2024strong}. Although a vast literature provides contraction rates for individual models encompassed by this framework, general techniques yielding optimal rates in positive-order Sobolev norms remain unavailable. Our work aims to help fill this gap.

To encode Sobolev-type norms and smoothness within the considered models, we embed the parameter space in a Hilbert scale $\{H_s\}_{s\in\R}$ \citep{krein1966scales}. This provides a unified framework to characterise the interplay among the strength of the loss, the regularity of the prior, the smoothness of the ground truth and the degree of ill-posedness of the statistical model. It also permits a concise formulation of the main assumptions through comparisons between the Fisher information and fixed powers of the scale generator. Hilbert scales are standard tools in the analysis of ill-posed problems \citep[e.g.,][]{natterer1984error} and Bayesian inversion \citep{knapik2011bayesian,knapik2016bayes,gugushvili2020bayesian}. In the concrete applications studied in Section \ref{Sec:Applications}, the resulting scales coincide, up to equivalent norms, with standard fractional Sobolev scales on bounded domains, while accommodating boundary or identifiability constraints.

We model the unknown parameter through Gaussian series priors built on the fixed eigenbasis generating the Hilbert scale, with regularity determined by the decay of the coordinate variances; see Section \ref{Subsec:Bayes}. Such priors form a canonical class for statistical models whose parameter space is a separable Hilbert space \citep[e.g.,][]{belitser2003adaptive,knapik2011bayesian,castillo2013nonparametric,knapik2016bayes,gugushvili2020bayesian}. In our main result, Theorem \ref{Theo:Main}, we show that if the ground truth belongs to $H_\beta$, the posterior arising from an $\alpha$-regular Gaussian series prior contracts in $H_s$-norm at rate $n^{-((\alpha\wedge\beta)-s)/(2\alpha+2\gamma+d)}$ for all $0\le s<\alpha\wedge\beta$, where $\gamma\ge0$ is the degree of ill-posedness and $d\in\N$ is the effective dimension encoded by the growth of the spectrum of the scale generator. For smoothness-matching priors, the rate becomes $n^{-(\beta-s)/(2\beta+2\gamma+d)}$, coinciding with the usual minimax benchmark; in direct models with $\gamma=0$, this recovers the classical rate $n^{-(\beta-s)/(2\beta+d)}$ \citep{stone1982optimal}.

Our proofs refine the approach of \citet{dolera2024strong}, which bypasses testing arguments by directly controlling the expected Wasserstein distance between the posterior and the Dirac mass at the ground truth, with transport cost induced by the metric of interest. This distance is decomposed into a deterministic concentration term, governed by the local interaction between likelihood and prior, and a stochastic stability term reflecting the sensitivity of the posterior to fluctuations in the data; see Section \ref{Subsec:WassApproach}. We control the deterministic term through a careful infinite-dimensional application of the Laplace methods for integral approximation \citep{wong2001asymptotic}, cf.~Appendix \ref{Subsec:DetTerm}. A two-sided Hilbert-scale link condition on the Fisher information removes the oracle requirement from \citet{dolera2024strong} that it be simultaneously diagonalised with the prior covariance. For the stochastic term, we prove a mixed-geometry stability estimate for the posterior kernels, which controls transport in the strong $H_s$-geometry while measuring sufficient-statistic fluctuations in a weaker geometry adapted to the full Hilbert scale. Its proof rests on a tailored Poincaré inequality for posterior distributions conditioned on a neighbourhood of the truth, established through Galerkin approximation and Brascamp--Lieb-type arguments \citep{bakry2008simple}. See Appendix \ref{Subsec:MixedTransport}. This decoupling avoids the loss in the rate incurred by the single-geometry analysis of \citet{dolera2024strong} and may be of independent interest.

In Section \ref{Sec:Applications}, we apply the general theory to three concrete statistical models. For density estimation under the logistic parametrisation (Section \ref{Subsec:DensityEstim}) and Poisson intensity estimation under the exponential link (Section \ref{Subsec:IntensityEstim}), smoothness-matching priors attain minimax contraction rates in Sobolev norms of every order below the regularity of the ground truth. The corresponding push-forward posteriors optimally recover derivatives of the p.d.f.~and intensity function. In density estimation, this includes minimax score estimation, or equivalently optimal contraction in Fisher divergence; for Poisson processes, our results appear to provide the first optimal posterior contraction rates for derivatives of a nonparametric intensity function. We then consider mildly ill-posed linear inverse problems in white noise (Section \ref{Subsec:InvProb}), obtaining minimax posterior contraction rates throughout the Hilbert scale. As an example, we study source recovery for elliptic partial differential equations using Gaussian priors constructed in the Dirichlet Laplacian eigenbasis. Unlike classical analyses based on the singular value decomposition (SVD) \citep[e.g.,][]{cavalier2008nonparametric}, our approach does not require the basis used to construct the Hilbert scale and the prior to diagonalise the normal operator.

%
%
%

\subsection{Basic notation and organisation of the paper}

For a given measure space $(\Xcal,\mathfrak{X},\mu)$ and a normed vector space $(\Vcal,\|\cdot\|_\Vcal)$, let $L^p(\Xcal,\mu;\Vcal)$, $p\in[1,\infty]$, be the usual Lebesgue space of $\Vcal$-valued functions, equipped, for $p<\infty$, with the norm defined by $\|h\|^p_{L^p(\Xcal,\mu;\Vcal)}:=\int_{\Xcal}\|h(x)\|_\Vcal^p\,d\mu(x)$, $h\in L^p(\Xcal,\mu;\Vcal)$, and by the essential supremum of $\|h(x)\|_\Vcal$ when $p=\infty$. For open $\Xcal\subseteq\R^m$, let $C(\Xcal)$ be the space of bounded continuous real-valued functions on $\Xcal$, equipped with the supremum norm. For positive integers $r\in\N$, denote by $C^r(\Xcal)$ the space of $r$-times continuously differentiable real-valued functions whose partial derivatives up to order $r$ are bounded, equipped with the norm $\|\cdot\|_{C^r}$. Further, let $W^{r,2}(\Xcal)$ be the usual Hilbert--Sobolev space of functions with square-integrable partial derivatives up to order $r$, with norm $\|\cdot\|_{W^{r,2}}$. When no confusion may arise, we omit the domain from the notation.

We denote the minimum and maximum between real numbers $a,b\in\R$ by $a\land b$ and $a\vee b$, respectively, and use the symbols $\lesssim, \ \gtrsim$ and $\asymp$ for one- and two-sided inequalities holding up to multiplicative constants.  For two real sequences $(a_n)_{n\ge1}$, $(b_n)_{n\ge1}$, we write $a_n=o(b_n)$ if $a_n/b_n\to0$ as $n\to\infty$, and $a_n=O(b_n)$ if $a_n/b_n\lesssim 1$ for all sufficiently large $n$.

The rest of the article is organised as follows. Section \ref{Sec:Setting} introduces Hilbert scales, infinite-dimensional exponential families and the Bayesian setup. Section \ref{Sec:GenPostRates} presents the Wasserstein dynamics approach, the model assumptions and the main contraction theorem. Section \ref{Sec:Applications} develops the applications to density estimation, Poisson intensity estimation and mildly ill-posed inverse problems, while Section \ref{Sec:Discussion} concludes with a discussion of extensions and open questions. Appendix \ref{Sec:ProofMain} proves the main theorem, while Appendices \ref{Sec:Proofs} and \ref{Sec:ProofApplications} contain the remaining auxiliary results and the proofs for the three applications.

%
%
%
%
%

\section{Infinite-dimensional exponential families in Hilbert scales}
\label{Sec:Setting}

We recall the definition of infinite-dimensional exponential families \citep{pistone1995infinite,canu2006kernel,fukumizu2009exponential}, casting it within the framework of Hilbert scales  \citep[e.g.,][]{krein1966scales}, where all model assumptions are made precise.

%
%
%

\subsection{Hilbert scales}
\label{Subsec:HilbertScales}

We assume the existence of a reference separable Hilbert space $H_0$, with inner product $\langle\cdot,\cdot\rangle_{H_0}$, norm $\|\cdot\|_{H_0}$ and orthonormal basis $\{e_k\}_{k=1}^\infty$, and fix a nondecreasing sequence of positive numbers $\{\lambda_k\}_{k=1}^\infty$ with $\lambda_k\to\infty$. For $s\ge 0$, define
$$
	H_s := \left\{ h = \sum_{k=1}^\infty h_k e_k :
	\|h\|_{H_s}^2 := \sum_{k=1}^\infty \lambda_k^{2s} h_k^2 <\infty\right\},
	\qquad
	\langle h, g\rangle_{H_s} := \sum_{k=1}^\infty \lambda_k^{2s}h_k g_k.
$$
For $s<0$ let $H_s$ be the completion of $H_0$ under $\|\cdot\|_{H_s}$, with $\|\cdot\|_{H_s}$ formally defined as above; equivalently, $H_s$ is the topological dual of $H_{-s}$ under the $\langle\cdot,\cdot\rangle_{H_0}$-pairing. The family $\{H_s\}_{s\in\R}$ is the \emph{Hilbert scale} generated by $\{(e_k,\lambda_k)\}_{k=1}^\infty$; see, e.g., \citet{krein1966scales,natterer1984error,engl1996regularization}. For every $s\in\R$, an orthonormal basis of $H_s$ is given by $\{\lambda_k^{-s}e_k\}_{k=1}^\infty$. We recall two standard properties used repeatedly: 
\begin{itemize}
	\item for $s<r$, $H_r$ is a dense and compactly embedded subspace of $H_s$;
	\item for $s_1\le s\le s_2$ with $s = (1-a)s_1+a s_2$, $a\in[0,1]$, we have the interpolation inequality:
\end{itemize}
	\begin{equation}
	\label{Eq:Interp}
		\|h\|_{H_s}\le \|h\|_{H_{s_1}}^{1-a}\|h\|_{H_{s_2}}^{a},
		\qquad h\in H_{s_2}.
	\end{equation}

Equivalently, $\{H_s\}_{s\in\R}$ is generated by the (unbounded) self-adjoint linear operator 
\begin{equation}
\label{Eq:ScaleGenerator}
	L e_k = \lambda_k e_k, \qquad k\in\N,
\end{equation}
on $H_0$, with domain $D(L)=H_1$. The real powers of $L$ are defined through spectral calculus: for $r\in\R$, $L^r e_k := \lambda_k^{r}e_k$, with domain $D(L^r)=\{h\in H_0:\sum_{k=1}^\infty\lambda_k^{2r}h_k^2<\infty\}$. Then, for any $s\in\R$, $L^r:H_s\to H_{s-r}$ is an isometric isomorphism; in particular, for $r>0$, one has $D(L^r) = H_r$ and $\|h\|_{H_r}=\|L^r h\|_{H_0}$, whereas the restriction of $L^{-r}$ to $H_0$ defines a bounded (in fact, compact) operator. In the following, we assume the polynomial growth
\begin{equation}
\label{Eq:PolyGrowth}
	\lambda_k \asymp k^{1/d}, \qquad \text{as}\ k\to\infty,
	\qquad\text{for some}\ d\in\N,
\end{equation}
matching, in view of Weyl's law \citep[e.g.,][Chapter 8]{taylor2011partial2}, the eigenvalue asymptotics of positive elliptic (pseudo-)differential operators of order one on $d$-dimensional domains. This recovers canonical scales of Hilbert-Sobolev spaces \citep{evans1998partial}, possibly up to equivalent norms and under appropriate boundary conditions.

\begin{example}[{Periodic Sobolev spaces}]
\label{Ex:TorusScale}
On the $d$-dimensional torus $\T^d$, interpreted as the unit hypercube $[0,1)^d$ with opposite endpoints identified, we take as reference $H_0=L^2(\T^d)$, the space of $1$-periodic square-integrable functions, with orthonormal basis given by the Fourier system $e_\xi(\cdot) := e^{2\pi i\langle \xi,\cdot\rangle}, \ \xi\in\Z^d$, namely the eigenfunctions of the Laplacian on $\T^d$,
$$
	-\Delta e_\xi := -\sum_{j=1}^d \frac{\partial^2}{\partial x_j^2}\, e_\xi = (2\pi|\xi|)^2 e_\xi,
	\quad e_\xi(\cdot+m) = e_\xi(\cdot),
	\quad m\in\Z^d,
	\quad \xi\in\Z^d.
$$
Noting that $e_0\equiv 1$ lies in the kernel of $-\Delta$, we take $L=(I-\Delta)^{1/2}$ as the scale generator, with eigenvalues $\lambda_k = (1+(2\pi|\xi_k|)^2)^{1/2}\ge1$. Enumerating the frequencies by increasing modulus as $\{\xi_k\}_{k\ge0}$, the lattice-point count $|\{\xi\in\Z^d:|\xi|\le u\}|\asymp u^d$ gives $\lambda_k \asymp k^{1/d}$, so that the polynomial growth \eqref{Eq:PolyGrowth} holds with $d$ the dimension of $\T^d$. The induced scale is the periodic Hilbert-Sobolev scale $H_s = H^s(\T^d)$, $s\in\R$, with
$$
	\|h\|_{H^s}^2 = \sum_{\xi\in\Z^d} (1+(2\pi|\xi|)^2)^{s}\,|\hat h(\xi)|^2 = \big\|(I-\Delta)^{s/2}h\big\|_{L^2}^2 ,
	\qquad \hat h(\xi) := \langle h,e_\xi\rangle_{L^2}.
$$
For nonnegative integers $s\in\N_0$, $H^s(\T^d)$ coincides, with equivalent norms, with the standard periodic Sobolev space $W^{s,2}(\T^d)$ comprising all functions $h$ with partial derivatives $D^{\alpha}h\in L^2(\T^d)$ for all multi-indices $\alpha = (\alpha_1,\dots,\alpha_d)\in\N_0^d$ such that $|\alpha|\le s$, normed by $\|h\|^2_{W^{s,2}} = \sum_{|\alpha|\le s}\|D^{\alpha}h\|_{L^2}^2$. This follows from differentiating term by term the Fourier series representation of any function $h\in H^s(\T^d)$. In fact, for $s=0,1$, the norms coincide.
\end{example}

Different boundary behaviours can be obtained through appropriate self-adjoint realisations of the Laplacian, cf.~Section \ref{Subsec:InvProb}. Below, we also employ wavelet-based Hilbert scales.

\begin{example}[{Wavelet Sobolev scale on $[0,1]^d$}]
\label{Ex:WaveletScale}
Let $\{\psi_k\}_{k=1}^\infty$ be a single-index reordering by increasing resolution level of a tensor-product orthonormal wavelet basis of $L^2([0,1]^d)$, built from $S$-regular ($S\in\N$), compactly supported, boundary-corrected Daubechies wavelets \citep[see, e.g.,][Chapter 4.3]{gine2016mathematical}. Setting $\lambda_k := k^{1/d}$, the induced scale $H_s = H^s([0,1]^d)$, normed by $\|h\|_{H^s}^2 = \sum_{k=1}^\infty k^{2s/d}\,|\langle h,\psi_k\rangle_{L^2}|^2$,
satisfies, in view of the wavelet characterisation of standard Sobolev spaces, $H^s([0,1]^d) = W^{s,2}([0,1]^d)$ and $\|\cdot \|_{H^s}^2\asymp \|\cdot\|_{W^{s,2}}^2$ for all integers $0\le s< S$, cf.~\cite[p.~371f]{gine2016mathematical}.
\end{example}

Further subspaces contained within the underlying Hilbert scale can be singled out through the series representation. For example, as the constant function is a basis element in both Examples \ref{Ex:TorusScale} and \ref{Ex:WaveletScale}, corresponding respectively to the zero-frequency mode $e_0$ and the coarsest scaling function $\psi_1$, 
removing it imposes the mean-zero constraint and gives rise to the homogeneous spaces $\dot H^s(\T^d)$ and $\dot H^s([0,1]^d)$, $s\in\R$. This provides natural parameter spaces for models identifiable only up to an additive constant; see Section \ref{Subsec:DensityEstim}.

%
%
%

\subsection{Infinite-dimensional exponential families}
\label{Subsec:Model}

Fix $p\ge0$ and take as parameter space the Hilbert space $\Theta := H_p$. In the rest of the article, we write $\langle\cdot,\cdot\rangle_\Theta := \langle\cdot,\cdot\rangle_{H_p}$ and $\|\cdot\|_\Theta := \|\cdot\|_{H_p}$. Let $(\X,\Xcal,\lambda)$ be a $\sigma$-finite measure space and let $\beta:\X\to\Theta^* = H_{-p}$ be a measurable map. Following \citet{pistone1995infinite,canu2006kernel,fukumizu2009exponential} (see also Definition 3 of \citealp{dolera2024strong}), the associated \emph{regular infinite-dimensional exponential family} is the collection of $\lambda$-densities
\begin{equation}
\label{Eq:ExpFamily}
	f_\theta(x) := \exp\big\{\beta_x(\theta) - M(\theta)\big\},
	\quad
	M(\theta):=\log\int_\X e^{\beta_x(\theta)}\,d\lambda(x),
	\quad x\in\X,
	\quad \theta\in\Theta,
\end{equation}
where $\beta_x(\theta)$ denotes the action of the linear functional $\beta_x\in \Theta^*$ on $\theta\in\Theta$, and the log-partition (or cumulant generating) function $M$ is assumed to be finite for all $\theta\in\Theta$. The family \eqref{Eq:ExpFamily} is expressed in natural form, with natural parameter space equal to $\Theta$.

We assume we are given a random sample $X^{(n)}:=(X_1,\dots,X_n)$ of independent and identically distributed (i.i.d.) $\X$-valued random variables drawn from an unknown probability density function (p.d.f.) belonging to the exponential family,
\begin{equation}
\label{Eq:Data}
	X_1,\dots,X_n \iid f_{\theta},
\end{equation}
for some $\theta\in\Theta$ and $f_\theta$ as in \eqref{Eq:ExpFamily}. The goal is to estimate $\theta$ from observations $X^{(n)}$. We use the notation $d\mu_\theta := f_{\theta}\,d\lambda$, $\lambda^{(n)} :=\lambda^{\otimes n}$, and write $\mu_\theta^{(n)}:=\mu_\theta^{\otimes n}$ for the data distribution and $\operatorname{E}_\theta^{(n)}$ for the expectation with respect to $\mu_\theta^{(n)}$. The likelihood equals
\begin{equation}
\label{Eq:Likelihood}
	L^{(n)}(\theta) := \frac{d\mu_\theta^{(n)}}{d\lambda^{(n)}}(X^{(n)}) 
	= \exp\left\{ n\left[\hat T_n(\theta) 
	- M(\theta)\right]\right\},
	\qquad \theta\in\Theta,
\end{equation}
with $\Theta^*$-valued sufficient statistic $\hat T_n := \frac1n\sum_{i=1}^n\beta_{X_i}$. Infinite-dimensional exponential families provide a flexible yet analytically tractable framework for nonparametric inference, encompassing several widely used statistical models that are recovered through different choices of the map $\beta$ and dominating measure $\lambda$. Concrete examples are treated in Section \ref{Sec:Applications}.

%
%
%

\subsection{The Bayesian approach}
\label{Subsec:Bayes}

We adopt the Bayesian approach to the problem of estimating the unknown parameter $\theta$ from data $X^{(n)}$ arising as in \eqref{Eq:Data}, modelling $\theta$ via a (prior) Borel probability measure $\Pi$ on $\Theta$. The log-partition function $M$ in \eqref{Eq:ExpFamily} is convex and lower semicontinuous on $\Theta$ \citep[cf.][Theorems 1.13, 2.2 and 2.7]{brown1986fundamentals}; hence, $M$ is Borel measurable, and so is the likelihood $L^{(n)}$ in \eqref{Eq:Likelihood}. By Bayes' formula \citep[e.g.,][p.~7]{ghosal2017fundamentals}, the resulting posterior distribution equals
\begin{equation}
\label{Eq:Posterior}
	\Pi(A|X^{(n)})
	=\frac{\int_A e^{n[\hat T_n(\theta) - M(\theta)]}\,d\Pi(\theta)}
	{\int_\Theta e^{n[\hat T_n(\theta) - M(\theta)]}\,d\Pi(\theta)},
	\qquad A\subseteq\Theta \ \text{measurable}.
\end{equation}
In writing \eqref{Eq:Posterior} we implicitly assume the denominator to be finite, namely that
\begin{equation}
\label{Eq:NormConst}
	\int_\Theta e^{n[\tau(\theta) - M(\theta)]}\,d\Pi(\theta)<\infty,
	\qquad \text{for all}\ n\in\N\ \text{and}\ \tau\in\Theta^*.
\end{equation}
The denominator is strictly positive in view of the model assumption $M(\theta)<\infty$ for every $\theta\in\Theta$; hence, the posterior $\Pi(\cdot|X^{(n)})$ in \eqref{Eq:Posterior} is well defined under \eqref{Eq:NormConst}. For $\tau\in\Theta^*$, we use the notation $\Pi_n^*(\cdot|\tau)$ for the probability measure defined by the right-hand side of \eqref{Eq:Posterior} with $\hat T_n$ replaced by $\tau$. Accordingly, we write $\Pi(\cdot|X^{(n)}) = \Pi_n^*(\cdot|\hat T_n)$.

In the following, we investigate the asymptotic behaviour of the posterior distribution as the sample size $n\to\infty$, under the frequentist assumption that the data have been generated by some fixed \emph{ground truth} $\theta_0\in\Theta$, namely $X^{(n)}\sim \mu_{\theta_0}^{(n)}$. We characterise the speed of concentration of $\Pi(\cdot|X^{(n)})$ around $\theta_0$ according to the usual notion of \emph{posterior contraction rates}, that is, positive real sequences $\varepsilon_n\to0$ as $n\to\infty$ such that, for every $M_n\to\infty$, 
\begin{equation}
\label{Eq:PCRDef}
	\operatorname{E}_{\theta_0}^{(n)}\left[\Pi\big(\theta\in\Theta:\, \ell(\theta,\theta_0)>M_n\varepsilon_n \,\big|\, X^{(n)}\big)\right]
	\to 0,
	\qquad \text{as}\ n\to\infty,
\end{equation}
where $\ell$ is a (semi-)metric on $\Theta$. See \citet{ghosal2017fundamentals} for an in-depth treatment of this research area. In particular, we are interested in posterior contraction rates relative to the Hilbert scale norms, taking $\ell(\theta_1,\theta_2) = \|\theta_1-\theta_2\|_{H_s}$ for suitable $s\ge0$. For the concrete scales of Examples \ref{Ex:TorusScale} and \ref{Ex:WaveletScale}, the resulting metrics are (equivalent to) Sobolev norms, so that contraction in $\|\cdot\|_{H_s}$ quantifies the recovery of $\theta_0$ jointly with its derivatives up to order $s$, the case $s=0$ corresponding to the usual $L^2$-distance.

The concentration in strong norms of Bayesian procedures for infinite-dimensional exponential families has recently been investigated by \citet{dolera2024strong}. However, their analysis hinges on oracle-type Gaussian priors calibrated to features of the unknown ground truth, and only results in suboptimal rates in the present setting. An earlier contribution, developing Bernstein--von Mises-type approximations, is by \citet{shang2018gaussian}. Results for frequentist estimators and tests in this framework were derived by \citet{fukumizu2009exponential,gretton2012kernel,sriperumbudur2017density}, among others.

There is an extensive nonparametric Bayesian literature for concrete instances of infinite-dimensional families. Landmark contributions include \citet{scricciolo2006convergence,vaart2008rates,rivoirard2012posterior} for density estimation under the logistic parametrisation, \citet{belitser2015rate,kirichenko2015optimality,donnet2017posterior} for Poisson intensity estimation, and \citet{belitser2003adaptive,knapik2011bayesian,castillo2013nonparametric} for the white noise model. The monograph by \citet{ghosal2017fundamentals} provides many further references.

%
%
%
%
%

\section{Posterior contraction rates by Wasserstein dynamics}
\label{Sec:GenPostRates}

We provide our main result concerning the concentration in Hilbert scale norms of the posterior distribution $\Pi(\cdot|X^{(n)})$ in \eqref{Eq:Posterior} around the ground truth $\theta_0$ generating the data \eqref{Eq:Data}. Our analysis is based on the novel \emph{Wasserstein dynamics} approach to posterior contraction rates in general metric spaces recently developed by \citet{dolera2024strong}, which we significantly refine in the context of infinite-dimensional exponential families. To clarify the model assumptions underlying the abstract theorem, we outline key elements of the proof in the next section.

%
%
%

\subsection{The Wasserstein dynamics approach to posterior contraction rates}
\label{Subsec:WassApproach}

The starting point is a general upper bound for posterior contraction rates expressed in terms of a Wasserstein distance. Recall $\Theta = H_p$ for some $p\ge0$, and let $s\ge0$ be the index of the target Hilbert scale norm for posterior contraction. Assume that $\theta_0\in H_\beta$ for some $\beta > s\vee p$, and that the prior (and hence the posterior) is supported on $H_s\cap \Theta = H_{s\vee p}$. Let $\Pcal_2(H_s)$ be the collection of all Borel probability measures $\nu$ on $H_s$ satisfying
\begin{equation}
\label{Eq:Integr}
	\int_{H_s} \|\theta - \theta^*\|_{H_s}^2d\nu(\theta)<\infty,
	\qquad \text{for some}\ \theta^*\in H_s,
\end{equation}
and consider the $2$-\emph{Wasserstein distance}
$$
	\Wcal_2^{\Pcal_2(H_s)}(\nu_1,\nu_2)
	:= \sqrt{\inf_{\varpi\in\mathcal F(\nu_1,\nu_2)} 
	\int_{H_s\times H_s}\|\theta_1-\theta_2\|_{H_s}^2\,d\varpi(\theta_1,\theta_2)},
$$
where $\mathcal F(\nu_1,\nu_2)$ is the collection of all couplings of $\nu_1$ and $\nu_2$, that is, Borel probability measures on $H_s\times H_s$ with marginals $\nu_1$ and $\nu_2$; see \citep[][Chapter 7]{ambrosio2008gradient} for details. Under a mild integrability condition ensuring that the prior and the posterior satisfy \eqref{Eq:Integr}, Lemma 1 of \citet{dolera2024strong} yields that the quantity
\begin{equation}
\label{Eq:WassPCR}
	\varepsilon_n := \operatorname{E}_{\theta_0}^{(n)}\big[\Wcal_2^{\Pcal_2(H_s)}\big(\Pi(\cdot|X^{(n)}),\delta_{\theta_0}\big)\big],
\end{equation}
where $\delta_{\theta_0}$ denotes the Dirac mass at $\theta_0$, is an upper bound for the $H_s$-posterior contraction rate at $\theta_0$, in that \eqref{Eq:PCRDef} holds for every $M_n\to\infty$ as $n\to\infty$. We then proceed to bound $\varepsilon_n$ in \eqref{Eq:WassPCR} through a combination of analytical and probabilistic techniques.

First, recall from Section \ref{Subsec:Bayes} that $\Pi(\cdot|X^{(n)}) = \Pi_n^*(\cdot|\hat T_n)$ with $\hat T_n = \tfrac1n\sum_{i=1}^n\beta_{X_i}$ the sufficient statistic in \eqref{Eq:Likelihood}. For all $\theta\in\Theta$ such that $\operatorname{E}_{\theta}^{(1)}[\|\beta_{X_1}\|_{\Theta^*}]<\infty$, denote the Bochner integral
\begin{equation}
\label{Eq:MeanMap}
	T_\theta:= \operatorname{E}^{(1)}_{\theta} [\beta_{X_1}] = \int_\X \beta_x f_{\theta}(x)\, d\lambda (x),
\end{equation}
which belongs to $\Theta^*$. Then, provided that $\operatorname{E}_{\theta_0}^{(1)}[\|\beta_{X_1}\|_{\Theta^*}]<\infty$, the strong law of large numbers in Hilbert spaces \citep[][Corollary 7.10]{ledoux1991probability} implies that $\|\hat T_n - T_{\theta_0}\|_{\Theta^*}\to0$ $\mu_{\theta_0}^{(\infty)}$-almost surely as $n\to\infty$. This suggests that the asymptotic behaviour of $\Pi_n^*(\cdot|\hat T_n)$ should be close to that of the deterministic measure $\Pi_n^*(\cdot|T_{\theta_0})$ obtained by replacing $\hat T_n$ in \eqref{Eq:Posterior} with its limit $T_{\theta_0}$, as long as the Markov kernel $\Pi_n^*(\cdot|\tau)$, $\tau\in\Theta^*$, is stable with respect to its second argument. Motivated by this observation, we apply the triangle inequality for the Wasserstein distance to upper bound $\varepsilon_n$ in \eqref{Eq:WassPCR} by
\begin{equation}
\label{Eq:WassDecomp}
\begin{split}
	\varepsilon_n&\le\Wcal_2^{\Pcal_2(H_s)}\big(\Pi_n^*(\cdot|T_{\theta_0}),\delta_{\theta_0}\big)
	\;+\;
	\operatorname{E}_{\theta_0}^{(n)}\Big[\Wcal_2^{\Pcal_2(H_s)}\big(\Pi_n^*(\cdot|T_{\theta_0}),\Pi_n^*(\cdot|\hat T_n)\big)\Big].
\end{split}
\end{equation}

The first summand in \eqref{Eq:WassDecomp} is a purely \emph{deterministic term} that measures the concentration of $\Pi_n^*(\cdot|T_{\theta_0})$ at $\theta_0$. Its analytical evaluation involves the asymptotics of the ratio between two vanishing integrals, cf.~\eqref{Eq:KLRepres} below, which we tackle via an infinite-dimensional extension of the classical Laplace methods for integral approximation. This is based on earlier ideas of \citet{dolera2024strong}, and its present application in the context of infinite-dimensional exponential families may be of independent interest. See Appendix \ref{Subsec:DetTerm}. Specifically, we rewrite $\Pi_n^*(\cdot|T_{\theta_0})$ in terms of the Kullback--Leibler divergence,
\begin{equation}
\label{Eq:KL}
	K(\vartheta|\theta) := \int_\X \log\frac{f_{\theta}(x)}{f_\vartheta(x)}\,f_{\theta}(x)\,d\lambda(x),
	\qquad \vartheta,\theta\in\Theta.
\end{equation}
Recalling \eqref{Eq:MeanMap}, and since $\log[f_{\theta}(x)/f_\vartheta(x)] = \beta_x(\theta-\vartheta) + M(\vartheta)-M(\theta)$, we obtain
\begin{equation}
\label{Eq:Bregman}
	K(\vartheta|\theta) 
	= M(\vartheta)-M(\theta) - \int_\X  \beta_x(\vartheta-\theta) \,f_{\theta}(x)\,d\lambda(x)
	= M(\vartheta)-M(\theta)-T_\theta(\vartheta-\theta),
\end{equation}
the interchange of the integral and the continuous linear functional being justified by the standard properties of Bochner integrals. Thus, $-nK(\theta|\theta_0) = n[T_{\theta_0}(\theta)-M(\theta)] + n[M(\theta_0)-T_{\theta_0}(\theta_0)]$ for any $\theta\in\Theta$, and since the second term does not depend on $\theta$, we obtain from \eqref{Eq:Posterior},
$$
	d\Pi_n^*(\theta|T_{\theta_0})
	=\frac{e^{-nK(\theta|\theta_0)}}
	{\int_{H_{s\vee p}} e^{-nK(\theta|\theta_0)}\,d\Pi(\theta)}\,d\Pi(\theta),
	\qquad \theta\in H_{s\vee p}.
$$
From this and from a standard identity for the Wasserstein distance of a probability measure from a Dirac mass \citep[eq.~(11)]{dolera2024strong}, we obtain  
\begin{equation}
\label{Eq:KLRepres}
\begin{split}
	\Wcal_2^{\Pcal_2(H_s)}\big(\Pi_n^*(\cdot|T_{\theta_0}),\delta_{\theta_0}\big)
	&=\sqrt{ \int_{H_{s\vee p}} \|\theta - \theta_0\|_{H_s}^2\,d\Pi_n^*(\theta|T_{\theta_0})}\\
	&=\sqrt{ \frac{\int_{H_{s\vee p}} \|\theta-\theta_0\|_{H_s}^2 e^{-nK(\theta|\theta_0)}\,d\Pi(\theta)}
	{\int_{H_{s\vee p}} e^{-nK(\theta|\theta_0)}\,d\Pi(\theta)} }.
\end{split}
\end{equation}

The second summand in \eqref{Eq:WassDecomp} is a \emph{stochastic term} measuring the effect on the posterior of the fluctuation of $\hat T_n$ around $T_{\theta_0}$. Its analysis rests on two main ingredients. First, we construct a localising event whose probability tends to one exponentially fast, on which $\hat T_n-T_{\theta_0}$ is small and both posterior kernels in the stochastic term can be replaced, at negligible cost, by their conditionings to a fixed neighbourhood of $\theta_0$. On the complementary event, a crude global transport estimate combined with a second-moment bound is sufficient. On the localising event, we connect $T_{\theta_0}$ and $\hat T_n$ through a linear path $\{\tau_t\}_{t\in[0,1]}$, and apply the dynamic formulation of the Wasserstein distance \citep{benamou2000computational} to the corresponding curve of localised posterior kernels. The exponential-family structure identifies the derivative along this curve with a covariance involving $\hat T_n-T_{\theta_0}$, which we control through a tailored Poincar\'e inequality derived from a local lower bound on the posterior curvature stemming from a local coercivity assumption on the Fisher information. Crucially, the resulting estimate controls the transported test functions in the target $H_s$-geometry while measuring the sufficient-statistic fluctuation through a preconditioned quadratic form adapted to the full Hilbert scale. This mixed-geometry estimate builds on the stability techniques of \citet{dolera2020uniform,dolera2023lipschitz} and provides a decisive improvement over the single-geometry analysis of \citet{dolera2024strong}, whose application in the present setting results in a suboptimal rate. See Appendix \ref{Subsec:MixedTransport} for details.

%
%
%

\subsection{A general posterior contraction rate theorem for Gaussian priors}
\label{Subsec:GenThm}

We are now in a position to present our main posterior contraction rate theorem. We assume that $\theta_0\in H_\beta$ for some $\beta>p$, and consider a centred Gaussian series prior $\Pi$ on the orthonormal basis $\{e_k\}_{k=1}^\infty$ of $H_0$ defined, for a given prior regularity parameter $\alpha>p$, as the law of
\begin{equation}
\label{Eq:Prior}
	\theta = \sum_{k=1}^\infty \lambda_k^{-\alpha-d/2}\,t_k\,e_k,
	\qquad t_k\iid \Ncal(0,1).
\end{equation}
The deterministic scaling in \eqref{Eq:Prior} calibrates the prior regularity within the Hilbert scale. Indeed, in view of the assumed eigenvalue asymptotics \eqref{Eq:PolyGrowth}, $\operatorname{E}[\|\theta\|_{H_s}^2] = \sum_{k=1}^\infty \lambda_k^{2(s-\alpha)-d}$ $<\infty$ if and only if $s<\alpha$; accordingly, the draws
from $\Pi$ belong almost surely to $H_{s}$ for every $s<\alpha$, while they almost surely belong to no $H_s$ with $s\ge\alpha$. As customary, we call $\Pi$ an \emph{$\alpha$-regular Gaussian series prior}. In particular, $\Pi$ is supported on $\Theta = H_p$ since $\alpha>p$.

Equivalently, for $L$ the scale generator from \eqref{Eq:ScaleGenerator}, $\Pi$ is a centred Gaussian Borel probability measure on $H_0$, with covariance operator $Q_\Pi = L^{-2\alpha-d}$. The associated reproducing kernel Hilbert space (RKHS, or Cameron--Martin space) is given by $H_\Pi = H_{\alpha+d/2}$, with $\|h\|_{H_\Pi}^2 = \langle h,Q^{-1}_\Pi h\rangle_{H_0}= \|h\|_{H_{\alpha+d/2}}^2$. We refer to Chapter 2 in \cite{gine2016mathematical} for definitions and background properties of Gaussian processes and measures used throughout.

For $\theta_0$ and $\Pi$ as above, we pursue $H_s$-posterior contraction rates via the Wasserstein dynamics approach outlined in Section \ref{Subsec:WassApproach}. The following model assumptions formalise the conditions under which explicit, matching algebraic rates can be obtained for the deterministic and stochastic terms appearing in the upper bound \eqref{Eq:WassDecomp}.

%
%
%

\subsubsection{Model assumptions}
\label{Subsec:ModelAss}

The main assumptions are formulated in terms of the \emph{Fisher information}, defined, for every $\theta\in\Theta$ for which the mean functional $T_\theta$ from \eqref{Eq:MeanMap} exists, as the following nonnegative (extended) quadratic form
\begin{equation}
\label{Eq:FisherInfo}
	I(\theta)[h] := \operatorname{Var}^{(1)}_\theta\big(\beta_{X_1}(h)\big)
	= \operatorname{E}^{(1)}_\theta\Big[\big(\beta_{X_1}(h)-T_\theta(h)\big)^2\Big]\in[0,\infty],
	\qquad h\in\Theta.
\end{equation}
At any $\theta\in\Theta$ where the log-partition function $M$ from \eqref{Eq:ExpFamily} is twice (Fr\'echet) differentiable, the identity \eqref{Eq:Bregman} identifies $T_\theta$ and $I(\theta)$ with the first two derivatives of $M(\theta)$, 
whence, for $K(\vartheta|\theta)$, $\theta,\vartheta\in\Theta$, the Kullback--Leibler divergence defined in \eqref{Eq:KL}, 
\begin{equation}
\label{Eq:KLDerivatives}
	D_\vartheta K(\vartheta|\theta) = T_\vartheta - T_\theta,
	\qquad
	D^2_\vartheta K(\vartheta|\theta)[h,h] = I(\vartheta)[h].
\end{equation}
The second identity in \eqref{Eq:KLDerivatives} characterises $I$ as the curvature of the Kullback--Leibler divergence in its first argument. Via the representation \eqref{Eq:KLRepres}, this regulates the asymptotics of the deterministic term in \eqref{Eq:WassDecomp}. At the same time, through \eqref{Eq:FisherInfo}, $I(\theta)$ also governs the fluctuations of the sufficient statistic $\hat T_n$ around $T_\theta$, as $\operatorname{Var}^{(n)}_{\theta}(\hat T_n(h)) = \frac1n I(\theta)[h]$. Thus, it is also key in understanding the decay of the stochastic term in \eqref{Eq:WassDecomp}.

The following central assumption requires a pointwise upper bound and a local coercivity bound for the Fisher information, both expressed in the Hilbert-scale geometry.

\begin{assumption}[{Coercivity and boundedness of the Fisher information}]
\label{Ass:Link}
There exists a constant $R_0>0$ such that $\operatorname{E}_{\theta}^{(1)}[\|\beta_{X_1}\|_{\Theta^*}]<\infty$ for all $\theta\in B_{R_0}^\Theta(\theta_0):=\{\theta\in\Theta:\|\theta-\theta_0\|_\Theta\le R_0\}$. Furthermore, there exist $\gamma\ge0$ and $0<c_0\le C_0<\infty$ such that
\begin{equation}
\label{Eq:LinkUpper}
	I(\theta_0)[h] \;\le\; C_0\|h\|^2_{H_{-\gamma}},
	\qquad\text{for all}\ h\in\Theta;
\end{equation}
\begin{equation}
\label{Eq:LinkLower}
	I(\theta)[h] \;\ge\; c_0\|h\|^2_{H_{-\gamma}},
	\qquad \text{for all}\ h\in\Theta\
	\text{and}\ \theta\in B_{R_0}^\Theta(\theta_0).
\end{equation}
\end{assumption}

The upper bound \eqref{Eq:LinkUpper} guarantees that $I(\theta_0)$ is finite and bounded on $\Theta$, since $\|h\|_{H_{-\gamma}}\le\lambda_1^{-(p+\gamma)}\|h\|_\Theta$. Hence, it induces a bounded symmetric bilinear form on $\Theta\times\Theta$ and a unique bounded operator $I(\theta_0):\Theta\to\Theta^*$ satisfying $\big(I(\theta_0)[h]\big)(g)= \operatorname{Cov}^{(1)}_{\theta_0}\big(\beta_{X_1}(g),\,\beta_{X_1}(h)\big)$, $g,h\in\Theta$. The index $\gamma$ quantifies the ill-posedness, in that in view of \eqref{Eq:LinkUpper} and since $\|h\|_{H_{-\gamma}} = \|L^{-\gamma} h\|_{H_0}$, $I(\theta_0)$ can be understood to be $2\gamma$-smoothing relative to the $\langle\cdot,\cdot\rangle_{H_0}$-pairing.

Assumption \ref{Ass:Link} entails that $c_0 L^{-2(p+\gamma)}\;\preceq\; I(\theta_0) \;\preceq\; C_0 L^{-2(p+\gamma)}$
. The comparison operator $L^{-2(p+\gamma)}$ is diagonal in the orthonormal basis $\{\lambda_k^{-p}e_k\}_{k=1}^{\infty}$ of $\Theta$, 
and hence in the same fixed Hilbert-scale basis on which the Gaussian series prior is constructed. This is crucial for bounding both the deterministic and stochastic terms in \eqref{Eq:WassDecomp}, allowing for tight control of certain Gaussian quadratic form ratios appearing in the analysis.

The second assumption formalises the differentiability of the Kullback--Leibler divergence and its link to the Fisher information via \eqref{Eq:KLDerivatives}, imposing a second-order Taylor expansion.

\begin{assumption}[{Expansion of the Kullback--Leibler divergence}]
\label{Ass:KLTaylor}
$M$ is twice continuously Fr\'echet differentiable on $\Theta$. Furthermore, there exists an index $m_K\in[0,p]$ and constants $q_K\in(0,1]$ and $C_K,R_K>0$ such that
\begin{equation}
\label{Eq:KLTaylor}
	\left| K(\theta|\theta_0) - \tfrac12 I(\theta_0)[\theta-\theta_0]\right|
	\le C_K \|\theta-\theta_0\|_{H_{m_K}}^{2+q_K},
	\qquad \text{for all}\ \theta\in \Theta\cap B^{H_{m_K}}_{R_K}(\theta_0).
\end{equation}
\end{assumption}

Next, we assume a (mild) global lower bound for the Kullback--Leibler divergence in terms of a negative-order Hilbert scale norm. Combined with properties of the prior \eqref{Eq:Prior}, this underpins a basic localisation property for the posterior around $\theta_0$ as $n\to\infty$, allowing the asymptotic analysis to be restricted to neighbourhoods where \eqref{Eq:LinkLower} and \eqref{Eq:KLTaylor} apply.

\begin{assumption}[{Global lower bound for the Kullback--Leibler divergence}]
\label{Ass:KLLB}
There exist a nondecreasing function $\varphi:[0,\infty)\to[0,\infty)$ with $\varphi(u)>0$ for $u>0$ and $\varphi(u)\gtrsim u^2$ as $u\to0^+$, and constants $c_\varphi>0$ and $r_K\ge\gamma$, such that
\begin{equation}
\label{Eq:KLLower}
	K(\theta|\theta_0)\;\ge\; \varphi\big(c_\varphi\|\theta-\theta_0\|_{H_{-r_K}}\big),
	\qquad \text{for all}\ \theta\in\Theta.
\end{equation}
\end{assumption}

%


%

Finally, we impose a sub-exponential moment condition on the centred sufficient statistic in a weak Hilbert scale geometry, which yields concentration of \(\hat T_n\) around \(T_{\theta_0}\).

\begin{assumption}[{Exponential integrability of the sufficient statistic}]
\label{Ass:Concentration}
For the index $m_K$ from Assumption \ref{Ass:KLTaylor}, there exists a constant $A_T>0$ such that
\begin{equation}
\label{Eq:Concentration}
	\operatorname{E}_{\theta_0}^{(1)}\exp\left\{\|\beta_{X_1} - T_{\theta_0}\|_{H_{-m_K}}/A_T\right\}<\infty.
\end{equation}
\end{assumption}

The restriction $m_K\le p$ from Assumption \ref{Ass:KLTaylor} guarantees, through the continuous embedding $H_{-m_K}\subseteq \Theta^*$, that $\|\hat T_n-T_{\theta_0}\|_{\Theta^*}$ is almost surely finite. On the other hand, \eqref{Eq:Concentration} implies that $\|\beta_{X_1}-T_{\theta_0}\|_{H_{-m_K}}$ has 
moments of all orders. In particular, by \eqref{Eq:FisherInfo} and Assumption \ref{Ass:Link},
$$
\operatorname{E}_{\theta_0}^{(1)}\|\beta_{X_1}-T_{\theta_0}\|_{H_{-m_K}}^2
=\sum_{k=1}^\infty \lambda_k^{-2m_K}I(\theta_0)[e_k]
\asymp \sum_{k=1}^\infty \lambda_k^{-2(m_K+\gamma)}.
$$
Thus, in view of \eqref{Eq:PolyGrowth}, Assumption \ref{Ass:Concentration} necessarily requires $m_K+\gamma>d/2$, as imposed below.

%
%
%

\subsubsection{Main Result}
\label{Subsec:Statement}

Under the above prior choice and model assumptions, the following posterior concentration theorem is obtained. Its proof is given in Appendix \ref{Sec:ProofMain}.

\begin{theorem}
\label{Theo:Main}
Let $\{f_\theta, \ \theta\in H_p\}$, for some $p\ge0$, be an infinite-dimensional exponential family as in \eqref{Eq:ExpFamily}. Let $\theta_0\in H_\beta$, $\beta>p$, and let $\Pi$ be an $\alpha$-regular Gaussian series prior as in \eqref{Eq:Prior}, for some $\alpha>p$. Let Assumptions \ref{Ass:Link}--\ref{Ass:Concentration} be satisfied for some $\gamma\ge0$, $q_K\in(0,1]$, $r_K\ge\gamma$ and some $m_K\in[0,p]$ such that $m_K+\gamma>d/2$. Further assume that
\begin{equation}
\label{Eq:Floor}
    q_K(\alpha\wedge\beta)
    >
    \max\big\{
        (2+q_K)p+2r_K,\;
	        (2+q_K)m_K+2\gamma+d
    \big\},
\end{equation}
and that, in the oversmoothing case $\alpha>\beta$, in addition
\begin{equation}
\label{Eq:FloorOversmooth}
	\alpha<\frac{(2+q_K)(\beta-m_K)-2\gamma-d}{2}.
\end{equation}
Then, for every $0\le s < \alpha\wedge\beta$ and every $M_n\to\infty$,
$$
	\operatorname{E}_{\theta_0}^{(n)}\left[\Pi\big(\theta\in\Theta:\|\theta-\theta_0\|_{H_s}>M_n 
		n^{-\frac{(\alpha\wedge\beta)-s}{2\alpha+2\gamma+d}}
	\big|\, X^{(n)}\big)\right]
	\to 0,
	\qquad \text{as}\ n\to\infty.
$$
\end{theorem}

Conditions \eqref{Eq:Floor} and \eqref{Eq:FloorOversmooth} arise from two localisation arguments in the proof. The bound in \eqref{Eq:Floor} involving $p$ and $r_K$ makes the likelihood contribution outside a fixed $H_p$-ball where the Fisher information is coercive negligible, by combining the global $H_{-r_K}$-lower bound \eqref{Eq:KLLower} with Hilbert-scale interpolation; see the proof of Proposition \ref{Prop:Det}. The other constraint, together with \eqref{Eq:FloorOversmooth} in the oversmoothing regime, permits localisation to a shrinking $H_{m_K}$-neighbourhood on which the Kullback--Leibler divergence can be replaced by its Fisher-information quadratic approximation \eqref{Eq:KLTaylor}; cf.~the proof of Proposition \ref{Prop:Det}. We expect these restrictions to be technical consequences of the employed proof techniques rather than intrinsic to the problem under consideration. In the density and intensity estimation models studied in Sections \ref{Subsec:DensityEstim} and \ref{Subsec:IntensityEstim}, one has $\gamma=0$, $q_K=1$ and may take $m_K$, $p$ and $r_K$ arbitrarily close to $d/2$, so that \eqref{Eq:Floor} reads $\alpha\wedge\beta>\tfrac52d$ and \eqref{Eq:FloorOversmooth} reads $\alpha<\tfrac32\beta-\tfrac54d$. For the $\gamma$-ill-posed linear inverse problem in white noise considered in Section \ref{Subsec:InvProb}, the Kullback--Leibler divergence is exactly quadratic and the link condition holds globally, so that one may take \(q_K=1\), \(r_K=\gamma\) and \(m_K=p\) for any \(p+\gamma>d/2\); exploiting this additional structure, conditions \eqref{Eq:Floor}--\eqref{Eq:FloorOversmooth} can be dispensed with, allowing any $\alpha,\beta>(d/2-\gamma)\vee 0$.

\begin{remark}[Comparison with \citet{dolera2024strong}]
Theorem \ref{Theo:Main} substantially improves the results in Section~3.1 of \citet{dolera2024strong}. Firstly, we note that their explicit bounds for infinite-dimensional exponential families \citep[Propositions 2 and 5]{dolera2024strong} require the prior covariance operator and the Fisher information $I(\theta_0)$ to be simultaneously diagonalised. Since outside linear models the eigenfunctions of $I(\theta_0)$ can depend nontrivially on the ground truth $\theta_0$, this restricts their results to an oracle prior. Our two-sided link condition in Assumption \ref{Ass:Link} avoids this issue by comparing the Fisher information with fixed powers of the Hilbert-scale generator, thereby allowing Gaussian series priors in a standard, truth-independent basis.

Furthermore, we quantitatively improve the rates. For the smoothness-matching case $\alpha=\beta$, Theorem \ref{Theo:Main} implies posterior concentration around $\theta_0\in H_\beta$ in $H_s$-norm at rate $n^{-\frac{\beta-s}{2\beta+2\gamma+d}}$, the usual benchmark for mildly ill-posed problems of degree $\gamma$. For the applications to density and intensity estimation, where $\gamma=0$, this recovers the minimax rate $n^{-\frac{\beta-s}{2\beta+d}}$ 
\citep{stone1982optimal}. See Sections \ref{Subsec:DensityEstim} and \ref{Subsec:IntensityEstim}. For comparison, in the direct case $\gamma=0$, under \eqref{Eq:PolyGrowth}, formally applying the argument of \citet{dolera2024strong} with $p=s$ as the index for the parameter space and target geometry yields 
the slower $H_s$-rate $n^{-\frac{\beta-2s+d/2}{2\beta+d}}$ for every $s>d/2$ (which, as discussed after Assumption \ref{Ass:Concentration}, is the range in which the choice $p=s$ is admissible). As hinted at the end of Section \ref{Subsec:WassApproach}, this suboptimality arises because their argument uses $H_s$ both as the target geometry and to control the fluctuations of the sufficient statistic. By decoupling these two geometries and deriving in Appendix \ref{Subsec:MixedTransport} a sharper mixed-geometry local stability estimate for the posterior kernel in Wasserstein distance, Theorem \ref{Theo:Main} instead recovers in applications the optimal benchmark.
\end{remark}

%
%
%
%
%

\section{Applications}
\label{Sec:Applications}

We apply the general theory to three important models that fall within the framework of infinite-dimensional exponential families, achieving in each instance minimax posterior contraction rates in Sobolev norms and optimal derivative estimation.

%
%
%

\subsection{Density estimation}
\label{Subsec:DensityEstim}

We assume we observe a random sample $X_1,\dots,X_n\iid f$ for some unknown p.d.f.~$f:[0,1]^d\to(0,\infty)$, modelled through the logistic parametrisation 
\begin{equation}
\label{Eq:DensityLogistic}
    f_{\theta}(x)
    =
    \frac{e^{\theta(x)}}
    {\int_{[0,1]^d}e^{\theta(x)}\,dx},
    \qquad x\in[0,1]^d,
\end{equation}
for some $\theta:[0,1]^d\to\R$. As the latter is identifiable only up to an additive constant, we take $\Theta$ to be a mean-zero wavelet Sobolev space $\dot {H}^p([0,1]^d)$ for some $p>d/2$, constructed as in Section \ref{Subsec:HilbertScales}. By the Sobolev embedding \citep[e.g.,][Chapter 4.3]{gine2016mathematical}, $\Theta$ is continuously embedded into the space of continuous functions $C([0,1]^d)$. It follows that for each $x\in[0,1]^d$ the point evaluation functional $\delta_x$ is linear and continuous on $\Theta$, and we recover an infinite-dimensional exponential family with $\beta_x = \delta_x$ and $M(\theta)=\log\int_{[0,1]^d} e^{\theta(x)}\,dx$.

The verification of Assumptions \ref{Ass:Link}--\ref{Ass:Concentration} for model \eqref{Eq:DensityLogistic}, developed in Appendix \ref{Subsec:ProofDensity}, yields $\gamma=0$, $q_K=1$ and $m_K$, $p$ and $r_K$ arbitrarily close to $d/2$. We then obtain the following posterior concentration result. We focus for conciseness on the non-oversmoothing regime, which slightly simplifies presentation of the regularity conditions.

\begin{theorem}
\label{Theo:DensityPCR}
Let $\theta_0\in \dot H^\beta([0,1]^d)$ for some $\beta>5d/2$, and let $\Pi$ be an $\alpha$-regular Gaussian series prior as in \eqref{Eq:Prior}, for $\alpha\in(5d/2,\beta]$ and $\{e_k\}_{k=1}^\infty$ an orthonormal basis of $\dot{L}^2([0,1]^d) = \dot{H}^0([0,1]^d)$ obtained by removing the constant function from an $S$-regular ($S>\beta$) wavelet basis as in Example \ref{Ex:WaveletScale}. Let $X^{(n)}=(X_1,\dots,X_n)$ with $X_1,\dots,X_n\iid f_{\theta_0}$ for $f_{\theta_0}$ as in \eqref{Eq:DensityLogistic} with $\theta = \theta_0$. Then, for every $0\le s< \alpha$ and every $M_n\to\infty$,
$$
	\operatorname{E}_{\theta_0}^{(n)}\left[\Pi\big(\theta\in\Theta:\|\theta-\theta_0\|_{H^s}>M_n n^{-\frac{\alpha-s}{2\alpha+d}}\big|\, X^{(n)}\big)\right]
	\to 0,
	\qquad \text{as} \ n\to\infty.
$$
\end{theorem}

Theorem \ref{Theo:DensityPCR} yields, for smoothness-matching priors, optimal recovery of the natural parameter $\theta$ and all its partial derivatives of order smaller than $\alpha$, in view of the equivalence $\|\cdot\|_{H^s}\asymp \|\cdot\|_{W^{s,2}}$ holding for $s < S$. The same conclusion transfers to the p.d.f. $f_\theta$: for any $t>d/2$ the space $H^t([0,1]^d)$ is closed under multiplication, so that by taking $t\in (s\vee d/2,\alpha)$ the map $\theta\mapsto f_\theta$ in \eqref{Eq:DensityLogistic} is readily seen to satisfy $\|f_{\theta_1}-f_{\theta_2}\|_{H^s}
\lesssim \|\theta_1-\theta_2\|_{H^s}$ uniformly over balls in $\dot{H}^t([0,1]^d)$. Theorem \ref{Theo:DensityPCR}, applied with index $t$, then localises the posterior to a fixed $\dot{H}^t$-ball, while its application with index $s$ gives the $H^s$-contraction rate $n^{-\frac{\alpha-s}{2\alpha+d}}$. Consequently, the induced posterior on $f_\theta$ also contracts in $H^s$-norm around $f_{\theta_0}$ at the same rate.

\begin{remark}[Score estimation]
The case $s=1$ deserves separate mention, as it concerns the \emph{score function} $\nabla\log f_\theta$ of the
model. Since $\log f_\theta=\theta-M(\theta)$ with $M(\theta)$ independent of $x$, one has $\nabla\log f_\theta=\nabla\theta$, so that Theorem \ref{Theo:DensityPCR} gives, for $\beta>5d/2$ and $\alpha\in (5d/2,\beta]$, 
$$
	\operatorname{E}_{\theta_0}^{(n)}\left[\Pi\left(\theta:
	\|\nabla\log f_\theta-\nabla\log f_{\theta_0}\|_{L^2}
	>M_n n^{-\frac{\alpha-1}{2\alpha+d}}\big|\, X^{(n)}\right)\right]
	\to 0,
	\qquad \text{as}\ n\to\infty.
$$
With $\alpha=\beta$, this shows that the posterior distribution attains  the minimax rate $n^{-\frac{\beta-1}{2\beta+d}}$ for score estimation \citep{wibisono2024optimal}. Moreover, since $f_{\theta_0}$ is continuous
and strictly positive on $[0,1]^d$, and thus bounded above and below, the \emph{Fisher
divergence} satisfies
$$
	D_F(f_{\theta_0}\|f_\theta)
	:=
	\int_{[0,1]^d}
	\big\|\nabla\log f_{\theta_0}(x)-\nabla\log f_\theta(x)\big\|^2
	f_{\theta_0}(x)\,dx
	\asymp
	\|\nabla\theta-\nabla\theta_0\|_{L^2}^2.
$$
Therefore, Theorem \ref{Theo:DensityPCR} directly implies the posterior concentrates in Fisher divergence at the squared rate $n^{-2(\alpha-1)/(2\alpha+d)}$, equal to $n^{-2(\beta-1)/(2\beta+d)}$ when $\alpha=\beta$. Score estimation has gained particular prominence in recent years because score matching is a central ingredient of diffusion-based generative models \citep{song2019generative}. Recent work has investigated estimation rates under Fisher divergence, both for infinite-dimensional exponential families \citep{sriperumbudur2017density} and for nonparametric score estimation \citep{wibisono2024optimal,ghosh2025stein,bonis2026optimal}.
\end{remark}

%
%
%

\subsection{Poisson intensity estimation}
\label{Subsec:IntensityEstim}

We next consider observing $n$ independent copies $N_1,\dots,N_n$ of an inhomogeneous Poisson point process $N$ on $[0,1]^d$ with unknown (first-order) intensity function $\rho:[0,1]^d\to(0,\infty)$, modelled through the exponential link $\rho = e^\theta$ for some $\theta:[0,1]^d\to\R$. That is, $N = \sum_{j=1}^J\delta_{X_j}$, arising as
$$
	J\sim \text{Po}\Bigg(\int_{[0,1]^d}e^{\theta(x)}dx\Bigg),
	\qquad 
	X_1,\dots,X_J|J\iid \frac{e^{\theta}dx}{\int_{[0,1]^d}e^{\theta(x)}dx}.
$$
For any measurable and bounded function $\theta\in L^\infty([0,1]^d)$, the law $\mu_\theta$ of $N$ is absolutely continuous with respect to the distribution $\mu_0$ of the unit-intensity Poisson process (i.e., with $\theta\equiv 0$), with Radon--Nikodym derivative \citep[][Chapter 1]{kutoyants1998statistical},
\begin{equation}
\label{Eq:PoissonLike}
	\frac{d\mu_\theta}{d\mu_0}(N) = \exp\left\{\sum_{j=1}^{J}\theta(X_j) 
	- \int_{[0,1]^d}\left[e^{\theta(x)}-1\right]\,dx\right\}.
\end{equation}
Here, $\theta$ is fully identifiable, and we take the parameter space $\Theta$ to be a wavelet Sobolev space $H^p([0,1]^d)$ with $p>d/2$, constructed as in Example \ref{Ex:WaveletScale}. Then, $H^p([0,1]^d)\subset C([0,1]^d)$ continuously and, for every $t>d/2$, finite signed measures on $[0,1]^d$ define elements of $H^{-t}([0,1]^d)=(H^t([0,1]^d))^*$. Arguing similarly to Section \ref{Subsec:DensityEstim}, we recover an infinite-dimensional exponential family with $\beta_N := \sum_{j=1}^J\delta_{X_j} = N$ and $M(\theta)=\int_{[0,1]^d} [e^{\theta(x)}-1]\,dx$.

Assumptions \ref{Ass:Link}--\ref{Ass:Concentration} are verified for model \eqref{Eq:PoissonLike} in Appendix \ref{Subsec:ProofIntensity}, again with $\gamma=0$, $q_K=1$ and  $m_K$, $p$ and $r_K$ arbitrarily close to $d/2$. Theorem \ref{Theo:Main} then yields the following posterior concentration result for non-oversmoothing Gaussian series priors.

\begin{theorem}
\label{Theo:IntensityPCR}
Let $\theta_0\in  H^\beta([0,1]^d)$ for some $\beta>5d/2$, and let $\Pi$ be an $\alpha$-regular Gaussian series prior as in \eqref{Eq:Prior}, for $\alpha\in(5d/2,\beta]$ and $\{e_k\}_{k=1}^\infty$ an $S$-regular ($S>\beta$) orthonormal wavelet basis of $L^2([0,1]^d)$ constructed as in Example \ref{Ex:WaveletScale}. Let $N^{(n)}=(N_1,\dots,N_n)$, where $N_1,\dots,N_n$ are i.i.d.~copies of an inhomogeneous Poisson process on $[0,1]^d$ with intensity $\rho_{\theta_0} = e^{\theta_0}$. Then, for every $0\le s< \alpha$ and every $M_n\to\infty$,
$$
	\operatorname{E}_{\theta_0}^{(n)}\left[\Pi\left(\theta\in\Theta:\|\theta-\theta_0\|_{H^s}>M_n n^{-\frac{\alpha-s}{2\alpha+d}}\big|\, N^{(n)}\right)\right]
	\to 0,
	\qquad \text{as} \ n\to\infty.
$$
\end{theorem}

The remarks after Theorem \ref{Theo:DensityPCR} apply here as well. In particular, for any $0\le s< \alpha$, the push-forward posterior on the intensity $\rho_\theta=e^\theta$ contracts around the ground truth $\rho_{\theta_0}$ in $H^s$-norm at rate $n^{-(\alpha-s)/(2\alpha+d)}$. This yields joint recovery of all the partial derivatives of order smaller than $\alpha$. For $\alpha=\beta$, the minimax rate is attained. To the best of our knowledge, this provides the first optimal posterior contraction rates in Sobolev norms for nonparametric Poisson intensity estimation, and hence for derivative estimation in this context.

%
%
%

\subsection{Mildly ill-posed inverse problems in white noise}
\label{Subsec:InvProb}

We lastly consider linear inverse problems observed in Gaussian white noise. Let $\mathcal Y$ be a separable Hilbert space and let $G:H_0\to\mathcal Y$ be a known injective bounded linear operator. We assume that $G$ is smoothing of order $\gamma\ge0$, namely that there exist constants $0<c_G\le C_G<\infty$ such that
\begin{equation}
\label{Eq:InvProbLink}
	c_G\|h\|_{H_{-\gamma}}^2
	\le
	\|Gh\|_{\mathcal Y}^2
	\le
	C_G\|h\|_{H_{-\gamma}}^2,
	\qquad h\in H_0.
\end{equation}
This is the standard link condition for mildly ill-posed inverse problems in Hilbert scales \citep{natterer1984error,engl1996regularization,mathe2003geometry}. Unlike classical approaches based on the singular value decomposition (SVD) of $G$ \citep[see, e.g.,][]{cavalier2008nonparametric}, we do not require $G^*G$ to be diagonal in the basis used to generate the Hilbert scale and construct the prior. This allows the use of standard, operator-independent bases even when the SVD of $G$ is unavailable or difficult to characterise. A closely related perspective is adopted by \citet{gugushvili2020bayesian}, who developed a testing-based approach for $H_0$-posterior contraction rates in this setting. See also \citet{giordano2020bernstein} for results on Bayesian uncertainty quantification with Gaussian priors.

For some unknown parameter $\theta\in H_0$, we observe a realisation of the Gaussian process
\begin{equation}
\label{Eq:InvProbModel}
	Y^{(n)}(g)
	=
	\langle G\theta,g\rangle_{\mathcal Y}
	+
	\frac1{\sqrt n} W(g),
	\qquad g\in\mathcal Y,
\end{equation}
where $ W$ is a Gaussian white noise indexed by $\mathcal Y$. See \cite[Chapters 1 and 2]{gine2016mathematical} for details. Denoting by $\mu_\theta^{(n)}$ the law of $Y^{(n)}$ in \eqref{Eq:InvProbModel} and by $\mu_0^{(n)}$ the law of the scaled noise $n^{-1/2}W$, the Cameron--Martin formula  \citep[see, e.g.,][Section 2.4]{bogachev1998gaussian} gives
\begin{equation}
\label{Eq:InvProbLikelihood}
	\frac{d\mu_\theta^{(n)}}{d\mu_0^{(n)}}\big(Y^{(n)}\big)
	=
	\exp\left\{
		n\left[
			Y^{(n)}(G\theta)
			-
			\frac12\|G\theta\|_{\mathcal Y}^2
		\right]
	\right\}.
\end{equation}
Fix $p\ge0$ such that $p+\gamma>d/2$ and take $\Theta=H_p$. The likelihood \eqref{Eq:InvProbLikelihood} is then in the infinite-dimensional exponential family from \eqref{Eq:Likelihood}, with $\hat T_n:=Y^{(n)}\circ G$ and $M(\theta)=\frac12\|G\theta\|_{\mathcal Y}^2$. The restriction $p+\gamma>d/2$ ensures that $\hat T_n$ is a $\Theta^*$-valued Gaussian random element: letting $T_\theta(h):
=\langle G\theta,Gh\rangle_{\mathcal Y}$, $h\in\Theta$, we have that $T_\theta\in\Theta^*$, that $\hat T_n-T_\theta=\frac1{\sqrt n}(W\circ G)$ and that
$$
	\operatorname{E}^{(n)}_\theta\|\hat T_n-T_\theta\|_{\Theta^*}^2
	=
	\frac1n\sum_{k=1}^\infty
	\lambda_k^{-2p}\operatorname{E}\!\left[W(Ge_k)^2\right]
	=
	\frac{1}{n}\sum_{k=1}^\infty
	\lambda_k^{-2p}\|Ge_k\|_{\mathcal Y}^2
	\asymp\frac{1}{n}
	\sum_{k=1}^\infty\lambda_k^{-2(p+\gamma)}
	<\infty,
$$
having used \eqref{Eq:InvProbLink} in the last step. For model \eqref{Eq:InvProbModel}, the (normalised) Fisher information equals $I(\theta)[h]=\|Gh\|_{\mathcal Y}^2$, $h\in\Theta$, and \eqref{Eq:InvProbLink} yields Assumption \ref{Ass:Link}, globally in $\theta$, with ill-posedness index $\gamma$. Moreover, the (normalised) Kullback--Leibler divergence is exactly quadratic,
$$
	K(\theta\mid\theta_0)
	=
	\frac12\|G(\theta-\theta_0)\|_{\mathcal Y}^2
	=
	\frac12 I(\theta_0)[\theta-\theta_0],
$$
making Assumption \ref{Ass:KLTaylor} automatically verified. Finally, as the centred sufficient statistic $\hat T_n-T_\theta$ is Gaussian and satisfies the above moment bound, we obtain from Fernique's theorem stronger concentration bounds than those derived from Assumption \ref{Ass:Concentration}. Exploiting these simplifications, the argument underlying the proof of Theorem \ref{Theo:Main} yields the following result without the regularity restrictions \eqref{Eq:Floor}--\eqref{Eq:FloorOversmooth}; see Appendix \ref{Subsec:ProofInvProb}.

\begin{theorem}
\label{Theo:InvProbPCR}
Let $G:H_0\to\mathcal Y$ satisfy \eqref{Eq:InvProbLink} for some $\gamma\ge0$. For $\alpha,\beta> (d/2-\gamma)\vee 0$, let $\theta_0\in H_\beta$ and let $\Pi$ be an $\alpha$-regular Gaussian series prior as in \eqref{Eq:Prior}. Let $Y^{(n)}$ be generated according to \eqref{Eq:InvProbModel} with $\theta = \theta_0$. Then, for every $0\le s<\alpha\wedge\beta$ and every $M_n\to\infty$,
$$
	\operatorname{E}_{\theta_0}^{(n)}
	\left[
		\Pi\left(
			\theta\in\Theta:
			\|\theta-\theta_0\|_{H_s}
			>
			M_n n^{-\frac{(\alpha\wedge\beta)-s}{2\alpha+2\gamma+d}}
			\,\middle|\,
			Y^{(n)}
		\right)
	\right]
	\to0,
	\qquad \text{as }n\to\infty.
$$
\end{theorem}

For smoothness-matching priors, Theorem \ref{Theo:InvProbPCR} yields the rate $n^{-(\beta-s)/(2\beta+2\gamma+d)}$, which is the minimax benchmark for recovery in $H_s$-loss from a $\gamma$-mildly ill-posed observation \citep{cavalier2008nonparametric}. In the direct case $G=\mathrm{Id}$, \eqref{Eq:InvProbModel} recovers the standard white noise model \cite[Chapter 1]{gine2016mathematical}. Then, $\gamma=0$ and the optimal rate $n^{-(\beta-s)/(2\beta+d)}$ is obtained.

\begin{example}[{Source identification in elliptic equations}]
\label{Ex:EllipticSource}

Let $\mathcal O\subset\R^d$ be a bounded domain with smooth boundary, and consider the Dirichlet problem
$$
	\mathcal A u
	:=
	-\sum_{j,k=1}^d
	\frac{\partial}{\partial x_k}
	\left(
		a_{jk}\frac{\partial u}{\partial x_j}
	\right)
	=
	\theta
	\quad\text{on }\mathcal O,
	\qquad
	u=0
	\quad\text{on }\partial\mathcal O,
$$
where the known coefficients $a_{jk}=a_{kj}\in C^\infty(\overline{\mathcal O})$ are uniformly elliptic, and where $\theta\in L^2(\mathcal O)$ is an unknown source function. By standard elliptic theory \citep[see, e.g.,][Chapters 6.2 and 6.3]{evans1998partial}, there exists a unique weak solution
$$
	u:=\mathcal A^{-1}\theta\in W^{2,2}(\mathcal O)\cap W_0^{1,2}(\mathcal O)
	=
	\big\{v\in W^{2,2}(\mathcal O):v=0
	\text{ on }\partial\mathcal O\big\},
$$
where the boundary condition is understood in the trace sense and, when $d\le3$, also pointwise via the Sobolev embedding. Furthermore, we have the elliptic estimate $\|\mathcal A^{-1}\theta\|_{W^{2,2}}\lesssim \|\theta\|_{L^2}$ with a multiplicative constant only depending on $\mathcal O$, $d$ and the $a_{jk}$'s. We assume that we observe $\Acal^{-1}\theta$ corrupted by additive Gaussian white noise indexed by $L^2(\mathcal O)$, and seek to recover $\theta$. Source recovery problems have been considered in several works in statistics and applied mathematics \citep[see, e.g.,][]{adavani2010fast,katzourakis2019regularization,giordano2020bernstein}. See also \citet{gugushvili2020bayesian}.

To cast this problem in the framework of Hilbert scales and apply our results, we take $H_0 = L^2(\mathcal O)$, and consider the eigenbasis $\{e_k\}_{k=1}^\infty$ of the Dirichlet Laplacian, 
$$
	-\Delta_D e_k=\varrho_ke_k \quad\text{on }\mathcal O,
	\qquad
	e_k=0
	\quad\text{on }\partial\mathcal O,
$$
with associated eigenvalues $\{\varrho_k\}_{k=1}^\infty$ satisfying $\varrho_k\asymp k^{2/d}$ by Weyl's law \citep[e.g.,][Chapter 8]{taylor2011partial2}. Taking $(-\Delta_D)^{1/2}$ as the generator, the resulting scale $H_s = H^s_D(\mathcal O)$ then satisfies $H^1_D(\mathcal O) = W^{1,2}_0(\mathcal O)$ and $H^2_D(\mathcal O) = W^{2,2}(\mathcal O)\cap W_0^{1,2}(\mathcal O)$, with norm equivalence. By symmetry and duality, the aforementioned elliptic regularity estimate, together with the continuity of the differential operator $\mathcal A :W^{2,2}(\mathcal O)\to L^2(\mathcal O)$, implies that $\|\mathcal A^{-1}\theta\|_{L^2}\asymp\|\theta\|_{H_D^{-2}}$ for all $\theta \in L^2(\mathcal O)$. This establishes the link condition \eqref{Eq:InvProbLink} with $\mathcal Y = H_0 = L^2(\mathcal O)$ and $\gamma = 2$. We note that, in general, $\Acal$ is not diagonal relative to the Dirichlet Laplacian eigenbasis.

For $\alpha,\beta>(d/2-2)\vee0$ and ground truth $\theta_0\in H_D^\beta(\mathcal O)$, Theorem \ref{Theo:InvProbPCR} implies that an $\alpha$-regular Gaussian series prior constructed in the above basis contracts in $H_D^s$-norm, for every $0\le s<\alpha\wedge\beta$, at rate $n^{-((\alpha\wedge\beta)-s)/(2\alpha+4+d)}$, which is optimal when $\alpha =\beta$. 
\end{example}

\section{Discussion}
\label{Sec:Discussion}

We have developed a general approach to posterior concentration in Hilbert scale norms under Gaussian priors for infinite-dimensional exponential families. Our main contribution is a mixed-geometry analysis of the posterior kernels, which separates the strong topology in which contraction is measured from the weaker one controlling the sufficient statistic fluctuations. Combined with a two-sided link condition on the Fisher information, this yields minimax contraction rates for smoothness-matching priors. The applications to density estimation, Poisson intensity estimation and mildly ill-posed inverse problems, where optimal recovery in Sobolev norms is obtained, illustrate the scope of the theory.

Several questions remain open. First, the Gaussian structure of the prior is used substantially in the Laplace-type analysis of the deterministic term (cf.~Appendix \ref{Subsec:DetTerm}) and in the functional inequalities underlying the stochastic stability estimate (cf.~Appendix \ref{Subsec:AuxStochTerm}). Nevertheless, the Wasserstein dynamics approach developed here offers a promising route for priors and models for which standard testing-based arguments or explicit calculations are unfeasible. A particularly interesting example is provided by heavy-tailed series priors, which \citet{agapiou2024heavy} recently showed to yield adaptation without hyperparameter tuning for ordinary posteriors in the white noise model, as well as for \textit{tempered posteriors} in more general ones, including density estimation and classification. Establishing corresponding adaptive rates for ordinary posteriors in infinite-dimensional exponential families would therefore constitute an important extension, requiring suitable replacements for the Gaussian quadratic-form estimates employed here, together with posterior moment and local functional-inequality bounds adapted to the heavy-tailed setting.

Concerning adaptation, we note that the prior regularity $\alpha$ is fixed in our results, and optimal rates are obtained when it is matched to the smoothness $\beta$ of the ground truth. A natural avenue to achieve adaptation within the present framework is via data-dependent selection of $\alpha$ or of a prior truncation level, either by a hierarchical prior model or through an empirical Bayes procedure. Extending our investigation to either setting would require deterministic and stochastic bounds that hold uniformly over the regularity or truncation level, together with suitable concentration results for the induced hyperposterior or empirical Bayes estimator. It would be particularly interesting to determine whether a single adaptive method can attain optimal rates simultaneously over a range of Sobolev losses, paralleling the recent results of \citet{liu2026optimal} for Gaussian process regression.

Finally, the Wasserstein approach may be applicable to a broader range of statistical models, beyond the exponential-family setting, provided that sufficiently uniform local quadratic expansions of the likelihood are available. Of particular interest are nonlinear inverse problems \citep{nickl2023bayesian}, where the derivative of the forward map could provide a local Hilbert-scale link condition, although controlling the nonlinear remainder and establishing posterior curvature would require additional ideas. The stability techniques developed here may also prove useful for studying Gaussian approximations to the posterior and the frequentist validity of uncertainty quantification for derivatives and other strongly identified functionals.

\appendix

\section{Proof of Theorem \ref{Theo:Main}}
\label{Sec:ProofMain}

\begin{proof}[Proof]
We follow the argument laid out in Section \ref{Subsec:WassApproach}. Write $ \nu_\tau:=\Pi_n^*(\cdot|\tau)$, $\tau\in\Theta^*$. We note that the kernel $\nu_\tau$ is well defined by \eqref{Eq:NormConst}, and that, under the assumptions of the theorem, the mean element $T_{\theta_0}$ exists and the expected Wasserstein distance $\varepsilon_n$ in \eqref{Eq:WassPCR} is finite. The decomposition  \eqref{Eq:WassDecomp} into the deterministic and stochastic terms is therefore valid. The first summand is shown in Proposition \ref{Prop:Det} to be of order $n^{-\frac{(\alpha\wedge\beta)-s}{2\alpha+2\gamma+d}}$ as $n\to\infty$, while the second  is shown in Proposition \ref{Prop:Stoch} below to be of the same or smaller order. Therefore, $\varepsilon_n=O( n^{-\frac{(\alpha\wedge\beta)-s}{2\alpha+2\gamma+d}})$ as $n\to\infty$.  The claim then follows since, as recalled after \eqref{Eq:WassPCR}, the convergence \eqref{Eq:PCRDef} holds with this choice of $\varepsilon_n$ for every $M_n\to\infty$. 
\end{proof}

In the rest of this appendix, we present the bound for the stochastic term in \eqref{Eq:WassDecomp}, leaving the treatment of the deterministic one to Appendix \ref{Subsec:DetTerm}.

%
%
%

\subsection{A mixed-geometry transport estimate}
\label{Subsec:MixedTransport}

Write $\hat D_n:=\hat T_n-T_{\theta_0}$ for the centred sufficient statistic. By \eqref{Eq:ExpFamily} and \eqref{Eq:Bregman}, the posterior distribution $\nu_{\hat T_n}$ has density proportional to $\exp\{n\hat D_n(\theta)\}$, $\theta\in\Theta$, relative to the deterministic probability measure $\nu_{T_{\theta_0}}$, a structural feature of exponential families that we use repeatedly.

The following auxiliary definitions are necessary. For $R>0$, write for shorthand $\Bcal:=B^\Theta_R(\theta_0)=\big\{\theta\in\Theta:\|\theta-\theta_0\|_\Theta\le R\big\}$. Using the definition of the Hilbert-scale norms in Section \ref{Subsec:HilbertScales}, and since $m_K\le p$, we have  $\|h\|_{H_{m_K}}\le \lambda_1^{m_K-p}\|h\|_{H_p}=\lambda_1^{m_K-p}\|h\|_\Theta$. Hence, we may fix $R<R_0$ sufficiently small that $\Bcal\subseteq B^\Theta_{R_0}(\theta_0)\cap B^{H_{m_K}}_{R_K}(\theta_0)$, so that both the coercivity property \eqref{Eq:LinkLower} and the Taylor expansion \eqref{Eq:KLTaylor} are available on $\Bcal$. For any $\tau\in\Theta^*$, let
\begin{equation}
\label{Eq:CondKernel}
    \tilde\nu_\tau(\cdot):=\nu_\tau(\cdot\,|\,\Bcal) = \frac{\nu_\tau(\cdot\,\cap\,\Bcal)}{\nu_\tau(\Bcal)}
\end{equation}
be the conditioning of $\nu_\tau$ to $\Bcal$ (well defined since $\nu_\tau$ and $\Pi$ are equivalent and $\Pi(\Bcal)>0$). The fundamental tool for bounding the stochastic term in \eqref{Eq:WassDecomp} is a mixed-geometry estimate for the Wasserstein distance between the conditioned probability measures $\tilde \nu_{\hat T_n}(\cdot) := \nu_{\hat T_n}(\cdot\,|\,\Bcal)$ and $\tilde \nu_{T_{\theta_0}}(\cdot) := \nu_{T_{\theta_0}}(\cdot\,|\,\Bcal)$. Its proof hinges on an approximation argument that involves the finite-dimensional Galerkin subspaces $\Theta_K:=\operatorname{span}\{e_1,\dots,e_K\}$, $K\in\N$, and the associated affine spaces $\Theta_{0,K}:=\theta_0+\Theta_K$. We denote by $P_K:\Theta\to\Theta_K$ the canonical projection operator onto $\Theta_K$ and let $P_{0,K}\theta:=P_K\theta+(I-P_K)\theta_0$, which satisfies
\begin{equation}
\label{Eq:GalerkinProjection}
    P_{0,K}\theta-\theta_0=P_K(\theta-\theta_0),
    \qquad \theta\in\Theta.
\end{equation}
Throughout, we identify $\Theta_{0,K}$ with $\R^K$ through the unique representation $\theta=(I-P_K)\theta_0+\sum_{k=1}^K\theta_ke_k$, $\theta\in\Theta_{0,K}$, holding by construction, and equip it with the metric induced by $H_s$. Its tangent space is $\Theta_K$, for which $\{\lambda_k^{-s}e_k\}_{k=1}^K$ is an orthonormal basis.

Next, define the push-forward prior distribution $\Pi_{0,K}:=(P_{0,K})_\#\Pi$ on $\Theta_{0,K}$. For $\tau\in\Theta^*$, denote by $\nu_{\tau,K}$ the $\tau$-induced Galerkin posterior on $\Theta_{0,K}$, given by
\begin{equation}
\label{Eq:CondGalPost}
	d\nu_{\tau,K}(\theta) = \frac{\exp\{n[\tau(\theta)-M(\theta)]\}}
	{\int_{\Theta_{0,K}}\exp\{n[\tau(\theta)-M(\theta)]\}d\Pi_{0,K}(\theta)}d\Pi_{0,K}(\theta),
	\qquad \theta\in\Theta_{0,K}.
\end{equation}
Further set $\Bcal_K:=\Bcal\cap\Theta_{0,K}$ and consider the conditioned Galerkin posterior $\tilde\nu_{\tau,K}(\cdot):=\nu_{\tau,K}(\cdot\,|\,\Bcal_K)$ on $\Theta_{0,K}$. By \eqref{Eq:GalerkinProjection}, $\Bcal_K=\{\theta\in\Theta_{0,K}:\|P_K(\theta-\theta_0)\|_{H_p}\le R\}$ is a closed, bounded and convex subset of $\Theta_{0,K}$ containing $\theta_0$; hence, $\Pi_{0,K}(\Bcal_K)>0$ and $\tilde\nu_{\tau,K}$ is well defined.

Lastly, we introduce two diagonal operators that control the reciprocal of the curvature of the negative log-posterior through a directional lower bound. For $ q\in[0,\alpha+d/2]$ and $n\in\N$, let $\Acal_{n,q}:H_q\to H_q$ be the linear diagonal operator
\begin{equation}
\label{Eq:Aq}
    \Acal_{n,q}e_k
    :=
    \frac{\lambda_k^{2(q+\gamma)}}{nc_0+\lambda_k^{2\alpha+2\gamma+d}}e_k,
    \qquad k\in\N,
\end{equation}
with $c_0>0$ the local coercivity constant of Assumption \ref{Ass:Link}, and let
\begin{equation}
\label{Eq:CondPoincareRate}
    \kappa_{n,q}
    :=
    \|\Acal_{n,q}\|_{\Lcal(H_q)}
    =\sup_{\|h\|_{H_q}=1}\|\Acal_{n,q}h\|_{H_q}
    =
    \sup_{k\in\N}\frac{\lambda_k^{2(q+\gamma)}}{nc_0+\lambda_k^{2\alpha+2\gamma+d}}
    \ \lesssim\
    n^{-1+\frac{2(q+\gamma)}{2\alpha+2\gamma+d}},
\end{equation}
the upper bound following by maximising over the spectral variable. The restriction $q\le\alpha+d/2$ ensures that the supremum is finite. Define also the linear operator $\Ccal_n:\Theta^*\to\Theta$ by
\begin{equation}
\label{Eq:Cn}
    \Ccal_n\chi
    :=
    \sum_{k=1}^\infty
    \frac{\lambda_k^{2\gamma}\chi(e_k)}{nc_0+\lambda_k^{2\alpha+2\gamma+d}}e_k,
    \qquad \chi\in\Theta^*,
\end{equation}
which is positive in the sense that $\chi(\Ccal_n\chi)\ge0$ for any $\chi\in\Theta^*$,  and for which, since $\alpha>p$,
$$
    \|\Ccal_n\chi\|_\Theta^2
    =
    \sum_{k=1}^\infty
    \frac{\lambda_k^{2p+4\gamma}\chi(e_k)^2}
    {(nc_0+\lambda_k^{2\alpha+2\gamma+d})^2}
    \lesssim
    \sum_{k=1}^\infty\lambda_k^{-2p}\chi(e_k)^2
    =
    \|\chi\|_{\Theta^*}^2.
$$

\begin{proposition}
\label{Prop:CondTransport}
Under the assumptions of Theorem \ref{Theo:Main}, for every $n\in\N$ and $0\le s<\alpha\wedge\beta$,
\begin{equation}
\label{Eq:CondTransport}
    \Wcal_2^{\Pcal_2(H_s)}
    \big(\tilde\nu_{\hat T_n},\tilde\nu_{T_{\theta_0}}\big)
    \ \le\
    n\sqrt{
        \kappa_{n,s}\;
        \hat D_n\big(\Ccal_n\hat D_n\big)
    },
\end{equation}
almost surely, where $\Ccal_n:\Theta^*\to\Theta$ and $\kappa_{n,s}>0$ are defined as in \eqref{Eq:Cn} and  \eqref{Eq:CondPoincareRate}, respectively.
\end{proposition}

\begin{proof}
We first establish the finite-dimensional analogue of \eqref{Eq:CondTransport} for the conditioned Galerkin posteriors $\tilde\nu_{\hat T_n,K}$ and $\tilde\nu_{T_{\theta_0},K}$, constructed as described after \eqref{Eq:CondGalPost} with $\tau = \hat T_n$ and $\tau = T_{\theta_0}$, respectively. This is based on the dynamic formulation of the Wasserstein distance \citep{benamou2000computational}. For any $t\in[0,1]$, define the $\Theta^*$-valued random element $\hat\tau_t:=T_{\theta_0}+t \hat D_n$, and let $\tilde\nu_{\hat\tau_t,K}$ be the associated conditioned Galerkin posterior. Then $\{\tilde\nu_{\hat\tau_t,K}\}_{t\in[0,1]}$ defines a curve of Borel probability measures in $\Pcal_2(H_s)$ that are supported on the convex set $\Bcal_K$ and join $\tilde \nu_{\hat\tau_0,K} = \tilde\nu_{T_{\theta_0},K}$ to $\tilde \nu_{\hat\tau_1,K} = \tilde\nu_{\hat T_n,K}$. An admissible velocity field for this curve is any jointly measurable vector field $v_{t,K}:\Bcal_K\to\Theta_K$, $t\in[0,1]$, such that
$$
    \int_0^1
    \left(
        \int_{\Bcal_K}
        \|v_{t,K}(\theta)\|_{H_s}^2
        \,d\tilde\nu_{\hat\tau_t,K}(\theta)
    \right)^{1/2}
    dt
    <\infty,
$$
and satisfying the continuity equation \citep[cf.~][Theorem~8.3.1]{ambrosio2008gradient}
\begin{equation}
\label{Eq:ContEq}
    \frac{d}{dt}
    \int_{\Bcal_K}
    \phi(\theta)\,d\tilde\nu_{\hat\tau_t,K}(\theta)
    =
    \int_{\Bcal_K}
    \left\langle
        \nabla^{H_s}\phi(\theta),
        v_{t,K}(\theta)
    \right\rangle_{H_s}
    \,d\tilde\nu_{\hat\tau_t,K}(\theta)
\end{equation}
for almost every $t\in(0,1)$ and every bounded smooth test function $\phi:\Theta_{0,K}\to\R$ of the form $\phi(\theta)=f(\theta_1,\ldots,\theta_K)$, where $\theta=(I-P_K)\theta_0+\sum_{k=1}^K\theta_ke_k\in\Theta_{0,K}$, for some bounded and smooth $f:\mathbb R^K\to\R$ with bounded derivatives of all orders. Above, the $H_s$-gradient of $\phi$ is defined by $\nabla^{H_s}\phi(\theta):=\sum_{k=1}^K\lambda_k^{-2s}\partial_k f(\theta_1,\ldots,\theta_K)e_k$, 
where $\partial_k f$ denotes the $k$-th partial derivative of $f$. The dynamic formulation of the Wasserstein distance then yields the upper bound \citep[cf.~][Theorem~8.3.1]{ambrosio2008gradient}
\begin{equation}
\label{Eq:DynamicWassBound}
    \Wcal_2^{\Pcal_2(H_s)}
    \bigl(
        \tilde\nu_{\hat T_n,K},
        \tilde\nu_{T_{\theta_0},K}
    \bigr)
    \le
    \int_0^1
    \left(
        \int_{\Bcal_K}
        \|v_{t,K}(\theta)\|_{H_s}^2
        \,d\tilde\nu_{\hat\tau_t,K}(\theta)
    \right)^{1/2}
    dt.
\end{equation}

We move to the construction of $v_{t,K}$ satisfying \eqref{Eq:ContEq} and obtain a uniform bound for the integrated norm appearing in \eqref{Eq:DynamicWassBound}. Using that $\frac{d}{dt}\hat\tau_t(\theta)
=\hat D_n(\theta)$ for all $t\in[0,1]$ and $\theta\in\Theta$, differentiating under the integral sign the left-hand side of \eqref{Eq:ContEq} gives
\begin{align*}
    \frac{d}{dt}
    \int_{\Bcal_K}
    \phi(\theta)\,d\tilde\nu_{\hat\tau_t,K}(\theta)
    &=
	n\int_{\Bcal_K}
	\phi(\theta)\hat D_n(\theta)
	\,d\tilde\nu_{\hat\tau_t,K}(\theta)
	\\
	&\quad
	-
	n\left(
	\int_{\Bcal_K}
	\phi(\theta)\,d\tilde\nu_{\hat\tau_t,K}(\theta)
	\right)
	\left(
	\int_{\Bcal_K}
	\hat D_n(\theta)\,d\tilde\nu_{\hat\tau_t,K}(\theta)
	\right)
	\\
	&=
    	n\,\operatorname{Cov}_{\tilde\nu_{\hat\tau_t,K}}
    	\bigl(\phi(\theta),\hat D_n(\theta)\bigr).
\end{align*}
Consequently, by the Cauchy--Schwarz inequality
\begin{equation}
\label{Eq:CovarianceBound}
    \left|
        \frac{d}{dt}
        \int_{\Bcal_K}
        \phi(\theta)\,d\tilde\nu_{\hat\tau_t,K}(\theta)
    \right|
    \le
    n
    \sqrt{
        \operatorname{Var}_{\tilde\nu_{\hat\tau_t,K}}
        [\phi(\theta)]
    }
    \sqrt{
        \operatorname{Var}_{\tilde\nu_{\hat\tau_t,K}}
        [\hat D_n(\theta)]
    }.
\end{equation}
We next estimate the two variances using the preconditioned Poincaré inequality established in Proposition \ref{Prop:CondPoincare}. Applying \eqref{Eq:CondPoincareGalerkin} with $q=s$, we directly obtain
$$
	\operatorname{Var}_{\tilde\nu_{\hat\tau_t,K}}
        [\phi(\theta)]
        \le \kappa_{n,s}
        \int_{\Bcal_K}
        \|\nabla^{H_s}\phi(\theta)\|_{H_s}^2
        \,d\tilde\nu_{\hat\tau_t,K}(\theta).
$$
For the second term in \eqref{Eq:CovarianceBound}, we apply the first inequality in \eqref{Eq:CondPoincareGalerkin} with $q=s$ to the restriction of the linear functional $\hat D_n$ to $\Theta_{0,K}$. Its $H_s$-gradient is the constant vector field $\nabla^{H_s}\hat D_n(\theta)= \sum_{k=1}^K\lambda_k^{-2s}\hat D_n(e_k)e_k$.
Consequently, with $\Acal_{n,s,K}$ the restriction of $\Acal_{n,s}$ to $\Theta_K$,
\begin{align*}
    \left\langle
        \Acal_{n,s,K}\nabla^{H_s}\hat D_n(\theta),
        \nabla^{H_s}\hat D_n(\theta)
    \right\rangle_{H_s}
    &\quad=
    \sum_{k=1}^K
    \lambda_k^{2s}
    \frac{\lambda_k^{2\gamma}\hat D_n(e_k)}
    {nc_0+\lambda_k^{2\alpha+2\gamma+d}}
    \lambda_k^{-2s}\hat D_n(e_k)
    \\
    &\quad=
    \sum_{k=1}^K
    \frac{\lambda_k^{2\gamma}\hat D_n(e_k)^2}
    {nc_0+\lambda_k^{2\alpha+2\gamma+d}},
\end{align*}
which is independent of $s$ (and $\theta$). The first inequality in \eqref{Eq:CondPoincareGalerkin} then yields
$$
    \operatorname{Var}_{\tilde\nu_{\hat\tau_t,K}}
    \bigl[\hat D_n(\theta)\bigr]
    \le
    \sum_{k=1}^K
    \frac{\lambda_k^{2\gamma}\hat D_n(e_k)^2}
    {nc_0+\lambda_k^{2\alpha+2\gamma+d}}.
$$
We note that, in principle, \eqref{Eq:CondPoincareGalerkin} applies to bounded functions. It nevertheless readily extends to $ \hat D_n$. Indeed, since $\Bcal_K$ is bounded, $| \hat D_n(\theta)|\le\| \hat D_n\|_{\Theta^*}(\|\theta_0\|_\Theta+R)$ on $\theta\in\Bcal_K$, and it suffices to apply the first inequality in \eqref{Eq:CondPoincareGalerkin} to $H\circ \hat D_n$ for a smooth bounded function $H:\R\to\R$ with bounded derivatives of all orders such that $H(z)=z$ for all $|z|\le 1 + \| \hat D_n\|_{\Theta^*}(\|\theta_0\|_\Theta+R)$. Then $H\circ \hat D_n= \hat D_n$ and $\nabla^{H_s}(H\circ \hat D_n)=\nabla^{H_s} \hat D_n$ on $\Bcal_K$, which is the support of $\tilde\nu_{\hat\tau_t,K}$. Combining the obtained bounds with \eqref{Eq:CovarianceBound} yields 
$$
       \left|
        \frac{d}{dt}
        \int_{\Bcal_K}
        \phi(\theta)\,d\tilde\nu_{\hat\tau_t,K}(\theta)
    \right|
    \le
    n\left(
        \sum_{k=1}^K
        \frac{\lambda_k^{2\gamma}\hat D_n(e_k)^2}{nc_0+\lambda_k^{2\alpha+2\gamma+d}}
    \right)^{1/2}
    \left(
    \kappa_{n,s}
        \int_{\Bcal_K}
        \|\nabla^{H_s}\phi\|_{H_s}^2
        \,d\tilde\nu_{\hat\tau_t,K}
    \right)^{1/2}.
$$

We now use the preceding bound to construct the admissible velocity field $v_{t,K}$ for \eqref{Eq:DynamicWassBound}. Under the above identification of $\Theta_{0,K}$ with $\mathbb R^K$, the set $\Bcal_K=\left\{\theta\in\Theta_{0,K}:\|\theta-\theta_0\|_{H_p}\le R\right\}$ is a bounded ellipsoid in $\mathbb R^K$ with smooth boundary. Let $H^1(\Bcal_K) = W^{1,2}(\Bcal_K)$ be the standard Hilbert-Sobolev space of square integrable functions with square integrable first-order weak derivatives, equipped with the usual inner product and norm, and denote by
$$
    \dot H^1(\Bcal_K)
    :=
    \left\{
        \varphi\in H^1(\Bcal_K):
        \int_{\Bcal_K}\varphi(\theta)\,d\theta=0
    \right\}.
$$
For each fixed $K$ and $t\in[0,1]$, the density of $\tilde\nu_{\hat\tau_t,K}$ with respect to the Lebesgue measure on $\Bcal_K$ is proportional to $\exp\{-V_{\hat \tau_t,K}(\theta)\}$, $\theta\in\Bcal_K$, where 
$$
	V_{\hat\tau_t,K}(\theta)
	:=
	n\big\{M(\theta)-\hat\tau_t(\theta)\big\}
	+
	\frac12\sum_{k=1}^K
	\lambda_k^{2\alpha+d}\theta_k^2,
	\qquad
	\theta=(I-P_K)\theta_0+\sum_{k=1}^K\theta_ke_k\in\Theta_{0,K}.
$$
Since $[0,1]\times \Bcal_K$ is compact, such density is bounded and bounded away from zero by constants that do not depend on $t$. Furthermore, since $\lambda_K^{-2s}\|\nabla \varphi(\theta)\|_{\mathbb R^K}^2\le\|\nabla^{H_s}\varphi(\theta)\|_{H_s}^2\le\lambda_1^{-2s}\|\nabla \varphi(\theta)\|_{\mathbb R^K}^2$ for all differentiable $\varphi:\Bcal_K\to\R$ and all $\theta\in\Bcal_K$, the norms
$$
    \|\varphi\|_{t,K}
    :=
    \left(
        \int_{\Bcal_K}
        \|\nabla^{H_s}\varphi(\theta)\|_{H_s}^2
        \,d\tilde\nu_{\hat\tau_t,K}(\theta)
    \right)^{1/2},
    \qquad t\in[0,1],
$$
are uniformly equivalent to the usual norm on $\dot H^1(\Bcal_K)$. For $\phi\in C^\infty(\Bcal_K)\cap\dot H^1(\Bcal_K)$, define
$$
    L_{t,K}(\phi)
    :=
    \frac{d}{dt}
    \int_{\Bcal_K}
    \phi(\theta)\,d\tilde\nu_{\hat\tau_t,K}(\theta),
$$
which, by the preceding bound, extends to a bounded linear functional on $\dot H^1(\Bcal_K)$ satisfying
\begin{equation}
\label{Eq:BoundedL}
    |L_{t,K}(\phi)|
    \le
    n
    \left(
        \kappa_{n,s}
        \sum_{k=1}^K
        \frac{\lambda_k^{2\gamma}\hat D_n(e_k)^2}
        {nc_0+\lambda_k^{2\alpha+2\gamma+d}}
    \right)^{1/2}
    \|\phi\|_{t,K}.
\end{equation}
By the Riesz representation theorem,  there exists a unique $u_{t,K}\in\dot H^1(\Bcal_K)$ such that
\begin{equation}
\label{Eq:VarCont}
    L_{t,K}(\phi)
    =
    \int_{\Bcal_K}
    \left\langle
        \nabla^{H_s}\phi(\theta),
        \nabla^{H_s}u_{t,K}(\theta)
    \right\rangle_{H_s}
    \,d\tilde\nu_{\hat\tau_t,K}(\theta),
    \qquad \text{for all}\
    \phi\in\dot H^1(\Bcal_K).
\end{equation}
Since $L_{t,K}$ annihilates constants and gradients are unchanged by adding constants, this identity also holds for every bounded smooth test function on $\Theta_{0,K}$. Moreover, the continuous dependence on $t$ of the density of $\tilde\nu_{\hat\tau_t,K}$ and of the functional $L_{t,K}$ imply that the map $t\in[0,1]\mapsto u_{t,K}$ is continuous. Hence, $v_{t,K} := \nabla^{H_s} u_{t,K}$ may be chosen jointly measurable in $t$ and $\theta$. In view of \eqref{Eq:VarCont}, $v_{t,K}$ solves the continuity equation \eqref{Eq:ContEq}, and therefore is an admissible velocity field for $\{\tilde\nu_{\hat\tau_t,K}\}_{t\in[0,1]}$. Thus, taking $\phi=u_{t,K}$ in \eqref{Eq:BoundedL} and \eqref{Eq:VarCont} gives
$$
    \int_{\Bcal_K}
    \|\nabla^{H_s}u_{t,K}(\theta)\|_{H_s}^2
    \,d\tilde\nu_{\hat\tau_t,K}(\theta)
    \le
    n^2\kappa_{n,s}
    \sum_{k=1}^K
    \frac{\lambda_k^{2\gamma}\hat D_n(e_k)^2}
    {nc_0+\lambda_k^{2\alpha+2\gamma+d}},
$$
and, recalling the Wasserstein dynamic bound \eqref{Eq:DynamicWassBound},
\begin{align}
\label{Eq:CondTranspK}
    \Wcal_2^{\Pcal_2(H_s)}
    \bigl(
        \tilde\nu_{\hat T_n,K},
        \tilde\nu_{T_{\theta_0},K}
    \bigr)
        &\le
    n
    \left(
        \kappa_{n,s}
        \sum_{k=1}^K
        \frac{\lambda_k^{2\gamma}\hat D_n(e_k)^2}
        {nc_0+\lambda_k^{2\alpha+2\gamma+d}}
    \right)^{1/2}.
\end{align}

We conclude with a limiting argument. By Lemma \ref{Lem:GalerkinConvergence},
$$
\Wcal_2^{\Pcal_2(H_s)}
\bigl(\tilde\nu_{\hat T_n,K},\tilde\nu_{\hat T_n}\bigr)
\to 0,
\qquad
\Wcal_2^{\Pcal_2(H_s)}
\bigl(\tilde\nu_{T_{\theta_0},K},\tilde\nu_{T_{\theta_0}}\bigr)
\to 0,
$$
as $K\to\infty$. Combining this with \eqref{Eq:CondTranspK} and the triangle inequality gives
$$
	\begin{aligned}
	\Wcal_2^{\Pcal_2(H_s)}
	\bigl(
	\tilde\nu_{\hat T_n},
	\tilde\nu_{T_{\theta_0}}
	\bigr)
	&\le
	n
	\left(
	\kappa_{n,s}
	\sum_{k=1}^\infty
	\frac{
	\lambda_k^{2\gamma}\hat D_n(e_k)^2
	}{
	nc_0+\lambda_k^{2\alpha+2\gamma+d}
	}
	\right)^{1/2}
	=
	n\sqrt{
	\kappa_{n,s}\,
	\hat D_n\bigl(\Ccal_n\hat D_n\bigr)
	},
\end{aligned}
$$
which proves \eqref{Eq:CondTransport}, the last identity following from the definition of $\Ccal_n:\Theta^*\to\Theta$ in \eqref{Eq:Cn}.
\end{proof}

\begin{remark}[Global transport estimate]
\label{Rem:GlobTransp}
The preceding argument also yields a global transport estimate for the unconditioned posterior kernels $\nu_\tau$, $\tau\in\Theta^*$. For each $K\in\N$, the corresponding Galerkin kernel $\nu_{\tau,K}$ has density $d\nu_{\tau,K}(\theta)\propto e^{-V_{\tau,K}(\theta)}\,d\theta$, $\theta\in\Theta_{0,K}$, with potential function $V_{\tau,K}(\theta)=n\{M(\theta)-\tau(\theta)\} +\frac12\sum_{k=1}^K\lambda_k^{2\alpha+d}\theta_k^2$. For every $\theta\in\Theta_{0,K}$ and any tangent direction $h=\sum_{k=1}^K h_ke_k\in\Theta_K$, its Hessian satisfies
$$
	D^2V_{\tau,K}(\theta)[h,h]
	=
	nD^2M(\theta)[h,h]
	+
	\sum_{k=1}^K\lambda_k^{2\alpha+d}h_k^2
	\ge
	\sum_{k=1}^K\lambda_k^{2\alpha+d}h_k^2,
$$
since $M$ is convex. The finite-dimensional Brascamp--Lieb inequality \citep[Theorem 4.1]{brascamp1976extensions}  consequently gives, for every smooth function $\phi:\Theta_{0,K}\to\R$,
$$
	\begin{aligned}
	\operatorname{Var}_{\nu_{\tau,K}}[\phi(\theta)]
	&\le
	\int_{\Theta_{0,K}}
	\left\langle
	\big[D^2V_{\tau,K}(\theta)\big]^{-1}\nabla \phi(\theta),
	\nabla \phi(\theta)
	\right\rangle_{\mathbb R^K}
	\,d\nu_{\tau,K}(\theta)
	\\
	&\le
	\int_{\Theta_{0,K}}
	\sum_{k=1}^K
	\lambda_k^{-2\alpha-d}
	|\partial_k\phi(\theta)|^2
	\,d\nu_{\tau,K}(\theta)
	\le
	\bar\kappa_s
	\int_{\Theta_{0,K}}
	\|\nabla^{H_s}\phi(\theta)\|_{H_s}^2
	\,d\nu_{\tau,K}(\theta),
\end{aligned}
$$
where  $\bar\kappa_s:=\sup_{k\ge1}\lambda_k^{2s-2\alpha-d}<\infty$ since $s<\alpha$. Further, for $\tau,\tau'\in\Theta^*$, applying the same Brascamp--Lieb inequality to the linear functional $\chi:=\tau-\tau'$ gives $\operatorname{Var}_{\nu_{\tau,K}}[\chi(\theta)]\le\sum_{k=1}^K\lambda_k^{-2\alpha-d}\chi(e_k)^2$. Define
$\tau_t:=\tau'+t(\tau-\tau')$, $t\in[0,1]$. Repeating along the curve
$\{\nu_{\tau_t,K}\}_{t\in[0,1]}$ the transport argument from the proof of Proposition \ref{Prop:CondTransport} then yields
$$
\Wcal_2^{\Pcal_2(H_s)}
\bigl(\nu_{\tau,K},\nu_{\tau',K}\bigr)
\le
n\sqrt{\bar\kappa_s}
\left(
\sum_{k=1}^K
\lambda_k^{-2\alpha-d}
(\tau-\tau')(e_k)^2
\right)^{1/2}.
$$
Finally, by passing to the limit as $K\to\infty$ as in the conclusion of the proof of Proposition \ref{Prop:CondTransport}, the following crude global transport estimate is obtained:
$$
\Wcal_2^{\Pcal_2(H_s)}
\bigl(\nu_\tau,\nu_{\tau'}\bigr)
\le
n\sqrt{\bar\kappa_s}\,
\|\tau-\tau'\|_{H_{-\alpha-d/2}},
\qquad
\tau,\tau'\in\Theta^*.
$$
\end{remark}

%
%
%

\subsection{Bounds for the stochastic term}
\label{Subsec:StochTerm}

We conclude by employing the above transport estimates, together with a localising argument, to upper bound the stochastic term in \eqref{Eq:WassDecomp}.

\begin{proposition}
\label{Prop:Stoch}
Under the assumptions of Theorem \ref{Theo:Main}, for every $0\le s<\alpha\wedge\beta$,
\begin{equation}
\label{Eq:StochRate}
    \operatorname{E}_{\theta_0}^{(n)}
    \left[
        \Wcal_2^{\Pcal_2(H_s)}
        \big(
            \nu_{\hat T_n},
            \nu_{T_{\theta_0}}
        \big)
    \right]
    =
    O\left(
        n^{-\frac{(\alpha\wedge\beta)-s}{2\alpha+2\gamma+d}}
    \right),
    \qquad \text{as $n\to\infty$}.
\end{equation}
\end{proposition}

\begin{proof}
Let $\Ecal_n$ be the event defined in Lemma \ref{Lem:GoodEvent}, whose complement satisfies 
$\mu_{\theta_0}^{(n)}(\Ecal_n^c)\le\exp\{-Cn^\varsigma\}$ for some constants $C,\varsigma>0$  and all $n$ large enough. Split
$$
    \operatorname{E}^{(n)}_{\theta_0}\big[\Wcal_2^{\Pcal_2(H_s)}(\nu_{\hat T_n},\nu_{T_{\theta_0}})\big]
    =
    \operatorname{E}^{(n)}_{\theta_0}\big[\Wcal_2^{\Pcal_2(H_s)}(\nu_{\hat T_n},\nu_{T_{\theta_0}})1_{\Ecal^c_n}\big]
    +
    \operatorname{E}^{(n)}_{\theta_0}\big[\Wcal_2^{\Pcal_2(H_s)}(\nu_{\hat T_n},\nu_{T_{\theta_0}})1_{\Ecal_n}\big].
$$

We upper bound the first term. By the global transport estimate from Remark \ref{Rem:GlobTransp}, $\Wcal_2^{\Pcal_2(H_s)}\big(\nu_{\hat T_n},\nu_{T_{\theta_0}}\big)^2 \le n^2\bar\kappa_s \|\hat D_n\|^2_{H_{-\alpha-d/2}}$.
Furthermore, by \eqref{Eq:FisherInfo}, \eqref{Eq:LinkUpper} and Fubini's theorem,
$$
    \operatorname{E}^{(n)}_{\theta_0}\|\hat D_n\|^2_{H_{-\alpha-d/2}}
    =
    \frac1n\sum_{k=1}^\infty\lambda_k^{-2\alpha-d}I(\theta_0)[e_k]
    \le
    \frac{C_0}{n}\sum_{k=1}^\infty\lambda_k^{-2\alpha-2\gamma-d}
    \lesssim
    \frac1n,
$$
the series converging in view of the eigenvalue asymptotics \eqref{Eq:PolyGrowth} for any $\alpha>0$. Thus, $\operatorname{E}^{(n)}_{\theta_0}[\Wcal_2^{\Pcal_2(H_s)}(\nu_{\hat T_n},\nu_{T_{\theta_0}})^2]\lesssim n$, which combined with the Cauchy--Schwarz inequality gives
$$
    \operatorname{E}^{(n)}_{\theta_0}\big[\Wcal_2^{\Pcal_2(H_s)}(\nu_{\hat T_n},\nu_{T_{\theta_0}})1_{\Ecal^c_n}\big]
    \lesssim
    n^{1/2}\exp\big\{-\tfrac C2n^{\varsigma}\big\}= o\left(n^{-\frac{(\alpha\wedge\beta)-s}{2\alpha+2\gamma+d}}\right).
$$

Turning to the second expectation, we introduce the further split 
$$
    \Wcal_2^{\Pcal_2(H_s)}
    \big(\nu_{\hat T_n},\nu_{T_{\theta_0}}\big)
    \le
    \Wcal_2^{\Pcal_2(H_s)}\big(\nu_{\hat T_n},\tilde\nu_{\hat T_n}\big)
    +
    \Wcal_2^{\Pcal_2(H_s)}\big(\tilde\nu_{T_{\theta_0}},\nu_{T_{\theta_0}}\big)
    +
    \Wcal_2^{\Pcal_2(H_s)}\big(\tilde\nu_{\hat T_n},\tilde\nu_{T_{\theta_0}}\big),
$$
where $\tilde\nu_{\hat T_n},\tilde\nu_{T_{\theta_0}}$ are defined as in \eqref{Eq:CondKernel} with $\tau = \hat T_n$ and $\tau = T_{\theta_0}$, respectively. Lemma \ref{Lem:CondError} shows that the expectations of the first two terms,  restricted to $\Ecal_n$, are also of lower order than the rate in \eqref{Eq:StochRate}. Lastly, by Proposition \ref{Prop:CondTransport} and the Cauchy--Schwarz inequality,
$$
    \operatorname{E}^{(n)}_{\theta_0}
    \Big[\Wcal_2^{\Pcal_2(H_s)}\big(\tilde\nu_{\hat T_n},\tilde\nu_{T_{\theta_0}}\big)1_{\Ecal_n}\Big]
    \le
    n\sqrt{\kappa_{n,s}}\,
    \sqrt{\operatorname{E}^{(n)}_{\theta_0}\big[\hat D_n(\Ccal_n\hat D_n)\big]},
$$
where $\Ccal_n:\Theta^*\to \Theta$ is the linear operator from \eqref{Eq:Cn}, and where $\kappa_{n,s}$ from \eqref{Eq:CondPoincareRate} satisfies $\kappa_{n,s}\lesssim n^{-1+2(s+\gamma)/(2\alpha+2\gamma+d)}$. By \eqref{Eq:FisherInfo}, the definition of $\Ccal_n$, \eqref{Eq:LinkUpper} and Fubini's theorem,
$$
    \operatorname{E}^{(n)}_{\theta_0}\big[\hat D_n(\Ccal_n\hat D_n)\big]
    =
    \frac1n\sum_{k=1}^\infty\frac{\lambda_k^{2\gamma}I(\theta_0)[e_k]}{nc_0+\lambda_k^{2\alpha+2\gamma+d}}
    \le
    \frac{C_0}{n}\sum_{k=1}^\infty\frac{1}{nc_0+\lambda_k^{2\alpha+2\gamma+d}}
    \lesssim
    n^{-2+\frac{d}{2\alpha+2\gamma+d}},
$$
the last inequality following from \eqref{Eq:PolyGrowth} upon splitting the series at the spectral cut-off $\lambda_k^{2\alpha+2\gamma+d}\asymp n$, with both resulting pieces being of order $n^{-1+d/(2\alpha+2\gamma+d)}$. It follows that
$$
    \operatorname{E}^{(n)}_{\theta_0}
    \Big[\Wcal_2^{\Pcal_2(H_s)}\big(\tilde\nu_{\hat T_n},\tilde\nu_{T_{\theta_0}}\big)1_{\Ecal_n}\Big]
    \lesssim
    n\ n^{-\frac12+\frac{s+\gamma}{2\alpha+2\gamma+d}}\ n^{-1+\frac{d/2}{2\alpha+2\gamma+d}}
    =
    n^{-\frac{\alpha-s}{2\alpha+2\gamma+d}} .
$$
This is no larger than the rate in \eqref{Eq:StochRate}, and combining the obtained bounds proves the claim.
\end{proof}

\section{Auxiliary results for the proof of Theorem \ref{Theo:Main}}
\label{Sec:Proofs}

In this section, we collect the supporting localisation and transport results (Appendix \ref{Subsec:AuxStochTerm}) used in the analysis of the stochastic term in \eqref{Eq:WassDecomp}, whose upper bound is presented in Appendix \ref{Sec:ProofMain}. We then establish the accompanying upper bound for the deterministic term (Appendix \ref{Subsec:DetTerm}).

%
%
%

\subsection{Localisation and transport techniques for the stochastic term}
\label{Subsec:AuxStochTerm}

The first main auxiliary result for the upper bound on the stochastic term in \eqref{Eq:WassDecomp} is the construction of a suitable localising event $\Ecal_n$ whose complement has exponentially small probability, to be used in the proof of Proposition \ref{Prop:Stoch}; see Lemma \ref{Lem:GoodEvent}. On $\Ecal_n$, we show in Lemma \ref{Lem:CondError} below that the posterior kernels can be replaced, at an exponentially small cost in Wasserstein distance, by their conditionings to the fixed ball $\Bcal$ from \eqref{Eq:CondKernel}, on which the Fisher information is coercive. Finally, Proposition \ref{Prop:CondPoincare} provides the key Poincar\'e inequality for conditioned posterior kernels underpinning the mixed-geometry transport estimate of Proposition \ref{Prop:CondTransport}.

%
%
%

\subsubsection{Construction of the localising event}

We start by deriving, from the exponential moment condition of Assumption \ref{Ass:Concentration}, two deviation inequalities for the centred sufficient statistic $\hat D_n =\hat T_n-T_{\theta_0}$.

\begin{lemma}
\label{Lem:Bernstein}
Let Assumption \ref{Ass:Concentration} hold for some $m_K\in[0,p]$ and $A_T>0$. Then, there exists a constant $V_T>0$ such that, for all $u>0$ and $n\in\N$,
\begin{equation}
\label{Eq:BernsteinNorm}
    \mu^{(n)}_{\theta_0}
    \big(\|\hat D_n\|_{H_{-m_K}}\ge u\big)
    \le
    2\exp\left\{-\frac{nu^2}{2(V_T+A_Tu)}\right\},
\end{equation}
and such that, for every fixed $h\in H_{m_K}$, $h\neq 0$,
\begin{equation}
\label{Eq:BernsteinLinear}
    \mu^{(n)}_{\theta_0}
    \big(|\hat D_n(h)|\ge u\big)
    \le
    2\exp\left\{-\frac{nu^2}
    {2\big(V_T\|h\|^2_{H_{m_K}}+A_T\|h\|_{H_{m_K}}u\big)}\right\}.
\end{equation}
\end{lemma}

\begin{proof}
Write $\xi_i:=\beta_{X_i}-T_{\theta_0}$, $i=1,\dots,n$. These are independent and identically distributed centred random elements of the separable Hilbert space $H_{-m_K}$. By \eqref{Eq:Concentration} and the elementary bound $w^k\le k!\,A_T^ke^{w/A_T}$, valid for any $w\ge0$,
$$
    \operatorname{E}^{(1)}_{\theta_0}\left[\|\xi_1\|^k_{H_{-m_K}}\right]
    \le
    k!\,A_T^k\,\operatorname{E}^{(1)}_{\theta_0}\left[e^{\|\xi_1\|_{H_{-m_K}}/A_T}\right]
    =
    \frac{k!}{2}\,V_T\,A_T^{k-2},
    \qquad k\ge2,
$$
with $V_T:=2A_T^2\operatorname{E}^{(1)}_{\theta_0}[e^{\|\xi_1\|_{H_{-m_K}}/A_T}]<\infty$. Thus, the random elements $\xi_i$ satisfy the Bernstein moment condition with variance parameter $V_T$ and scale parameter $A_T$. Noting that $\hat D_n=n^{-1}\sum_{i=1}^n\xi_i$, the Bernstein inequality for sums of independent centred random elements in a separable Hilbert space \citep[Theorem~3.3]{pinelis1994optimum} yields \eqref{Eq:BernsteinNorm}.

To deduce \eqref{Eq:BernsteinLinear}, note that, for every fixed $h\in H_{m_K}$, $|\xi_i(h)|\le\|\xi_i\|_{H_{-m_K}}\|h\|_{H_{m_K}}$, so that the real random variables $\xi_i(h)$ are centred and satisfy the same moment condition with $V_T\|h\|^2_{H_{m_K}}$ and $A_T\|h\|_{H_{m_K}}$ in place of $V_T$ and $A_T$; the classical scalar Bernstein inequality then applies.
\end{proof}

Next, we record an upper bound in $H_{m_K}$-norm for the Kullback--Leibler divergence, following from Assumptions \ref{Ass:Link} and \ref{Ass:KLTaylor}. The notation $B^{H_m}_\upsilon(\theta_0)=\{ \theta \in H_m:\|\theta - \theta_0\|_{H_m}\le \upsilon\}$, $m < \alpha\land\beta$, $\upsilon>0$, is used throughout.

\begin{lemma}
\label{Lem:KLUpper}
Let Assumptions \ref{Ass:Link} and \ref{Ass:KLTaylor} hold for some $\gamma\ge0$, $C_0, C_K,R_K>0$, $m_K\in[0,p]$ and $q_K\in(0,1]$, and let $C_U:=\tfrac12C_0\lambda_1^{-2(m_K+\gamma)}+C_KR_K^{q_K}$. Then,
$$
    K(\theta|\theta_0)\ \le\ C_U\|\theta-\theta_0\|^2_{H_{m_K}},
    \qquad \text{for all}\ \theta\in \Theta\cap  B^{H_{m_K}}_{R_K}(\theta_0).
$$
\end{lemma}

\begin{proof}
Fix $\theta\in \Theta\cap B^{H_{m_K}}_{R_K}(\theta_0)$ and write $h=\theta-\theta_0$. By \eqref{Eq:KLTaylor} and \eqref{Eq:LinkUpper},
$$
    K(\theta|\theta_0)
    \le
    \tfrac12 I(\theta_0)[h]+C_K\|h\|^{2+q_K}_{H_{m_K}}
    \le
    \tfrac12 C_0\|h\|^2_{H_{-\gamma}}+C_KR_K^{q_K}\|h\|^2_{H_{m_K}}.
$$
The claim follows since $\|h\|_{H_{-\gamma}}\le\lambda_1^{-(m_K+\gamma)}\|h\|_{H_{m_K}}$ because $m_K+\gamma\ge0$ and $\lambda_k\ge\lambda_1$ for all $k$.
\end{proof}

Finally, we derive a small ball probability estimate for the prior $\Pi$ in Hilbert scale norms. Its verification is standard in Bayesian nonparametrics \citep[Chapter 11]{ghosal2017fundamentals}. It is included for completeness and the convenience of the reader. Define
\begin{equation}
\label{Eq:Rho}
    \rho:=\frac{\alpha\wedge\beta-m_K}{2(\alpha-m_K)+d}\in(0,1/2),
    \qquad
    \delta:=1-2\rho\in(0,1),
    \qquad
    \upsilon_n:=n^{-\rho}.
\end{equation}

\begin{lemma}
\label{Lem:SmallBall}
Let $\theta_0\in H_\beta$ for some $\beta>0$, and let $\Pi$ be an $\alpha$-regular Gaussian series prior as in \eqref{Eq:Prior}, for some $\alpha>p$. Then, for any $0\le m < \alpha\wedge\beta$ and all sufficiently small $\upsilon>0$,
\begin{equation}
\label{Eq:SmallBall}
    -\log\Pi\big(B_\upsilon^{H_{m}}(\theta_0)\big)
    \ \lesssim\
    \upsilon^{-\frac{d}{\alpha-m}}
    +
    \upsilon^{-\frac{(2\alpha+d-2\beta)_+}{\beta-m}}.
\end{equation}
In particular, for $\rho,\delta$ and $\upsilon_n$ as in \eqref{Eq:Rho}, and all sufficiently large $n\in\N$,
\begin{equation}
\label{Eq:SmallBallRate}
    -\log\Pi(B^{H_{m_K}}_{\upsilon_n}(\theta_0))\ \lesssim\ n^{\delta}=n\upsilon_n^2 .
\end{equation}
\end{lemma}

\begin{proof}
In the orthonormal basis $\{\lambda_k^{-m}e_k\}_{k=1}^\infty$ of $H_{m}$, the coordinates of a draw from $\Pi$ are independent centred Gaussian random variables with variances $\lambda_k^{-2(\alpha-m)-d}\asymp k^{-1-2(\alpha-m)/d}$, $k\in\N$, cf.~\eqref{Eq:PolyGrowth}. Classical small-deviation asymptotics for weighted sums of squared Gaussians \citep[][Section 4.3]{bogachev1998gaussian} give, as $\upsilon\to0$,
$$
	-\log \Pi(B_\upsilon^{H_{m}}(0))
	=
	-\log\Pr\left(
	\sum_{k=1}^\infty
	\lambda_k^{-2(\alpha-m)-d}Z_k^2
	\le \upsilon^2
	\right)
	\asymp \upsilon^{-d/(\alpha-m)},
$$ 
yielding the first term of \eqref{Eq:SmallBall}. Next, recall from Section \ref{Subsec:GenThm} that the RKHS of $\Pi$ is $H_{\alpha+d/2}$. Let $\Theta_K=\operatorname{span}\{e_1,\dots,e_K\}$, $K\in\N$, with $P_K:H_m\to\Theta_K$ the associated projection operator.  For $\upsilon>0$, take $K(\upsilon)\in\N$ to be the smallest integer such that $\lambda_{K(\upsilon)}\geq \big((1+\|\theta_0\|_{H_\beta})/\upsilon\big)^{1/(\beta-m)}$. Then, $\lambda_{K(\upsilon)}\asymp \big((1+\|\theta_0\|_{H_\beta})/\upsilon\big)^{1/(\beta-m)}$ and
$$
    \|\theta_0-P_{K(\upsilon)}\theta_0\|_{H_{m}}
    \le\lambda_{K(\upsilon)}^{-(\beta-m)}\|\theta_0\|_{H_\beta}
    \le\upsilon;
    \qquad
    \|P_{K(\upsilon)}\theta_0\|^2_{H_{\alpha+d/2}}
    \le\lambda_{K(\upsilon)}^{(2\alpha+d-2\beta)_+}\|\theta_0\|^2_{H_\beta}.
$$
The first claim follows from the last display and an application of the Cameron--Martin theorem \citep[Corollary 2.6.18]{gine2016mathematical}. The second claim follows from the first upon choosing $m = m_K$ and $\upsilon = \upsilon_n$, since $\upsilon_n\to0$.
\end{proof}

Unlike the standard prior-mass-and-testing approach of \citet{ghosal2000convergence}, our proof does not identify the posterior contraction rate through the small ball relation \eqref{Eq:SmallBallRate}. Rather, Lemma \ref{Lem:SmallBall} identifies an auxiliary prior-mass scale used to lower bound the posterior normalising constant in the construction of the localising event.

We are now in a position to construct the localising event used in the proof of Proposition \ref{Prop:Stoch}. Denote the log-likelihood by $\ell_n(\theta):=\sum_{i=1}^n\log[f_\theta(X_i)/f_{\theta_0}(X_i)]$, $\theta\in\Theta$, which, in view of \eqref{Eq:ExpFamily} and \eqref{Eq:Bregman}, satisfies
\begin{equation}
\label{Eq:LikelihoodIdentity}
    \ell_n(\theta)
    =n\hat D_n(\theta-\theta_0)-nK(\theta\mid\theta_0),
    \qquad \theta\in\Theta.
\end{equation}

\begin{lemma}
\label{Lem:GoodEvent}
Let $\beta^\ast\in(s\vee p,\alpha\wedge\beta)$, $a\in(0,1/2)$ and $b>0$, and let $\delta\in(0,1)$ and $\upsilon_n>0$ be as in \eqref{Eq:Rho}. For constants $C_1,C_2>0$ to be chosen below, define the events 
$$
	\Ecal^{(1)}_n:=\{\|\hat D_n\|_{H_{-m_K}}\le n^{-a}\};
	\qquad \Ecal^{(2)}_n:=\{\textstyle\int_{H_{s\vee p}} e^{\ell_n(\theta)}\,d\Pi(\theta)\ge\ e^{-C_1n^{\delta}}\};
$$
\begin{align*}
    	\Ecal^{(3)}_n
	&:=\Big\{\textstyle\int_{H_{s\vee p}} (1+\|\theta-\theta_0\|_{H_s}^2)\,e^{\ell_n(\theta)}\,d\Pi(\theta)
        \ \le\ e^{n^{\delta}}\Big\};
\end{align*}
\begin{align*}
    \Ecal^{(4)}_n
    &:=\Big\{\textstyle\int_{(B^{H_{\beta^\ast}}_{n^b}(\theta_0))^c} (1+\|\theta-\theta_0\|_{H_s}^2)\,
    e^{\ell_n(\theta)}\,d\Pi(\theta)
        \ \le\ e^{-C_2n^{2b}/2}\Big\}.
\end{align*}
Set $\Ecal_n:=\bigcap_{j=1}^4\Ecal^{(j)}_n$. Then, $C_1,C_2$ can be chosen, depending only on the
model, the prior and the indices $s,\beta^\ast,a,b$, such that, for some $C>0$ and all $n\in\N$ large enough,
\begin{equation}
\label{Eq:GoodProb}
    \mu^{(n)}_{\theta_0}\big(\Ecal_n^c\big)
    \ \le\
    \exp\big\{-C\,n^{\varsigma}\big\},
    \qquad
    \varsigma:=\min\{1-2a,\ \delta,\ 2b\}>0 .
\end{equation}
\end{lemma}

\begin{proof}
We upper bound the probability of each event individually. For $\Ecal^{(1)}_n$, we apply \eqref{Eq:BernsteinNorm} with $u=n^{-a}$ obtaining $\mu^{(n)}_{\theta_0}\big((\Ecal^{(1)}_n)^c\big)\le2e^{-c_1 n^{1-2a}}$ for some $c_1>0$, where $1-2a>0$ since $a<1/2$.

The estimate for the second event is a standard argument in Bayesian nonparametrics and follows from the small ball probability lower bound in Lemma \ref{Lem:SmallBall}, cf.~\citep[Lemma 8.10]{ghosal2017fundamentals}. Indeed, restricting the integral to $B^{H_{m_K}}_{\upsilon_n}(\theta_0)$ and applying Jensen's inequality to the probability measure $\Pi\big(\cdot \mid B^{H_{m_K}}_{\upsilon_n}(\theta_0)\big)$,
\begin{align*}
    \int_{H_{s\vee p}}e^{\ell_n(\theta)}\,d\Pi(\theta)
    &\ \ge\
    \Pi\big(B^{H_{m_K}}_{\upsilon_n}(\theta_0)\big)
    \exp\left\{\int_{B^{H_{m_K}}_{\upsilon_n}(\theta_0)}\ell_n(\theta)\,d\Pi(\theta|B^{H_{m_K}}_{\upsilon_n}(\theta_0))\right\}.
\end{align*}
By the identity \eqref{Eq:LikelihoodIdentity},
$$
\begin{aligned}
	&\int_{B^{H_{m_K}}_{\upsilon_n}(\theta_0)}
	\ell_n(\theta)\,
	d\Pi\bigl(\theta\mid B^{H_{m_K}}_{\upsilon_n}(\theta_0)\bigr)\\
	&\quad=
	n\int_{B^{H_{m_K}}_{\upsilon_n}(\theta_0)}
	\hat D_n(\theta-\theta_0)\,
	d\Pi\bigl(\theta\mid B^{H_{m_K}}_{\upsilon_n}(\theta_0)\bigr)
	-
	n\int_{B^{H_{m_K}}_{\upsilon_n}(\theta_0)}
	K(\theta\mid\theta_0)\,
	d\Pi\bigl(\theta\mid B^{H_{m_K}}_{\upsilon_n}(\theta_0)\bigr)\\
	&\quad=
	n\hat D_n\left(
	\int_{B^{H_{m_K}}_{\upsilon_n}(\theta_0)}
	(\theta-\theta_0)\,
	d\Pi\bigl(\theta\mid B^{H_{m_K}}_{\upsilon_n}(\theta_0)\bigr)
	\right)
	-
	n\bar K_n
	=
	n\hat D_n(\bar h_n)-n\bar K_n.
	\end{aligned}
$$
where $\bar h_n:=\int_{B^{H_{m_K}}_{\upsilon_n}(\theta_0)}(\theta-\theta_0)\,d\Pi\bigl(\theta\mid B^{H_{m_K}}_{\upsilon_n}(\theta_0)\bigr)$ is a Bochner integral in $H_{m_K}$ and $\bar K_n
:=\int_{B^{H_{m_K}}_{\upsilon_n}(\theta_0)}K(\theta\mid\theta_0)\,d\Pi\bigl(\theta\mid B^{H_{m_K}}_{\upsilon_n}(\theta_0)\bigr)$. In the second identity, the interchange of $\hat D_n$ and the integral is justified since $\hat D_n\in H_{-m_K}=(H_{m_K})^*$ almost surely and $\bar h_n$ is a Bochner integral in $H_{m_K}$.  Since $B^{H_{m_K}}_{\upsilon_n}(\theta_0)$ is closed and convex in $H_{m_K}$, $\theta_0+\bar h_n$ lies in it, and it follows that $\|\bar h_n\|_{H_{m_K}}\le\upsilon_n$. On the other hand, since $\upsilon_n\le R_K$ for $n$ large, Lemma \ref{Lem:KLUpper} gives $\bar K_n\le C_U\upsilon_n^2$, whence $n\bar K_n\lesssim n\upsilon_n^2 = n^\delta$. If $\bar h_n=0$ the required stochastic bound follows immediately; otherwise, applying \eqref{Eq:BernsteinLinear} with $h=\bar h_n$ (which is deterministic) and $u=\upsilon_n^2$ yields
\begin{align*}
    \mu_{\theta_0}^{(n)}
    \left(
        n|\hat D_n(\bar h_n)|\geq n^\delta
    \right)
    &=
    \mu_{\theta_0}^{(n)}
    \left(
        |\hat D_n(\bar h_n)|\geq\upsilon_n^2
    \right)\\
    &\leq
    2\exp\left\{
        -\frac{n\upsilon_n^4}
        {2\bigl(
            V_T\|\bar h_n\|_{H_{m_K}}^2
            +
            A_T\|\bar h_n\|_{H_{m_K}}\upsilon_n^2
        \bigr)}
    \right\}\\
    &\leq
    2\exp\left\{
        -\frac{n\upsilon_n^2}
        {2(V_T+A_T\upsilon_n)}
    \right\}
    \leq
    2e^{-c_2n^\delta},
\end{align*}
for some $c_2>0$ and all sufficiently large $n$. Combining the above with the small ball probability lower bound \eqref{Eq:SmallBallRate} and taking sufficiently large $C_1>0$ then yields $ \mu_{\theta_0}^{(n)}((\Ecal^{(2)}_n)^c) \le 2e^{-c_2n^\delta}$.

Turning to the events $\Ecal^{(3)}_n$ and $\Ecal^{(4)}_n$, note that for every fixed $\theta$, $\operatorname{E}^{(n)}_{\theta_0}\big[e^{\ell_n(\theta)}\big]=\operatorname{E}_\theta^{(n)}[1]=1$. Fubini's theorem therefore gives, for any Borel set $A\subseteq H_{s\vee p}$,
\begin{equation}
\label{Eq:Fubini}
    \operatorname{E}^{(n)}_{\theta_0}
    \left[\int_A (1+\|\theta-\theta_0\|_{H_s}^2)\,e^{\ell_n(\theta)}\,d\Pi(\theta) \right]
    =
    \int_A (1+\|\theta-\theta_0\|_{H_s}^2)d\Pi(\theta).
\end{equation}
Taking $A=H_{s\vee p}$, we obtain $\operatorname{E}^{(n)}_{\theta_0}[\int_{H_{s\vee p}} (1+\|\theta-\theta_0\|_{H_s}^2)\,e^{\ell_n(\theta)}\,d\Pi(\theta)]=1+\operatorname{E}_\Pi[\|\theta-\theta_0\|^2_{H_s}]<\infty$ since $s<\alpha\wedge\beta$. Then, by Markov's inequality, $\mu^{(n)}_{\theta_0}\big((\Ecal^{(3)}_n)^c\big)=O(e^{-n^{\delta}})$. Lastly, taking $A=(B^{H_{\beta^\ast}}_{n^b}(\theta_0))^c$ in \eqref{Eq:Fubini}, the Cauchy--Schwarz inequality gives 
$$
	\int_{(B^{H_{\beta^\ast}}_{n^b}(\theta_0))^c}(1+\|\theta-\theta_0\|_{H_s}^2)d\Pi(\theta)
	\le\left( \operatorname{E}_\Pi\left[(1+\|\theta-\theta_0\|_{H_s}^2)^2\right]\right)^{1/2}\Pi\left((B^{H_{\beta^\ast}}_{n^b}(\theta_0))^c
	\right)^{1/2}.
$$
The expectation on the right-hand side is finite by Fernique's theorem \citep[Theorem~2.8.5]{bogachev1998gaussian}. Furthermore, since $\beta^\ast<\alpha$, $\Pi$ is supported on $H_{\beta^\ast}$, and the Borell--Sudakov--Tsirelson inequality
\citep[Theorem~2.5.8]{gine2016mathematical} gives 
$$
	\Pi((B^{H_{\beta^\ast}}_{n^b}(\theta_0))^c)
	\le \Pi(\theta : \|\theta\|_{H_{\beta^\ast}}> n^b - \|\theta_0\|_{H_{\beta^\ast}}) \le e^{-c_3n^{2b}}
$$ 
for some $c_3>0$ and all sufficiently large $n$. Another application of Markov's inequality implies that upon taking $C_2>0$ small enough, $\mu^{(n)}_{\theta_0}\big((\Ecal^{(4)}_n)^c\big)\le e^{-c_4n^{2b}/2}$ for some $c_4>0$. Combining the four obtained upper bounds proves \eqref{Eq:GoodProb}.
\end{proof}

%
%
%

\subsubsection{Bounds for the conditioning errors}

The next preparatory step for the proof of Proposition \ref{Prop:Stoch} is to control, over the localising event $\Ecal_n$ constructed in Lemma \ref{Lem:GoodEvent}, the Wasserstein distances between the probability measures $\nu_{\hat T_n},\nu_{T_{\theta_0}}$ and their respective conditionings $\tilde\nu_{\hat T_n},\tilde\nu_{T_{\theta_0}}$ to the set $\Bcal$ from \eqref{Eq:CondKernel}. To do so, we fix the auxiliary indices $\beta^\ast\in(s\vee p,\alpha\wedge\beta)$, $a\in(0,1/2)$ and $b>0$, appearing in Lemma \ref{Lem:GoodEvent} as follows. Recall the exponents $\rho = (\alpha\wedge\beta-m_K)/(2(\alpha-m_K)+d) \in(0,1/2)$, $\delta = 1-2\rho\in(0,1)$ from \eqref{Eq:Rho}. For $\tilde\beta\in(p,\alpha\wedge\beta]$, set $C_\ast(\tilde\beta):=(\tilde\beta+p+2r_K)/(\tilde\beta-p)$. The compatibility conditions \eqref{Eq:Floor} and \eqref{Eq:FloorOversmooth} of Theorem \ref{Theo:Main} imply $\delta(1+C_\ast(\alpha\wedge\beta))<1$. Since the function $C_\ast$ is decreasing and continuous, we can therefore choose $\beta^\ast\in(s\vee p,\alpha\wedge\beta)$, $a<1/2$ and $b>\delta/2$ such that
\begin{equation}
\label{Eq:ExponentWindows}
    b\,C_\ast(\beta^\ast)<a<\frac{1-\delta}{2}.
\end{equation}
Furthermore, for these choices, the quantity $\omega:=2b(p+r_K)/(\beta^\ast-p)$ satisfies
\begin{equation}
\label{Eq:ThetaB}
    a>b+\omega;
    \qquad
    1-\omega>\delta .
\end{equation}

We first show that, under this specification, on the event $\Ecal_n$, both the posterior probability of $\Bcal^c$ and its contribution to the posterior $H_s$-second moment are exponentially negligible. This ensures that the conditioning operation introduces an error smaller than any polynomial order.

\begin{lemma}
\label{Lem:Exit}
Let $b$ and $\omega$ be as in \eqref{Eq:ExponentWindows} and \eqref{Eq:ThetaB}, and let $\Ecal_n$ be the event from Lemma \ref{Lem:GoodEvent}. Then, there exists a constant $C>0$ such that, on $\Ecal_n$ and for all $n\in\N$ large enough,
\begin{equation}
\label{Eq:ExitBound}
    \int_{\Bcal^c} (1+\|\theta-\theta_0\|_{H_s}^2)\,
    e^{\ell_n(\theta)}\,d\Pi(\theta)
    \ \le\
    e^{-Cn^{\varsigma'}},
    \qquad
    \varsigma':=\min\{1-\omega,\ 2b\} >\delta.
\end{equation}
Consequently, on $\Ecal_n$, both
$\nu_{\hat T_n}(\Bcal^c)$ and $\int_{\Bcal^c}\|\theta-\theta_0\|^2_{H_s}\,d\nu_{\hat T_n}(\theta)$ are bounded by $e^{C_1n^{\delta}-Cn^{\varsigma'}}$, and hence decay faster than any negative power of $n$.
\end{lemma}

\begin{proof}
Split $\Bcal^c=(\Bcal^c\cap B^{H_{\beta^\ast}}_{n^b}(\theta_0))\cup(\Bcal^c\cap(B^{H_{\beta^\ast}}_{n^b}(\theta_0))^c)$. The integral over $\Bcal^c\cap(B^{H_{\beta^\ast}}_{n^b}(\theta_0))^c$ is bounded by $e^{-C_2n^{2b}/2}$ on $\Ecal^{(4)}_n$. Hence, it suffices to consider the integral over $\Bcal^c\cap B^{H_{\beta^\ast}}_{n^b}(\theta_0)$, on which $h:=\theta-\theta_0$ satisfies $\|h\|_{H_p}>R$ and $\|h\|_{H_{\beta^\ast}}\le n^b$, upper bounding, in view of the likelihood identity \eqref{Eq:LikelihoodIdentity},
\begin{equation}
\label{Eq:ThingToBound}
	\int_{\Bcal^c\cap B_{n^b}^{H_{\beta^\ast}}(\theta_0)}
	\bigl(1+\|h\|_{H_s}^2\bigr)
	\exp\left\{
	n\hat D_n(h)
	-
	nK(\theta\mid\theta_0)
	\right\}
	\,d\Pi(\theta).
\end{equation}

Fix any $\theta\in \Bcal^c\cap B_{n^b}^{H_{\beta^\ast}}(\theta_0)$. Since $m_K\le p<\beta^\ast$ by the above choice of $\beta^\ast$, we have $\|h\|_{H_{m_K}}\le\lambda_1^{m_K-\beta^\ast}\|h\|_{H_{\beta^\ast}}$. Hence, on $\Ecal^{(1)}_n$,
$$
    n\big|\hat D_n(h)\big|
    \le
    n\|\hat D_n\|_{H_{-m_K}}\|h\|_{H_{m_K}}
    \lesssim
    n^{1-a+b}.
$$
Next, writing $p=-r_K(1-z)+\beta^\ast z$ with $z=(p+r_K)/(\beta^\ast+r_K)$, the Hilbert-scale interpolation inequality \eqref{Eq:Interp} gives $\|h\|_{H_p}\le\|h\|^{1-z}_{H_{-r_K}}\|h\|^{z}_{H_{\beta^\ast}}$ and hence
$$
    \|h\|_{H_{-r_K}}
    \ \ge\
    \|h\|_{H_p}^{\frac{\beta^\ast+r_K}{\beta^\ast-p}}
    \|h\|_{H_{\beta^\ast}}^{-\frac{p+r_K}{\beta^\ast-p}}
    \ \ge\
    R^{\frac{\beta^\ast+r_K}{\beta^\ast-p}}\,n^{-\omega/2},
$$
with $\omega$ as in \eqref{Eq:ThetaB}. Since the right-hand side tends to zero, the monotonicity of the lower bound function $\varphi$ in Assumption \ref{Ass:KLLB} and the relation $\varphi(u)\gtrsim u^2$ as $u\to0^+$ imply that $K(\theta\mid\theta_0)\gtrsim n^{-\omega}$, and hence $nK(\theta\mid\theta_0)\gtrsim n^{1-\omega}$. Finally, since $s\le\beta^\ast$, $1+\|h\|_{H_s}^2 \le1+\lambda_1^{2(s-\beta^\ast)}\|h\|^2_{H_{\beta^\ast}}\lesssim n^{2b}$. Combining the three obtained estimates and recalling $a>b+\omega$, cf.~\eqref{Eq:ThetaB}, which ensures $1-a+b<1-\omega$, we obtain for all sufficiently large $n$ that \eqref{Eq:ThingToBound} is upper bounded by
$$
    n^{2b}\exp\big\{c_1 n^{1-a+b}-c_2n^{1-\omega}\big\}
    \ \le\
    e^{-\frac{c_2}{2}n^{1-\omega}},
$$
for some $c_1,c_2>0$. This concludes the verification of \eqref{Eq:ExitBound}. The final assertion follows upon dividing \eqref{Eq:ExitBound} by the denominator bound of $\Ecal^{(2)}_n$, using that $1\vee \|\theta-\theta_0\|^2_{H_s}\le 1 + \|\theta-\theta_0\|^2_{H_s}$, and noting that $\varsigma'>\delta$.
\end{proof}

Based on the preceding exterior-mass estimates, we now prove that conditioning either posterior kernel $\nu_{T_{\theta_0}}$ or $\nu_{\hat T_n}$ on $\Bcal$ introduces a negligible error in Wasserstein distance (on $\Ecal_n$ in the latter case).

\begin{lemma}
\label{Lem:CondError}
For every $C>0$ and all $n\in\N$ large enough, 
\begin{equation}
\label{Eq:CondErrDet}
    \Wcal_2^{\Pcal_2(H_s)}
    \big(\nu_{T_{\theta_0}},\tilde\nu_{T_{\theta_0}}\big)
    \lesssim n^{-C}.
\end{equation}
Moreover, on the event $\Ecal_n$ from Lemma \ref{Lem:GoodEvent}, with $\varsigma'>0$ the exponent from Lemma \ref{Lem:Exit}, for some constant $C'>0$ and for all $n\in\N$ large enough,
\begin{equation}
\label{Eq:CondErrRandom}
    \Wcal_2^{\Pcal_2(H_s)}
    \big(\nu_{\hat T_n},\tilde\nu_{\hat T_n}\big)
    \ \le\
    e^{-C'n^{\varsigma'}}.
\end{equation}
\end{lemma}

\begin{proof}
Fix any Borel probability measure $\nu\in\Pcal_2(H_s)$ such that $\nu(\Bcal)>0$, and denote its conditioning on $\Bcal$ as
$$
	\tilde\nu(\cdot)=\nu(\cdot|\Bcal)
    =
    \frac{\nu|_{\Bcal}(\cdot)}{\nu(\Bcal)}
    =
    \frac{\nu|_{\Bcal}(\cdot)}
    {1-\nu(\Bcal^c)}.
$$
A simple coupling between $\nu$ and $\tilde\nu$ is provided by
$$
    \varpi
    :=
    (\operatorname{Id},\operatorname{Id})_\#
    \bigl(\nu|_{\Bcal}\bigr)
    +
    \frac{1}{1-\nu(\Bcal^c)}
    \bigl(\nu|_{\Bcal^c}\bigr)
    \otimes
    \bigl(\nu|_{\Bcal}\bigr).
$$
Indeed, for any measurable $A\subseteq H_s$, its first marginal is given by
\begin{align*}
    \varpi_1(A)
    &=
    \bigl[
	(\operatorname{Id},\operatorname{Id})_\#
	(\nu|_{\Bcal})
	\bigr](A\times H_s)
	+
	\frac{1}{1-\nu(\Bcal^c)}
	\bigl[
	(\nu|_{\Bcal^c})\otimes(\nu|_{\Bcal})
	\bigr](A\times H_s)\\
	&= \nu(A\cap \Bcal)
	+\frac{\nu(A\cap\Bcal^c)\nu(\Bcal)}{1-\nu(\Bcal^c)}
	=\nu(A\cap \Bcal) + \nu(A\cap \Bcal^c) = \nu(A),
\end{align*}
and, similarly, the second marginal satisfies $\varpi_2=\tilde\nu$. Consequently, since there is no transport associated to the diagonal component in the definition of the coupling $\varpi$,
\begin{align*}
    \Wcal_2^{\Pcal_2(H_s)}(\nu,\tilde\nu)^2
    &\le
    \int_{H_s\times H_s}
    \|\theta-\eta\|_{H_s}^2
    \,d\varpi(\theta,\eta)\\
    &
	=
	\frac{1}{1-\nu(\Bcal^c)}
	\int_{\Bcal^c}\int_{\Bcal}
	\|\theta-\eta\|_{H_s}^2
	\,d\nu(\eta)\,d\nu(\theta)\\
    &\le
    2\int_{\Bcal^c}
    \|\theta-\theta_0\|_{H_s}^2
    \,d\nu(\theta)
    +
    \frac{2\nu(\Bcal^c)}
    {1-\nu(\Bcal^c)}
    \int_{\Bcal}
    \|\eta-\theta_0\|_{H_s}^2
    \,d\nu(\eta)\\
    &=2\int_{\Bcal^c}
    \|\theta-\theta_0\|_{H_s}^2
    \,d\nu(\theta)
    +
    2\nu(\Bcal^c)
    \int_{\Bcal}
    \|\eta-\theta_0\|_{H_s}^2
    \,d\tilde\nu(\eta).
\end{align*}

We apply the preceding inequality with $\nu=\nu_{T_{\theta_0}}$. The estimates \eqref{Eq:DetExteriorBound} and \eqref{Eq:QuadraticGaussianDenom}, established in the proof of Proposition \ref{Prop:Det} and Lemma \ref{Lem:QuadraticDet}, respectively, imply that, for every $M>0$,
$$
    \int_{\Bcal^c}
    \bigl(
        1+\|\theta-\theta_0\|_{H_s}^2
    \bigr)
    e^{-nK(\theta\mid\theta_0)}
    \,d\Pi(\theta)
    \le
        n^{-M}Z_{I,\theta_0}^{(n)}
   ;
    \quad
    \int_{H_{s\vee p}}
    e^{-nK(\theta\mid\theta_0)}
    \,d\Pi(\theta)
    \ge
    \frac12Z_{I,\theta_0}^{(n)},
$$
for all sufficiently large $n$, where $Z_{I,\theta_0}^{(n)}$ is the normalising constant in \eqref{Eq:QuadraticDetMeasure}. Taking ratios shows that
$$
    \nu_{T_{\theta_0}}(\Bcal^c)
    \vee
    \int_{\Bcal^c}
    \|\theta-\theta_0\|_{H_s}^2
    \,d\nu_{T_{\theta_0}}(\theta)
    =
    O(n^{-M})
$$
for every $M>0$. In particular, $\nu_{T_{\theta_0}}(\Bcal^c)\le1/2$ for all sufficiently large $n$. Proposition \ref{Prop:Det}  also gives
$$
    \int_{\Bcal}
    \|\theta-\theta_0\|_{H_s}^2
    \,d\tilde\nu_{T_{\theta_0}}(\theta)
    \le
    2
    \int_{H_{s\vee p}}
    \|\theta-\theta_0\|_{H_s}^2
    \,d\nu_{T_{\theta_0}}(\theta)
    =
    O(1).
$$
Applying the coupling bound once more yields $\Wcal_2^{\Pcal_2(H_s)}\bigl(\nu_{T_{\theta_0}}, \tilde\nu_{T_{\theta_0}}\bigr)^2 =o(n^{-M})$ for every $M>0$. Since $M$ is arbitrary, replacing $M$ by $2C$ and taking square roots proves \eqref{Eq:CondErrDet}.

There remains to consider $\nu=\nu_{\hat T_n}$. By Lemma \ref{Lem:Exit}, on $\Ecal_n$,
$$
    \nu_{\hat T_n}(\Bcal^c)
    \vee
    \int_{\Bcal^c}
    \|\theta-\theta_0\|_{H_s}^2
    \,d\nu_{\hat T_n}(\theta)
    \le
    \exp\left\{
        C_1n^\delta-Cn^{\varsigma'}
    \right\}.
$$
Since $\varsigma'>\delta$, the right-hand side converges to zero. In particular, $\nu_{\hat T_n}(\Bcal^c)\le 1/2$ for all sufficiently large $n$. We next control the second moment of $\tilde\nu_{\hat T_n}$, which equals
$$
    \int_{\Bcal}
    \|\theta-\theta_0\|_{H_s}^2
    \,d\tilde\nu_{\hat T_n}(\theta)
    =
    \frac{
        \displaystyle
        \int_{\Bcal}
        \|\theta-\theta_0\|_{H_s}^2
        e^{\ell_n(\theta)}
        \,d\Pi(\theta)
    }{
        \displaystyle
        \int_{\Bcal}
        e^{\ell_n(\theta)}
        \,d\Pi(\theta)
    }.
$$
On $\Ecal_n$,
$$
    \int_{\Bcal}
    e^{\ell_n(\theta)}
    \,d\Pi(\theta)
    =
    \nu_{\hat T_n}(\Bcal)
    \int_{H_{s\vee p}}
    e^{\ell_n(\theta)}
    \,d\Pi(\theta)
    \ge
    \frac12e^{-C_1n^\delta};
    \ \
    \int_{\Bcal}
    \|\theta-\theta_0\|_{H_s}^2
    e^{\ell_n(\theta)}
    \,d\Pi(\theta)
    \le
    e^{n^\delta},
$$
whence $\int_{\Bcal}\|\theta-\theta_0\|_{H_s}^2\,d\tilde\nu_{\hat T_n}(\theta)\le2e^{(1+C_1)n^\delta}$. Thus, on this event, the coupling bound gives
$$
    \Wcal_2^{\Pcal_2(H_s)}
    \bigl(
        \nu_{\hat T_n},
        \tilde\nu_{\hat T_n}
    \bigr)^2
    \le
    2e^{C_1n^\delta-Cn^{\varsigma'}}
    +
    4e^{(1+2C_1)n^\delta-Cn^{\varsigma'}}
    \le
    e^{-2C'n^{\varsigma'}},
$$
holding for some $C'>0$ and for all sufficiently large $n$, having used that $\varsigma'>\delta$. Taking square roots proves \eqref{Eq:CondErrRandom}.
\end{proof}

%
%
%

\subsubsection{A Poincaré inequality for conditioned posteriors}

The central machinery used in the proof of the mixed-geometry transport estimate in Proposition \ref{Prop:CondTransport} is a Poincaré inequality for the conditioned posterior kernels $\tilde\nu_{\tau}$, $\tau\in\Theta^*$, defined in \eqref{Eq:CondKernel}. Its proof proceeds through a finite-dimensional approximation based on the conditioned Galerkin posteriors $\tilde\nu_{\tau,K}$, $K\in\N$, introduced in Appendix \ref{Subsec:MixedTransport}. Recall the notation $\Theta_K$, $\Theta_{0,K}$, $P_K$, $P_{0,K}$, $\Pi_{0,K}$, $\nu_{\tau,K}$ and $\Bcal_K$  introduced therein. We also use the operators $\Acal_{n,q}$ and the constants $\kappa_{n,q}$ defined in \eqref{Eq:Aq} and \eqref{Eq:CondPoincareRate}, respectively. The following lemma establishes convergence in Wasserstein distance of the Galerkin posteriors as $K\to\infty$.

\begin{lemma}
\label{Lem:GalerkinConvergence}
Suppose that $\alpha>p$, that $\theta_0\in\Theta=H_p$, and that $M$ is convex and Fréchet differentiable at $\theta_0$. Then, for every $n\in\N$, $\tau\in\Theta^*$ and $0\le r<\alpha$ such that $\theta_0\in H_r$,
$$
    \Wcal_2^{\Pcal_2(H_r)}(\nu_{\tau,K},\nu_\tau)
    \to 0,
    \qquad
    \Wcal_2^{\Pcal_2(H_r)}(\tilde\nu_{\tau,K},\tilde\nu_\tau)
    \to 0,
    \qquad K\to\infty.
$$
\end{lemma}

\begin{proof}
By the definition of the push-forward prior $\Pi_{0,K}$, for every nonnegative measurable function $F:H_r\to[0,\infty)$,
\begin{equation}
\label{Eq:GalerkinPullback}
    \int_{\Theta_{0,K}}F(u)\,d\nu_{\tau,K}(u)
    =
    \frac{
        \displaystyle
        \int_\Theta F(P_{0,K}\theta)
        e^{n[\tau(P_{0,K}\theta)-M(P_{0,K}\theta)]}\,d\Pi(\theta)
    }{
        \displaystyle
        \int_\Theta
        e^{n[\tau(P_{0,K}\theta)-M(P_{0,K}\theta)]}\,d\Pi(\theta)
    }.
\end{equation}
The analogous identity for $\tilde\nu_{\tau,K}$ is obtained by inserting in both integrals the indicator $1_{\{\|P_K(\theta-\theta_0)\|_{H_p}\le R\}}$. For $\Pi$-almost every $\theta$, $P_{0,K}\theta\to\theta$ as $K\to\infty$ in both $H_p$ and $H_r$, since $p,r<\alpha$ and $\theta_0\in H_p\cap H_r$. Moreover, by \eqref{Eq:GalerkinProjection}, $\|P_K(\theta-\theta_0)\|_{H_p}\to\|\theta-\theta_0\|_{H_p}$ (monotonically), and hence $1_{\{\|P_K(\theta-\theta_0)\|_{H_p}\le R\}} \to 1_\Bcal(\theta)$. Thus, for every bounded continuous function $F:H_r\to\R$, the integrands in \eqref{Eq:GalerkinPullback}, with or without the indicators, converge pointwise to their infinite-dimensional counterparts. The same pointwise convergence holds for $F(u)=\|u\|_{H_r}^2$.

Convexity of $M$ gives
$$
    \tau(P_{0,K}\theta)-M(P_{0,K}\theta)
    \le
    \tau(\theta_0)-M(\theta_0)
    +[\tau-DM(\theta_0)]\big(P_K(\theta-\theta_0)\big).
$$
Recalling that $\Theta=H_p$ and the identity $P_{0,K}\theta-\theta_0=P_K(\theta-\theta_0)$, we have
$$
\begin{aligned}
    \bigl[\tau-DM(\theta_0)\bigr]\big(P_K(\theta-\theta_0)\big)
    &\le
    \|\tau-DM(\theta_0)\|_{\Theta^*}
    \|P_K(\theta-\theta_0)\|_{\Theta}\\
    &\le
    \|\tau-DM(\theta_0)\|_{\Theta^*}
    \|\theta-\theta_0\|_{H_p},
\end{aligned}
$$
the last inequality following because $P_K$ is a contraction on $H_p$. Consequently,
$$
    \tau(P_{0,K}\theta)-M(P_{0,K}\theta)
    \le
    \tau(\theta_0)-M(\theta_0)
    +
    \|\tau-DM(\theta_0)\|_{\Theta^*}
    \|\theta-\theta_0\|_{H_p},
$$
uniformly in $K$. Moreover, $\|P_{0,K}\theta\|_{H_r}\le \|\theta_0\|_{H_r}+\|\theta-\theta_0\|_{H_r}$ uniformly in $K$. Since $p\vee r<\alpha$ by assumption, $\Pi$ defines a Gaussian probability measure on $H_{p\vee r}$, so that by Fernique’s theorem $\|\theta-\theta_0\|_{H_{p\vee r}}^2$ is exponentially integrable under $\Pi$. The preceding estimates then provide a common $\Pi$-integrable dominating function, uniformly in $K$, both for any bounded function $F$ and for $F(u)=\|u\|_{H_r}^2$.

By applying the dominated convergence theorem to both integrals on the right-hand side of \eqref{Eq:GalerkinPullback} we obtain, for every bounded continuous function $F: H_r\to\R$, as $K\to\infty$,
$$
	\int_{\Theta_{0,K}}F(u)\,d\nu_{\tau,K}(u)
	\to
	\int_\Theta F(\theta)\,d\nu_\tau(\theta),
$$
that is to say, $\nu_{\tau,K}\to\nu_\tau$ weakly in $H_r$. Applying the same argument with $F(u)=\|u\|_{H_r}^2$ gives
$$
	\int_{\Theta_{0,K}}\|u\|_{H_r}^2\,d\nu_{\tau,K}(u)
	\to
	\int_\Theta\|\theta\|_{H_r}^2\,d\nu_\tau(\theta).
$$
Analogous considerations, with the indicator $1_{\{\|P_K(\theta-\theta_0)\|_{H_p}\le R\}}$ inserted in the above integrals, extend these conclusions to $\tilde \nu_{\tau,K}$ and $\tilde \nu_{\tau}$. Both claims now follow from the characterisation of convergence in Wasserstein distance as weak convergence together with convergence of second moments \citep[Chapter~7]{ambrosio2008gradient}.
\end{proof}

We conclude by presenting the aforementioned Poincaré inequality, which provides a uniform (in $\tau$) upper bound for the variances $\operatorname{Var}_{\tilde\nu_\tau}[\phi(\theta)]$ of \emph{regular cylindrical functions} $\phi:\Theta\to\R$. These take the form $\phi(\theta)=f(\theta_1,\dots,\theta_J)$ for some $J\in\N$ and some bounded smooth function $f:\R^J\to\R$ with bounded derivatives of all orders. Recall, for any $q\in\R$, the gradient $\nabla^{H_q}\phi$ of $\phi$ in the $H_q$-geometry, given by
$$
	\nabla^{H_q}\phi(\theta)
	=
	\sum_{k=1}^J
	\lambda_k^{-2q}\partial_kf(\theta_1,\dots,\theta_J)e_k,
	\qquad \theta\in\Theta.
$$

\begin{proposition}
\label{Prop:CondPoincare}
Let Assumption \ref{Ass:Link} be satisfied and suppose that $M$ is twice continuously Fr\'echet differentiable on $\Theta$. Then, for every $q\in[0,\alpha+d/2]$, $n\in\N$, $\tau\in\Theta^*$, and any regular cylindrical function $\phi:\Theta\to\R$,
\begin{equation}
\label{Eq:CondPoincare}
    \operatorname{Var}_{\tilde\nu_\tau}\big[\phi(\theta)\big]
    \le
    \int_{\Bcal}
    \big\langle
        \Acal_{n,q}\nabla^{H_q}\phi(\theta),
        \nabla^{H_q}\phi(\theta)
    \big\rangle_{H_q}
    \,d\tilde\nu_\tau(\theta)
    \le
    \kappa_{n,q}
    \int_{\Bcal}
    \|\nabla^{H_q}\phi(\theta)\|_{H_q}^2
    \,d\tilde\nu_\tau(\theta).
\end{equation}
\end{proposition}

\begin{proof}
We first establish the corresponding inequality for the conditioned Galerkin measures $\tilde \nu_{\tau,K}$, $K\in\N$, namely that for any bounded and smooth function $\phi:\Theta_{0,K}\to\R$ with bounded derivatives,
\begin{equation}
\begin{split}
\label{Eq:CondPoincareGalerkin}
    \operatorname{Var}_{\tilde\nu_{\tau,K}}\big[\phi(\theta)\big]
    &\ \le\
    \int_{\Bcal_K}
    \big\langle
        \Acal_{n,q}\nabla^{H_q}\phi(\theta),\nabla^{H_q}\phi(\theta)
    \big\rangle_{H_q}
    d\tilde\nu_{\tau,K}(\theta)\\
    &\ \le\
    \kappa_{n,q}
    \int_{\Bcal_K}\big\|\nabla^{H_q}\phi(\theta)\big\|^2_{H_q}\,d\tilde\nu_{\tau,K}(\theta).
\end{split}
\end{equation}
Let $q\in[0,\alpha+d/2]$, $n, K\in\N$, $\tau\in\Theta^*$ and $\phi$ be fixed. Identifying $\Theta_{0,K}$ with $\R^K$, recall from Remark \ref{Rem:GlobTransp} that the negative log-density of $\nu_{\tau,K}$, up to an additive constant, is equal to 
$V_{\tau,K}(\theta)=n\{M(\theta)-\tau(\theta)\}
+\frac12\sum_{k=1}^K\lambda_k^{2\alpha+d}\theta_k^2$,
$\theta\in\Theta_{0,K}$.
Crucially, $\tau$ enters $V_{\tau,K}$ linearly, and we have, for all $\theta\in\Theta_{0,K}$,
$$
	DV_{\tau,K}(\theta)[h]
	=
	n\bigl\{DM(\theta)[h]-\tau(h)\bigr\}
	+
	\sum_{k=1}^K
	\lambda_k^{2\alpha+d}\theta_kh_k,
	\qquad h\in\Theta_K,
$$
and 
$$
    D^2V_{\tau,K}(\theta)[h,g]
    =
    nD^2M(\theta)[h,g]+\langle h,g\rangle_{H_{\alpha+d/2}},
    \qquad h,g\in\Theta_K,
$$
the latter being independent of $\tau$. For $\theta\in\Theta_{0,K}$, let $ \operatorname{Hess}^{H_q}V_{\tau,K}(\theta):\Theta_K\to\Theta_K$ denote the $H_q$-Hessian of $V_{\tau,K}$ at $\theta$, namely the unique self-adjoint operator satisfying
$$
	\left\langle \operatorname{Hess}^{H_q}V_{\tau,K}(\theta)[h],g\right\rangle_{H_q}
	=D^2V_{\tau,K}(\theta)[h,g],\qquad h,g\in\Theta_K,
$$
given by
$$ 
	\operatorname{Hess}^{H_q}V_{\tau,K}(\theta)[h]
	=\sum_{k=1}^K\lambda_k^{-2q}\left[nD^2M(\theta)[h,e_k]+\lambda_k^{2\alpha+d}h_k\right]e_k,
	\qquad h\in\Theta_K.
$$

The set $\Bcal_K = \Bcal\cap\Theta_{0,K}$ is convex and contained in $B^\Theta_{R_0}(\theta_0)$ by the choice of $\Bcal$ preceding \eqref{Eq:CondKernel}, so the local coercivity condition \eqref{Eq:LinkLower} gives, for $\theta\in\Bcal_K$ and $h\in\Theta_K$,
$$
    D^2V_{\tau,K}(\theta)[h,h]
    \ \ge\
    nc_0\|h\|^2_{H_{-\gamma}}+\|h\|^2_{H_{\alpha+d/2}}
    =
    \sum_{k=1}^K
    \lambda_k^{-2q}\big(nc_0\lambda_k^{-2\gamma}+\lambda_k^{2\alpha+d}\big)\lambda_k^{2q}h_k^2.
$$
This implies that, as self-adjoint operators on $\Theta_K$ equipped with the $H_q$-inner product,
$$
    \Acal_{n,q,K}^{1/2}\,
    \operatorname{Hess}^{H_q}V_{\tau,K}(\theta)\,
    \Acal_{n,q,K}^{1/2}
    \ \succeq\ I,
    \qquad \text{for all}\ \theta\in\Bcal_K,
$$
where $\Acal_{n,q,K}$ is the restriction of \eqref{Eq:Aq} to $\Theta_K$ (i.e., to the first $K$ coordinates). The linear change of variables $z=\Acal_{n,q,K}^{-1/2}(\theta-\theta_0)$ maps $\Bcal_K$ onto the convex set $\Dcal_{n,q,K}:=\Acal_{n,q,K}^{-1/2}(\Bcal_K-\theta_0)\subset \Theta_K$. Let $\bar\nu_{\tau,K}$ denote the push-forward of $\tilde\nu_{\tau,K}$ on $\Theta_K$ under this transformation. As the constant Jacobian of the change of variables is absorbed into the normalising constant, we have
$$
    d\bar\nu_{\tau,K}(z)
    =
    \frac{
        \mathbf 1_{\Dcal_{n,q,K}}(z)
        e^{-W_{\tau,K}(z)}
    }{
        \displaystyle
        \int_{\Dcal_{n,q,K}}
        e^{-W_{\tau,K}(z)}\,dz
    }
    \,dz,
    \qquad z \in \Theta_K,
$$
where $W_{\tau,K}(z):=V_{\tau,K}\bigl(\theta_0+\Acal_{n,q,K}^{1/2}z\bigr)$. Since $\Acal_{n,q,K}^{1/2}$ is self-adjoint in the $H_q$-inner product, the chain rule and the preceding curvature bound give
$$
    \operatorname{Hess}^{H_q}W_{\tau,K}(z)
    =
    \Acal_{n,q,K}^{1/2}
    \operatorname{Hess}^{H_q}V_{\tau,K}
    \bigl(
        \theta_0+\Acal_{n,q,K}^{1/2}z
    \bigr)
    \Acal_{n,q,K}^{1/2}
    \succeq I,
    \quad
    \text{for all}\ z\in\Dcal_{n,q,K}.
$$
Thus, $\bar\nu_{\tau,K}$ has unit lower curvature on a convex domain. The Bakry--\'Emery criterion on convex domains \citep[][Theorem 2.1]{kolesnikov2016riemannian} therefore yields the log-Sobolev inequality
$$
	\operatorname{Ent}_{\bar\nu_{\tau,K}}\left[\bar\phi^2(z)\right]
    \leq
    2
    \int_{\Dcal_{n,q,K}}
    \|\nabla^{H_q}\bar\phi(z)\|_{H_q}^2
    \,d\bar\nu_{\tau,K}(z)
$$
holding for every bounded smooth function $\bar\phi:\Theta_K\to\mathbb R$ with bounded derivatives, where $\operatorname{Ent}_{\bar\nu_{\tau,K}}[\bar\phi^2]:= \int_{\Dcal_{n,q,K}} \bar\phi^2(z)\log\left(\bar\phi^2(z)/\int_{\Dcal_{n,q,K}}\bar\phi^2\,d\bar\nu_{\tau,K}\right)d\bar\nu_{\tau,K}(z)$. A standard linearisation argument applied to the above log-Sobolev inequality \citep[e.g.][Proposition~5.1.3]{bakry2014analysis} then yields the Poincaré inequality
$$
    \operatorname{Var}_{\bar\nu_{\tau,K}}\left[\bar\phi(z)\right]
    \leq
    \int_{\Dcal_{n,q,K}}
    \|\nabla^{H_q}\bar\phi(z)\|_{H_q}^2
    \,d\bar\nu_{\tau,K}(z).
$$
In particular, for $\bar\phi(z):=\phi\bigl(\theta_0+\Acal_{n,q,K}^{1/2}z\bigr)$, with $\phi$ the function fixed at the beginning of the proof, we have
$$
    \nabla^{H_q}\bar\phi(z)
    =
    \Acal_{n,q,K}^{1/2}
    \nabla^{H_q}\phi
    \bigl(
        \theta_0+\Acal_{n,q,K}^{1/2}z
    \bigr),
    \qquad z\in\Theta_K.
$$
Applying the inverse change of variables $\theta = \theta_0 + \Acal_{n,q,K}^{1/2}z$ in the Poincar\'e inequality for $\bar\nu_{\tau,K}$ gives
$$
    \operatorname{Var}_{\tilde\nu_{\tau,K}}[\phi(\theta)]
    \leq
    \int_{\Bcal_K}
    \left\langle
        \Acal_{n,q,K}\nabla^{H_q}\phi(\theta),
        \nabla^{H_q}\phi(\theta)
    \right\rangle_{H_q}
    \,d\tilde\nu_{\tau,K}(\theta).
$$
Since $\Acal_{n,q,K}$ is the restriction of $\Acal_{n,q}$ to $\Theta_K$, this proves the first inequality in \eqref{Eq:CondPoincareGalerkin}. The second follows upon noting that 
$$
	\kappa_{n,q,K}:=\sup_{\theta\in\Theta_K:\|\theta\|_{H_q}=1}\|\Acal_{n,q,K}\theta\|_{H_q} 
	=
	\max_{k=1,\dots,K}\frac{\lambda_k^{2(q+\gamma)}}{nc_0+\lambda_k^{2\alpha+2\gamma+d}}
	\le \kappa_{n,q},
$$
and that in view of the positivity and self-adjointness of $\Acal_{n,q,K}$,
$$
    \big\langle
        \Acal_{n,q,K}\nabla^{H_q}\phi,
        \nabla^{H_q}\phi
    \big\rangle_{H_q}
    \le \kappa_{n,q,K}
    \|\nabla^{H_q}\phi\|_{H_q}^2
    \le  \kappa_{n,q}\|\nabla^{H_q}\phi\|_{H_q}^2.
$$

There remains to let $K\to\infty$ in \eqref{Eq:CondPoincareGalerkin}. Let $\phi$ depend on the first $J$ coordinates. For $K\ge J$, the functions $\phi$, $\phi^2$, $\langle\Acal_{n,q}\nabla^{H_q}\phi,\nabla^{H_q}\phi\rangle_{H_q}$ and $\|\nabla^{H_q}\phi\|_{H_q}^2$ are bounded continuous cylindrical functions, and $\Acal_{n,q,K}\nabla^{H_q}\phi=\Acal_{n,q}\nabla^{H_q}\phi$. Lemma \ref{Lem:GalerkinConvergence} therefore gives convergence of every term in \eqref{Eq:CondPoincareGalerkin} to its counterpart in \eqref{Eq:CondPoincare}, proving the claim.
\end{proof}

%
%
%

\subsection{Bounds for the deterministic term}
\label{Subsec:DetTerm}

The other main ingredient in the proof of Theorem \ref{Theo:Main} is an upper bound for the deterministic term in \eqref{Eq:WassDecomp}, which we present in this section. Recall from \eqref{Eq:KLRepres} that the deterministic term can be written as the square root of the ratio of two Laplace-type integrals. We study these by developing a suitable infinite-dimensional extension of the classical Laplace method for integral approximation \citep{wong2001asymptotic}, building on ideas developed by \citet{dolera2024strong}.

\begin{proposition}
\label{Prop:Det}
Under the assumptions of Theorem \ref{Theo:Main}, for every $0\le s<\alpha\wedge\beta$,
\begin{equation}
\label{Eq:DetRate}
    \Wcal_2^{\Pcal_2(H_s)}
    \big(\nu_{T_{\theta_0}},\delta_{\theta_0}\big)    =
    O\left(
        n^{-\frac{\alpha\wedge\beta-s}{2\alpha+2\gamma+d}}
    \right),
     \qquad \text{as $n\to\infty$.}
\end{equation}
\end{proposition}

\begin{proof}
Recalling \eqref{Eq:KLRepres}, the squared deterministic term can be written as
\begin{equation}
\label{Eq:DetRatio}
\begin{split}
    \Wcal_2^{\Pcal_2(H_s)}
    \big(\nu_{T_{\theta_0}},\delta_{\theta_0}\big)^2
    =
    \frac{
        \int_{H_{s\vee p}}
        \|\theta-\theta_0\|_{H_s}^2
        e^{-nK(\theta|\theta_0)}\,d\Pi(\theta)
    }{
        \int_{H_{s\vee p}}
        e^{-nK(\theta|\theta_0)}\,d\Pi(\theta)
    }.
\end{split}
\end{equation}
We use a localisation argument to compare this ratio with the one arising for the probability measure $\mu^{(n)}_{I,\theta_0}$ from Lemma \ref{Lem:QuadraticDet} below, whose normalising constant is denoted by $Z^{(n)}_{I,\theta_0}$.

In view of the smoothness constraints \eqref{Eq:Floor} and \eqref{Eq:FloorOversmooth}, we may fix $\beta^\ast \in(s\vee p,\alpha\wedge\beta)$ such that
$$
    q_K\beta^\ast>(2+q_K)p+2r_K,
    \qquad
    (2+q_K)(\beta^\ast-m_K)>2\alpha+2\gamma+d.
$$
We then pick $a\in (1/(2+q_K),(\beta^\ast-m_K)/(2\alpha+2\gamma+d))$ and set $\eta_n:=n^{-a}$. With this choice,
\begin{equation}
\label{Eq:DetScaleCompatibility}
    n\eta_n^{2+q_K}\to0,
    \qquad
    \eta_n^{-2}
    n^{-\frac{2((\alpha\wedge\beta)-m_K)}{2\alpha+2\gamma+d}}
    \to0.
\end{equation}
Finally, recall the fixed ball $\Bcal=B_R^\Theta(\theta_0)\subset B^\Theta_{R_0}(\theta_0)\cap B^{H_{m_K}}_{R_K}(\theta_0)$ of $\Theta$-radius $R<R_0$ from \eqref{Eq:CondKernel}. Integrating the local Fisher lower bound \eqref{Eq:LinkLower} along the segment joining $\theta_0$ and $\theta$ gives
\begin{equation}
\label{Eq:DetLocalCoercivity}
    K(\theta|\theta_0)
    \ge \frac{c_0}{2}\|\theta-\theta_0\|_{H_{-\gamma}}^2,
    \qquad \theta\in \Bcal.
\end{equation}

Now let $\mu_{c_0,\theta_0}^{(n)}$ be the probability measure obtained from \eqref{Eq:QuadraticDetMeasure} below by replacing $I(\theta_0)[\theta-\theta_0]$ with $c_0\|\theta-\theta_0\|_{H_{-\gamma}}^2$, and let $Z_{c_0,\theta_0}^{(n)}$ denote its normalising constant. Equivalently, this corresponds to taking $c=c_0/2$ in the final assertion of Lemma \ref{Lem:QuadraticDet}. Hence, the conclusions of that lemma also apply to $\mu_{c_0,\theta_0}^{(n)}$. Note that, since $\Pi$ is Gaussian, $\mu_{c_0,\theta_0}^{(n)}$ also is Gaussian. Denote by $\bar\theta^{(n)}_{c_0}$ its mean. The covariance operator of $\mu_{c_0,\theta_0}^{(n)}$ is given by $\Sigma_{c_0}^{(n)} = \bigl(nc_0 L^{-2\gamma}+Q_\Pi^{-1}\bigr)^{-1}$. Write
\begin{align*}
    \sigma_{n,m_K}^2
    &:=
    \bigl\|
        L^{m_K}\Sigma_{c_0}^{(n)}L^{m_K}
    \bigr\|_{\Lcal(H_0)}
    =
    \sup_{k\geq1}
    \frac{\lambda_k^{2(m_K+\gamma)}}
         {nc_0+\lambda_k^{2\alpha+2\gamma+d}}
    \lesssim
    n^{-1+\frac{2(m_K+\gamma)}{2\alpha+2\gamma+d}},
\end{align*}
for the maximal variance of the centred Gaussian measure in the $H_{m_K}$-geometry, the last inequality following from the eigenvalue asymptotics \eqref{Eq:PolyGrowth}. Lemma \ref{Lem:QuadraticDet}, applied with $t=m_K$, gives
\begin{align*}
    \operatorname{E}_{\mu_{c_0,\theta_0}^{(n)}}
    \left[
        \|\theta-\bar\theta_{c_0}^{(n)}\|_{H_{m_K}}^2
    \right]
    +
    \|\bar\theta_{c_0}^{(n)}-\theta_0\|_{H_{m_K}}^2
    &=
    \operatorname{E}_{\mu_{c_0,\theta_0}^{(n)}}
    \left[
        \|\theta-\theta_0\|_{H_{m_K}}^2
    \right]
    \lesssim
    n^{-\frac{2((\alpha\wedge\beta)-m_K)}
    {2\alpha+2\gamma+d}},
\end{align*}
so that by Jensen's inequality,
$$
    \operatorname{E}_{\mu_{c_0,\theta_0}^{(n)}}
    \left[
        \|\theta-\bar\theta_{c_0}^{(n)}\|_{H_{m_K}}
    \right]
    \lesssim
    n^{-\frac{(\alpha\wedge\beta)-m_K}
    {2\alpha+2\gamma+d}};
    \qquad
    \|\bar\theta_{c_0}^{(n)}-\theta_0\|_{H_{m_K}}
    \lesssim
    n^{-\frac{(\alpha\wedge\beta)-m_K}
    {2\alpha+2\gamma+d}}.
$$

Since, by the choice made at the beginning of the proof,
$$
    a<
    \frac{\beta^\ast-m_K}{2\alpha+2\gamma+d}
    <
    \frac{(\alpha\wedge\beta)-m_K}{2\alpha+2\gamma+d}
    \leq
    \frac{\alpha-m_K}{2\alpha+2\gamma+d},
$$
both quantities in the preceding display are of smaller order than $\eta_n = n^{-a}$. In particular, for all sufficiently large $n$,
\begin{align*}
    \eta_n
    &-
    \|\bar\theta^{(n)}_{c_0}-\theta_0\|_{H_{m_K}}
    -
    \operatorname{E}_{\mu_{c_0,\theta_0}^{(n)}}
    \left[
        \|\theta-\bar\theta^{(n)}_{c_0}\|_{H_{m_K}}
    \right]
    \geq
    \frac{\eta_n}{2}.
\end{align*}
Therefore, the Gaussian isoperimetric inequality \citep[Theorem 2.6.12]{gine2016mathematical} applied to the $H_{m_K}$-valued Gaussian random element $\theta-\bar\theta_{c_0}^{(n)}$, $\theta\sim\mu_{c_0,\theta_0}^{(n)}$ gives, for some $c_1,c_2>0$,
\begin{align*}
    \mu_{c_0,\theta_0}^{(n)}
    \left(
        \|\theta-\theta_0\|_{H_{m_K}}>\eta_n
    \right)
    \leq
    \exp\left\{
        -\frac{c_1\eta_n^2}{\sigma_{n,m_K}^2}
    \right\}
    \leq
    \exp\left\{
        -c_2
        n^{1-\frac{2(m_K+\gamma)}{2\alpha+2\gamma+d}-2a}
    \right\}.
\end{align*}
Then, by \eqref{Eq:DetLocalCoercivity}, the definition of $\mu_{c_0,\theta_0}^{(n)}$, and the Cauchy--Schwarz inequality,
\begin{align*}
    &\int_{\Bcal\cap
	B^{H_{m_K}}_{\eta_n}(\theta_0)^c}
    \bigl(1+\|\theta-\theta_0\|_{H_s}^2\bigr)
    e^{-nK(\theta\mid\theta_0)}
    \,d\Pi(\theta)\\
    &\quad\leq
    Z_{c_0,\theta_0}^{(n)}
    \bigg[
        \mu_{c_0,\theta_0}^{(n)}
        \left(
            B^{H_{m_K}}_{\eta_n}(\theta_0)^c
        \right)
        +
        \left\{
            \int
            \|\theta-\theta_0\|_{H_s}^4
            \,d\mu_{c_0,\theta_0}^{(n)}(\theta)
        \right\}^{1/2}
        \mu_{c_0,\theta_0}^{(n)}
        \left(
            B^{H_{m_K}}_{\eta_n}(\theta_0)^c
        \right)^{1/2}
    \bigg]\\
    &\quad\lesssim
    Z_{c_0,\theta_0}^{(n)}
    \exp\left\{
        -\frac{c_2}{2}
        n^{1-\frac{2(m_K+\gamma)}{2\alpha+2\gamma+d}-2a}
    \right\},
\end{align*}
where the last inequality follows from the Gaussian concentration bound in the second to last display and the fourth moment estimate in Lemma \ref{Lem:QuadraticDet} below, applied with $t=s$. Moreover, since $Z_{c_0,\theta_0}^{(n)}\leq1$, the lower bound
\eqref{Eq:QuadraticGaussianDenom} gives
$$
    \frac{Z_{c_0,\theta_0}^{(n)}}
         {Z_{I,\theta_0}^{(n)}}
    \leq
    \exp\left\{
        C
        n^{1-\frac{2(\alpha\wedge\beta+\gamma)}
        {2\alpha+2\gamma+d}}
    \right\}.
$$
Thus, noting that
$$
    1-\frac{2(m_K+\gamma)}{2\alpha+2\gamma+d}-2a
    >
    1-\frac{2(\alpha\wedge\beta+\gamma)}{2\alpha+2\gamma+d},
$$
we may choose
$$
    \delta\in \left(1-\frac{2(\alpha\wedge\beta+\gamma)}{2\alpha+2\gamma+d},
    1-\frac{2(m_K+\gamma)}{2\alpha+2\gamma+d}-2a\right)
$$
to obtain, for some $c_3>0$ and all sufficiently large $n$,
\begin{equation}
\label{Eq:DetAnnulusBound}
\begin{split}
    &\int_{\Bcal
    \cap B^{H_{m_K}}_{\eta_n}(\theta_0)^c}
    \bigl(1+\|\theta-\theta_0\|_{H_s}^2\bigr)
    e^{-nK(\theta\mid\theta_0)}
    \,d\Pi(\theta)
    \leq
    e^{-c_3n^\delta}Z_{I,\theta_0}^{(n)}.
\end{split}
\end{equation}

There remains to bound the integral over $\Bcal^c$. We choose
$$
   b\in\left( \frac12\left(
        1-\frac{2(\alpha\wedge\beta+\gamma)}{2\alpha+2\gamma+d}
    \right)
   ,
    \frac{
        (\alpha\wedge\beta+\gamma)(\beta^\ast-p)
    }{
        (2\alpha+2\gamma+d)(p+r_K)
    }\right),
$$
and proceed by splitting $\Bcal^c$ $ = \Bcal_{n,1}\cup \Bcal_{n,2}$ with $\Bcal_{n,1}:=\{\theta: \|\theta-\theta_0\|_{H_p}>R,\|\theta-\theta_0\|_{H_{\beta^\ast}}\leq n^b\}
$, and $\Bcal_{n,2}:=\{\theta: \|\theta-\theta_0\|_{H_p}>R,\|\theta-\theta_0\|_{H_{\beta^\ast}}>n^b\}$.  For every
$\theta\in\Bcal_{n,1}$, the Hilbert scale interpolation inequality \eqref{Eq:Interp} gives
$$
    \|\theta-\theta_0\|_{H_{-r_K}}
    \geq
    \|\theta-\theta_0\|_{H_p}^{
        \frac{\beta^\ast+r_K}{\beta^\ast-p}
    }
    \|\theta-\theta_0\|_{H_{\beta^\ast}}^{
        -\frac{p+r_K}{\beta^\ast-p}
    }
    \geq
    R^{\frac{\beta^\ast+r_K}{\beta^\ast-p}}
    n^{-\frac{b(p+r_K)}{\beta^\ast-p}}.
$$
Combining this with the global KL lower bound \eqref{Eq:KLLower} yields $K(\theta\mid\theta_0)\gtrsim n^{-\frac{2b(p+r_K)}{\beta^\ast-p}}$ for all $\theta\in\Bcal_{n,1}$. Moreover, since $s<\beta^\ast$, $\|\theta-\theta_0\|_{H_s}\lesssim \|\theta-\theta_0\|_{H_{\beta^\ast}} \leq n^b$ for all $\theta\in\Bcal_{n,1}$. Therefore
\begin{align*}
    \int_{\Bcal_{n,1}}
    \bigl(
        1+\|\theta-\theta_0\|_{H_s}^2
    \bigr)
    e^{-nK(\theta\mid\theta_0)}
    \,d\Pi(\theta)
    &\lesssim
    (1+n^{2b})
    \exp\left\{
        -c_3
        n^{1-\frac{2b(p+r_K)}{\beta^\ast-p}}
    \right\}\\
    &\qquad\lesssim
    \exp\left\{
        -c_4
        n^{1-\frac{2b(p+r_K)}{\beta^\ast-p}}
    \right\},
\end{align*}
for constants $c_3,c_4>0$. For the integral over $\Bcal_{n,2}$, we instead use the trivial upper bound $e^{-nK(\theta\mid\theta_0)}\leq1$, which yields by the Cauchy--Schwarz inequality,
\begin{align*}
    &\int_{\Bcal_{n,2}}
    \bigl(
        1+\|\theta-\theta_0\|_{H_s}^2
    \bigr)
    e^{-nK(\theta\mid\theta_0)}
    \,d\Pi(\theta)\\
    &\qquad\leq
    \Pi\left(
        \|\theta-\theta_0\|_{H_{\beta^\ast}}>n^b
    \right)
    +
    \left\{
        \int
        \|\theta-\theta_0\|_{H_s}^4
        \,d\Pi(\theta)
    \right\}^{1/2}
    \Pi\left(
        \|\theta-\theta_0\|_{H_{\beta^\ast}}>n^b
    \right)^{1/2}.
\end{align*}
Since $\beta^\ast<\alpha$, the prior $\Pi$ can be regarded as a centred Gaussian Borel probability measure on $H_{\beta^\ast}$. Furthermore, $\theta_0\in H_{\beta^\ast}$, and Fernique's theorem \citep[Theorem~2.8.5]{bogachev1998gaussian} gives $\Pi(\theta: \|\theta-\theta_0\|_{H_{\beta^\ast}}>n^b)\lesssim e^{- c_5 n^{2b}}$,
for some $c_5>0$. The prior has a finite fourth moment in $H_s$, so this also implies
$$
    \int_{\Bcal_{n,2}}
    \bigl(
        1+\|\theta-\theta_0\|_{H_s}^2
    \bigr)
    e^{-nK(\theta\mid\theta_0)}
    \,d\Pi(\theta)
    \lesssim
    e^{-c_6n^{2b}},
$$
for some $c_6>0$. By the choice of $b$, we may find $\delta'>0$ such that
\begin{align*}
    1-\frac{2(\alpha\wedge\beta+\gamma)}{2\alpha+2\gamma+d}
    <\delta'
    <
    \min\left\{
        2b,\,
        1-\frac{2b(p+r_K)}{\beta^\ast-p}
    \right\}.
\end{align*}
Thus, adding the two obtained bounds for the integral over $\Bcal^c$ gives, for some $c_7>0$ and all sufficiently large $n$,
\begin{equation}
\label{Eq:DetExteriorBound}
    \int_{\Bcal^c}
    \bigl(
        1+\|\theta-\theta_0\|_{H_s}^2
    \bigr)
    e^{-nK(\theta\mid\theta_0)}
    \,d\Pi(\theta)
    \leq
    e^{-c_7n^{\delta'}}.
\end{equation}

By \eqref{Eq:DetAnnulusBound}, for every fixed $C>0$,
\begin{align*}
&\frac{
    \displaystyle
    \int_{\Bcal
    \cap B^{H_{m_K}}_{\eta_n}(\theta_0)^c}
    \bigl(
        1+\|\theta-\theta_0\|_{H_s}^2
    \bigr)
    e^{-nK(\theta\mid\theta_0)}
    \,d\Pi(\theta)
	}{
    n^{-C}Z_{I,\theta_0}^{(n)}
	}
	\leq
	n^Ce^{-c_3n^\delta}
	\to0.
\end{align*}
Moreover, \eqref{Eq:DetExteriorBound}, the denominator lower bound \eqref{Eq:QuadraticGaussianDenom} and the choice of $\delta'$ imply
\begin{align*}
	\frac{
    	\displaystyle
    	\int_{\Bcal^c}
    	\bigl(
    	    1+\|\theta-\theta_0\|_{H_s}^2
    	\bigr)
    	e^{-nK(\theta\mid\theta_0)}
    	\,d\Pi(\theta)
	}{
    	n^{-C}Z_{I,\theta_0}^{(n)}
	}
	\leq
	n^C
	\exp\left\{
	-c_7n^{\delta'}
	+
	C_{I,\theta_0}\,
	n^{1-\frac{2(\alpha\wedge\beta+\gamma)}
	{2\alpha+2\gamma+d}}
	\right\}
	\to0.
\end{align*}
Conclude that, for every fixed $C>0$,
\begin{align*}
	\int_{\Bcal
	\cap B^{H_{m_K}}_{\eta_n}(\theta_0)^c}
	\bigl(
	    1+\|\theta-\theta_0\|_{H_s}^2
	\bigr)
	e^{-nK(\theta\mid\theta_0)}
	\,d\Pi(\theta)
	&=
	o\left(
	    n^{-C}Z_{I,\theta_0}^{(n)}
	\right)
\end{align*}
as well as
\begin{align*}
	\int_{\Bcal^c}
	\bigl(
	    1+\|\theta-\theta_0\|_{H_s}^2
	\bigr)
	e^{-nK(\theta\mid\theta_0)}
	\,d\Pi(\theta)
	=
	o\left(
	    n^{-C}Z_{I,\theta_0}^{(n)}
	\right).
\end{align*}
Combining these estimates shows that the contributions to both integrals in \eqref{Eq:DetRatio} outside the shrinking $H_{m_K}$-ball are negligible relative to $Z_{I,\theta_0}^{(n)}$ at every polynomial order: for every fixed $C>0$,
\begin{align}
\label{Eq:DetRelativeLocalisation}
    \int_{H_{s\vee p}\cap B^{H_{m_K}}_{\eta_n}(\theta_0)^c}
    \bigl(
        1+\|\theta-\theta_0\|_{H_s}^2
    \bigr)
    e^{-nK(\theta\mid\theta_0)}
    \,d\Pi(\theta)
    =
    o\left(
        n^{-C}Z_{I,\theta_0}^{(n)}
    \right).
\end{align}

We conclude by comparing the two exponential weights inside $B^{H_{m_K}}_{\eta_n}(\theta_0)$. We first show that the quadratic comparison measure $\mu_{I,\theta_0}^{(n)}$ assigns asymptotically full mass to this shrinking ball. By Markov's inequality, \eqref{Eq:QuadraticGaussianBounds} applied with $t=m_K$, and \eqref{Eq:DetScaleCompatibility},
\begin{align*}
    \mu_{I,\theta_0}^{(n)}
    \left(
        B^{H_{m_K}}_{\eta_n}(\theta_0)^c
    \right)
    &\leq
    \eta_n^{-2}
    \int_{H_{s\vee p}}
    \|\theta-\theta_0\|_{H_{m_K}}^2
    \,d\mu_{I,\theta_0}^{(n)}(\theta)
    \lesssim
    \eta_n^{-2}
    n^{-\frac{2((\alpha\wedge\beta)-m_K)}
    {2\alpha+2\gamma+d}}
    =
    o(1).
\end{align*}
Set
$$
    \zeta_n
    :=
    \sup_{\theta\in
    H_{s\vee p}\cap B^{H_{m_K}}_{\eta_n}(\theta_0)}
    n\left|
        K(\theta\mid\theta_0)
        -
        \frac12
        I(\theta_0)[\theta-\theta_0]
    \right|.
$$
By Assumption \ref{Ass:KLTaylor} and the choice of $\eta_n$, $\zeta_n \lesssim n\eta_n^{2+q_K}=o(1)$.
Therefore, for every $\theta\in H_{s\vee p}\cap B^{H_{m_K}}_{\eta_n}(\theta_0)$,
\begin{align*}
e^{-\zeta_n}
	\exp\left\{
    -\frac n2I(\theta_0)[\theta-\theta_0]
	\right\}
	\leq
	e^{-nK(\theta\mid\theta_0)}
	\leq
	e^{\zeta_n}
	\exp\left\{
 	   -\frac n2I(\theta_0)[\theta-\theta_0]
	\right\}.
\end{align*}
It follows that
$$
	\int_{H_{s\vee p}\cap B^{H_{m_K}}_{\eta_n}(\theta_0)}
	e^{-nK(\theta\mid\theta_0)}
	\,d\Pi(\theta)
	\geq
	e^{-\zeta_n}
	Z_{I,\theta_0}^{(n)}
	\mu_{I,\theta_0}^{(n)}
	\left(
	    B^{H_{m_K}}_{\eta_n}(\theta_0)
	\right)
	=
	\bigl(1+o(1)\bigr)
	Z_{I,\theta_0}^{(n)}.
$$
Similarly,
\begin{align*}
	\int_{H_{s\vee p}\cap B^{H_{m_K}}_{\eta_n}(\theta_0)}
	\|\theta-\theta_0\|_{H_s}^2
	e^{-nK(\theta\mid\theta_0)}
	\,d\Pi(\theta)
	\leq
	e^{\zeta_n}
	Z_{I,\theta_0}^{(n)}
	\int_{H_{s\vee p}}
	\|\theta-\theta_0\|_{H_s}^2
	\,d\mu_{I,\theta_0}^{(n)}(\theta).
\end{align*}
Consequently,
\begin{align*}
	\frac{
    \displaystyle
    \int_{H_{s\vee p}\cap B^{H_{m_K}}_{\eta_n}(\theta_0)}
    \|\theta-\theta_0\|_{H_s}^2
    e^{-nK(\theta\mid\theta_0)}
    \,d\Pi(\theta)
	}{
    \displaystyle
    \int_{H_{s\vee p}\cap B^{H_{m_K}}_{\eta_n}(\theta_0)}
    e^{-nK(\theta\mid\theta_0)}
    \,d\Pi(\theta)
	}
	&\leq
	\frac{
	    e^{2\zeta_n}
	}{
	    \mu_{I,\theta_0}^{(n)}
	    \left(
	        B^{H_{m_K}}_{\eta_n}(\theta_0)
	    \right)
	}
	\int_{H_{s\vee p}}
	\|\theta-\theta_0\|_{H_s}^2
	\,d\mu_{I,\theta_0}^{(n)}(\theta)\\
	&=
	\bigl(1+o(1)\bigr)
	\int_{H_{s\vee p}}
	\|\theta-\theta_0\|_{H_s}^2
	\,d\mu_{I,\theta_0}^{(n)}(\theta).
\end{align*}
In view of \eqref{Eq:DetRelativeLocalisation} and the preceding lower bound for the local denominator, decomposing the numerator and denominator over $B^{H_{m_K}}_{\eta_n}(\theta_0)$ and its complement gives, for every fixed $C>0$,
\begin{align*}
	&\frac{
    \displaystyle
    \int_{H_{s\vee p}}
    \|\theta-\theta_0\|_{H_s}^2
    e^{-nK(\theta\mid\theta_0)}
    \,d\Pi(\theta)
	}{
    \displaystyle
    \int_{H_{s\vee p}}
    e^{-nK(\theta\mid\theta_0)}
    \,d\Pi(\theta)
	}\\
	&\quad=
	\frac{
    \displaystyle
	    \int_{H_{s\vee p}\cap B^{H_{m_K}}_{\eta_n}(\theta_0)}
    \|\theta-\theta_0\|_{H_s}^2
    e^{-nK(\theta\mid\theta_0)}
    \,d\Pi(\theta)
	}{
    \displaystyle
	    \int_{H_{s\vee p}\cap B^{H_{m_K}}_{\eta_n}(\theta_0)}
    e^{-nK(\theta\mid\theta_0)}
    \,d\Pi(\theta)
	}
	+
	o(n^{-C}).
\end{align*}
Hence,
\begin{align*}
	&\frac{
	    \displaystyle
	    \int_{H_{s\vee p}}
	    \|\theta-\theta_0\|_{H_s}^2
	    e^{-nK(\theta\mid\theta_0)}
	    \,d\Pi(\theta)
	}{
	    \displaystyle
	    \int_{H_{s\vee p}}
	    e^{-nK(\theta\mid\theta_0)}
	    \,d\Pi(\theta)
	}
	\leq
	\bigl(1+o(1)\bigr)
	\int_{H_{s\vee p}}
	\|\theta-\theta_0\|_{H_s}^2
	\,d\mu_{I,\theta_0}^{(n)}(\theta)
	+
	o(n^{-C}).
\end{align*}
Choose $C>\frac{2((\alpha\wedge\beta)-s)}{2\alpha+2\gamma+d}$. Lemma \ref{Lem:QuadraticDet}, applied with $t=s$, then yields
$$
\frac{
    \displaystyle
    \int_{H_{s\vee p}}
    \|\theta-\theta_0\|_{H_s}^2
    e^{-nK(\theta\mid\theta_0)}
    \,d\Pi(\theta)
}{
    \displaystyle
    \int_{H_{s\vee p}}
    e^{-nK(\theta\mid\theta_0)}
    \,d\Pi(\theta)
}
\lesssim
n^{-\frac{2((\alpha\wedge\beta)-s)}
{2\alpha+2\gamma+d}}.
$$
Taking square roots in \eqref{Eq:DetRatio} proves
\eqref{Eq:DetRate}.
\end{proof}

The last auxiliary lemma provides the normalising constant and moment estimates for the quadratic Gaussian comparison measures used in the preceding proof. These stem from the two-sided Fisher-information link condition in Assumption \ref{Ass:Link}, which allows analytical treatment of cases in which the Fisher information is not diagonal in the prior basis.

\begin{lemma}
\label{Lem:QuadraticDet}
Under the assumptions of Proposition \ref{Prop:Det}, consider the probability measure on $\Theta$ defined, for $\theta\in\Theta$, by
\begin{equation}
\label{Eq:QuadraticDetMeasure}
    d\mu^{(n)}_{I,\theta_0}(\theta) := \frac{\exp\left\{-\frac n2 I(\theta_0)[\theta-\theta_0]\right\}}{Z^{(n)}_{I,\theta_0}}d\Pi(\theta),
    \ \
    Z^{(n)}_{I,\theta_0}
    :=
    \int_{H_{s\vee p}}
    \exp\left\{-\frac n2 I(\theta_0)[\theta-\theta_0]\right\}
    d\Pi(\theta).
\end{equation}
Then, there exists a constant $C_{I,\theta_0}>0$ such that, for all $n$ large enough,
\begin{equation}
\label{Eq:QuadraticGaussianDenom}
 	Z^{(n)}_{I,\theta_0}
   	\ge
    	\exp\left\{
        -C_{I,\theta_0} n^{1-\frac{2(\alpha\wedge\beta+\gamma)}{2\alpha+2\gamma+d}}
    	\right\}.
\end{equation}
Furthermore, for every $0\le t<\alpha\wedge\beta$ and all $n$ large enough,
\begin{equation}
\label{Eq:QuadraticGaussianBounds}
    \int \|\theta-\theta_0\|_{H_t}^2\,d\mu^{(n)}_{I,\theta_0}(\theta)
    \lesssim
    n^{-\frac{2((\alpha\wedge\beta)-t)}{2\alpha+2\gamma+d}};
    \qquad
    \int \|\theta-\theta_0\|_{H_t}^4\,d\mu^{(n)}_{I,\theta_0}(\theta)
    \lesssim
    n^{-\frac{4((\alpha\wedge\beta)-t)}{2\alpha+2\gamma+d}}.
\end{equation}
These conclusions remain valid, with possibly different constants, if $I(\theta_0)[h]/2$ in \eqref{Eq:QuadraticDetMeasure} is replaced by $c\|h\|_{H_{-\gamma}}^2$, for any fixed $c>0$.
\end{lemma}

\begin{proof}
As in the proofs of Propositions \ref{Prop:CondTransport} and
\ref{Prop:CondPoincare}, we first establish the estimates on finite-dimensional Galerkin spaces $\Theta_K:=\operatorname{span}\{e_1,\ldots,e_K\}$, $K\in\N$, and then pass to the limit. Recall the canonical projection $P_K:\Theta\to\Theta_K$, and set $\theta_{0,K}:=P_K\theta_0$ and $\Pi_K:=(P_K)_\#\Pi$. Define the probability measure $\mu_{I,\theta_0,K}^{(n)}$ on $\Theta_K$ by
$$
    d\mu_{I,\theta_0,K}^{(n)}(\theta)
    :=
    \frac{
        \exp\left\{
            -\frac n2 I(\theta_0)
            [\theta-\theta_{0,K}]
        \right\}
    }{
        Z_{I,\theta_0,K}^{(n)}
    }
    d\Pi_K(\theta),
$$
with
$$
    Z_{I,\theta_0,K}^{(n)}
    :=
    \int_{\Theta_K}
    \exp\left\{
        -\frac n2 I(\theta_0)
        [\theta-\theta_{0,K}]
    \right\}
    d\Pi_K(\theta).
$$
Identifying $\Theta_K$ with $\mathbb R^K$ through the basis $\{e_1,\ldots,e_K\}$, let $I_K(\theta_0):=\bigl[(I(\theta_0)[e_j])(e_i)\bigr]_{i,j=1}^K$ be the Fisher-information matrix, and write $L_K=\operatorname{diag}(\lambda_1,\ldots,\lambda_K)$ and $Q_{\Pi,K}=L_K^{-2\alpha-d}$ for the restricted scale generator and prior covariance. A standard Gaussian conjugate computation then shows that $\mu_{I,\theta_0,K}^{(n)}
=\mathcal N_K(\theta_{0,K}+b_{n,K},\Sigma_{n,K})$, where
$$
    b_{n,K}
    =
    -\bigl(nI_K(\theta_0)+Q_{\Pi,K}^{-1}\bigr)^{-1}
    Q_{\Pi,K}^{-1}\theta_{0,K},
    \qquad
    \Sigma_{n,K}
    =
    \bigl(nI_K(\theta_0)+Q_{\Pi,K}^{-1}\bigr)^{-1}.
$$

By Assumption \ref{Ass:Link}, the normalising constant satisfies
$$
    Z_{I,\theta_0,K}^{(n)}
    \ge
    \int_{\Theta_K}
    \exp\left\{
        -\frac{nC_0}{2}
        \|\theta-\theta_{0,K}\|_{H_{-\gamma}}^2
    \right\}
    d\Pi_K(\theta).
$$
The integral on the right-hand side can be evaluated analytically by another Gaussian conjugate computation (now with diagonal covariance matrices), leading to the upper bound
\begin{align*}
    -\log Z_{I,\theta_0,K}^{(n)}
    &\lesssim
    \sum_{k=1}^K
    \log\bigl(1+n\lambda_k^{-(2\alpha+2\gamma+d)}\bigr)
    +
    \sum_{k=1}^K
    \frac{n\lambda_k^{2\alpha+d}}
         {n+\lambda_k^{2\alpha+2\gamma+d}}
    \theta_{0,k}^2\\
    &\lesssim
    n^{\frac d{2\alpha+2\gamma+d}}
    +
    n^{\frac{(2\alpha+d-2\beta)_+}{2\alpha+2\gamma+d}}
    \|\theta_0\|_{H_\beta}^2
   \lesssim
    n^{1-\frac{2(\alpha\wedge\beta+\gamma)}{2\alpha+2\gamma+d}},
\end{align*}
holding uniformly in $K$, where the second to last inequality follows from the eigenvalue asymptotics \eqref{Eq:PolyGrowth}. This establishes the finite-dimensional analogue of \eqref{Eq:QuadraticGaussianDenom}, uniformly in $K$.

Moving to the second moment, we have
\begin{align*}
	\int_{\Theta_K}
	\|\theta-\theta_{0,K}\|_{H_t}^2
	\,d\mu_{I,\theta_0,K}^{(n)}(\theta)
	&=\operatorname{E}_{\mu_{I,\theta_0,K}^{(n)}}\left[\|\theta - (\theta_{0,K} + b_{n,K})\|_{H_t}^2\right]
	+\|b_{n,K}\|_{H_t}^2\\
	&=\operatorname{E}_{\mu_{I,\theta_0,K}^{(n)}}\left[\|L^t_K[\theta - (\theta_{0,K} + b_{n,K})]\|_{H_0}^2\right]
	+\|b_{n,K}\|_{H_t}^2\\
	&=
	\operatorname{Tr}[L_K^{2t}\Sigma_{n,K}]
	+
	\|b_{n,K}\|_{H_t}^2.
\end{align*}
having used that under $\mu_{I,\theta_0,K}^{(n)}$, the vector $\theta-(\theta_{0,K}+b_{n,K})$ is centred multivariate Gaussian with covariance $\Sigma_{n,K}$, and hence
$
\operatorname{E}_{\mu_{I,\theta_0,K}^{(n)}}\!\left[
\|L_K^t[\theta-(\theta_{0,K}+b_{n,K})]\|_{H_0}^2
\right]
=
\operatorname{Tr}(L_K^t\Sigma_{n,K}L_K^t)
=
\operatorname{Tr}(L_K^{2t}\Sigma_{n,K})$.
We bound the two terms in the last display. By Assumption \ref{Ass:Link}, and by the order-reversing property of inversion for positive-definite matrices \citep[Section~7.7]{horn2012matrix},
$$
    \Sigma_{n,K}
    \preceq
    \bigl(nc_0L_K^{-2\gamma}
          +Q_{\Pi,K}^{-1}\bigr)^{-1}
    =
    \bigl(nc_0L_K^{-2\gamma}
          +L_K^{2\alpha+d}\bigr)^{-1},
$$
we then obtain, for every $0\le t<\alpha$,
\begin{align}
\label{Eq:QuadraticCovarianceGalerkin}
    \operatorname{Tr}
    [L_K^{2t}\Sigma_{n,K}]
    &\lesssim
    \sum_{k=1}^K
    \frac{\lambda_k^{2(t+\gamma)-(2\alpha+2\gamma+d)}}
         {1+n\lambda_k^{-(2\alpha+2\gamma+d)}}
    \lesssim
    n^{-\frac{2(\alpha-t)}{2\alpha+2\gamma+d}},
\end{align}
uniformly in $K$. To estimate the bias term, write $\bar\theta_{n,K}:=\theta_{0,K}+b_{n,K}$  for the mean of $\mu_{I,\theta_0,K}^{(n)}$. By its variational characterisation, it is the unique minimiser over $\Theta_K$ of
$$
    u\in\Theta_K\mapsto
    nI(\theta_0)[u-\theta_{0,K}]
    +\|u\|_{H_{\alpha+d/2}}^2.
$$
Let $K_n\in\N$ be the spectral cut-off chosen so that $\lambda_{K_n}\lesssim n^{1/(2\alpha+2\gamma+d)}\lesssim\lambda_{K_n+1}$. Since $\theta_0\in H_{\alpha\wedge\beta}$, the above display and Assumption \ref{Ass:Link} give, uniformly in $K$,
\begin{equation*}
\begin{split}
    n\|\theta_{0,K} - \bar\theta_{n,K}\|_{H_{-\gamma}}^2
    +\|\bar\theta_{n,K}\|_{H_{\alpha+d/2}}^2
    &\lesssim
    n\|
        \theta_{0,K}-P_{K_n\wedge K}\theta_0
    \|_{H_{-\gamma}}^2
    +
    \|P_{K_n\wedge K}\theta_0\|_{H_{\alpha+d/2}}^2
    \\
    &\lesssim
    n^{1-\frac{2(\alpha\wedge\beta+\gamma)}{2\alpha+2\gamma+d}}
    \|\theta_0\|_{H_{\alpha\wedge\beta}}^2,
\end{split}
\end{equation*}
the last inequality following from the standard projection estimates and the eigenvalue asymptotics \eqref{Eq:PolyGrowth}. It follows that
$$
    \|
        \bar\theta_{n,K}-P_{K_n\wedge K}\theta_0
    \|_{H_{-\gamma}}
    \le 
    \|
        \bar\theta_{n,K}-\theta_{0,K}
    \|_{H_{-\gamma}}
    +
    \|
        P_{K_n\wedge K}\theta_0-\theta_{0,K}
    \|_{H_{-\gamma}}
    \lesssim
    n^{-\frac{\alpha\wedge\beta+\gamma}{2\alpha+2\gamma+d}}
    \|\theta_0\|_{H_{\alpha\wedge\beta}}
$$
and, similarly,
$$
    \|
        \bar\theta_{n,K}-P_{K_n\wedge K}\theta_0
    \|_{H_{\alpha+d/2}}
    \lesssim
    n^{\frac{\alpha+d/2-(\alpha\wedge\beta)}
                 {2\alpha+2\gamma+d}}
    \|\theta_0\|_{H_{\alpha\wedge\beta}}.
$$
Interpolating between $H_{-\gamma}$ and $H_{\alpha+d/2}$ via the inequality \eqref{Eq:Interp}, 
\begin{align*}
    	\|\bar\theta_{n,K}-P_{K_n\wedge K}\theta_0\|_{H_t}
	&\lesssim
	\left(
	n^{-\frac{\alpha\wedge\beta+\gamma}{2\alpha+2\gamma+d}}
	\right)^{1-\frac{t+\gamma}{\alpha+d/2+\gamma}}
	\left(
	n^{\frac{\alpha+d/2-(\alpha\wedge\beta)}
	{2\alpha+2\gamma+d}}
	\right)^{\frac{t+\gamma}{\alpha+d/2+\gamma}}
	\|\theta_0\|_{H_{\alpha\wedge\beta}}\\
	&=
	n^{-\frac{(\alpha\wedge\beta)-t}{2\alpha+2\gamma+d}}
	\|\theta_0\|_{H_{\alpha\wedge\beta}},
\end{align*}
and since also
$$
    \|\theta_{0,K}-P_{K_n\wedge K}\theta_0
    \|_{H_t}
    \lesssim
    n^{-\frac{(\alpha\wedge\beta)-t}{2\alpha+2\gamma+d}}
    \|\theta_0\|_{H_{\alpha\wedge\beta}},
$$
we obtain from the triangle inequality, for all $0\le t<\alpha\wedge\beta$, uniformly in $K$,
\begin{equation}
\label{Eq:QuadraticBiasGalerkin}
    \|b_{n,K}\|_{H_t}
    =
    \|\bar\theta_{n,K}-\theta_{0,K}\|_{H_t}
    \lesssim
    n^{-\frac{(\alpha\wedge\beta)-t}{2\alpha+2\gamma+d}}
    \|\theta_0\|_{H_{\alpha\wedge\beta}}.
\end{equation}
Combining  \eqref{Eq:QuadraticCovarianceGalerkin} and \eqref{Eq:QuadraticBiasGalerkin} shows that
\begin{align*}
	\int_{\Theta_K}
	\|\theta-\theta_{0,K}\|_{H_t}^2
	\,d\mu_{I,\theta_0,K}^{(n)}(\theta)
	&
    	\lesssim
    	n^{-\frac{2((\alpha\wedge\beta)-t)}{2\alpha+2\gamma+d}}.
\end{align*}

Finally, the standard Gaussian fourth-moment identity gives
\begin{align*}
    \operatorname{E}_{\mu_{I,\theta_0,K}^{(n)}}
    \left[
        \|\theta-(\theta_{0,K}+b_{n,K})\|_{H_t}^4
    \right]
    &=
    \left\{
        \operatorname{Tr}
        \bigl[L_K^{2t}\Sigma_{n,K}\bigr]
    \right\}^2
    +
    2\operatorname{Tr}
    \left[
        \bigl(
            L_K^{2t}\Sigma_{n,K}
        \bigr)^2
    \right]\\
    &\le
    3
    \left\{
        \operatorname{Tr}
        \bigl[L_K^{2t}\Sigma_{n,K}\bigr]
    \right\}^2
    \lesssim
    	n^{-\frac{4((\alpha\wedge\beta)-t)}{2\alpha+2\gamma+d}},
\end{align*}
where for the last inequality we have again applied \eqref{Eq:QuadraticCovarianceGalerkin}. Thus, using \eqref{Eq:QuadraticBiasGalerkin} and the basic fact that $\|h+g\|_{H_t}^4\leq 8\|h\|_{H_t}^4+8\|g\|_{H_t}^4$, for all $h,g\in H_t$ and all $t\in\R$,
\begin{align*}
    \int_{\Theta_K}
    \|\theta-\theta_{0,K}\|_{H_t}^4
    \,d\mu_{I,\theta_0,K}^{(n)}(\theta)
    &\leq
    8\operatorname{E}_{\mu_{I,\theta_0,K}^{(n)}}
    \left[
        \|\theta-(\theta_{0,K}+b_{n,K})\|_{H_t}^4
    \right]
    +
    8\|b_{n,K}\|_{H_t}^4\\
    &\lesssim
    n^{-\frac{4((\alpha\wedge\beta)-t)}
    {2\alpha+2\gamma+d}}.
\end{align*}

To conclude, we pass to the limit as $K\to\infty$. Since $\Pi(H_{s\vee p})=1$, the normalising constant in \eqref{Eq:QuadraticDetMeasure} may equivalently be integrated over $\Theta$. Moreover, by the definition of $\Pi_K$,
$$
    Z_{I,\theta_0,K}^{(n)}
    =
    \int_\Theta
    \exp\left\{
        -\frac n2
        I(\theta_0)
        [P_K(\theta-\theta_0)]
    \right\}
    d\Pi(\theta).
$$
Similarly, for $q\in\{2,4\}$,
\begin{equation*}
\begin{split}
    \int_{\Theta_K}
    \|\theta-\theta_{0,K}\|_{H_t}^q
    \,d\mu_{I,\theta_0,K}^{(n)}(\theta)
    =
    \frac{
        \displaystyle
        \int_\Theta
        \|P_K(\theta-\theta_0)\|_{H_t}^q
        \exp\left\{
            -\frac n2
            I(\theta_0)[P_K(\theta-\theta_0)]
        \right\}
        d\Pi(\theta)
    }{
        Z_{I,\theta_0,K}^{(n)}
    }.
\end{split}
\end{equation*}
For every $t<\alpha$, $P_K\theta\to\theta$ in $H_t$, $\Pi$-almost surely, while $P_K\theta_0\to\theta_0$ in $H_t$ for every $t<\beta$. Furthermore,
$$
    \|P_K(\theta-\theta_0)\|_{H_t}^q
    \leq
    \|\theta-\theta_0\|_{H_t}^q,
$$
and the right-hand side is $\Pi$-integrable for $q\in\{2,4\}$. Assumption \ref{Ass:Link} also implies that the quadratic form $h\in\Theta\mapsto I(\theta_0)[h]$ extends continuously to $H_{-\gamma}$. Dominated convergence therefore gives $Z_{I,\theta_0,K}^{(n)}\to Z_{I,\theta_0}^{(n)}$
and
$$
    \int_{\Theta_K}
    \|\theta-\theta_{0,K}\|_{H_t}^q
    \,d\mu_{I,\theta_0,K}^{(n)}(\theta)
    \to
    \int_\Theta
    \|\theta-\theta_0\|_{H_t}^q
    \,d\mu_{I,\theta_0}^{(n)}(\theta).
$$
Since all the preceding estimates are uniform in $K$, letting $K\to\infty$ proves \eqref{Eq:QuadraticGaussianDenom} and \eqref{Eq:QuadraticGaussianBounds}. Furthermore, exactly the same argument applies when $I(\theta_0)[h]$ is replaced by $2c\|h||_{H_{-\gamma}}^2$, proving the final assertion.
\end{proof}

%
%
%
%
%

\section{Proofs for Section \ref{Sec:Applications}}
\label{Sec:ProofApplications}

We provide detailed verifications for the application of the general theory to the concrete exponential families considered in Section \ref{Sec:Applications}.

%
%
%

\subsection{Proof of Theorem \ref{Theo:DensityPCR}}
\label{Subsec:ProofDensity}

\begin{proof}
We work in the mean-zero wavelet Sobolev scale $\{\dot H^t([0,1]^d)\}_{t\in\R}$, taking $\Theta=\dot H^p([0,1]^d)$ for some $p>d/2$ that will be chosen sufficiently close to $d/2$ in the sequel. The mean map and Fisher information of model \eqref{Eq:DensityLogistic} are given, respectively, for $h\in\Theta$, by
$$
    T_\theta(h)= \operatorname{E}_\theta^{(1)}[h(X_1)] =\int_{[0,1]^d}h(x)f_\theta(x)\,dx;
$$
$$
    I(\theta)[h]
    =
    \operatorname{Var}_\theta^{(1)}(h(X_1))
    =
    \int_{[0,1]^d}h(x)^2f_\theta(x)\,dx
    -
    \left(
        \int_{[0,1]^d}h(x)f_\theta(x)\,dx
    \right)^2.
$$

We verify Assumptions \ref{Ass:Link}--\ref{Ass:Concentration}, partly following the exposition in Section 4.4 of \citet{dolera2024strong}.
First, note that
\begin{align*}
    \sup_{x\in[0,1]^d}
    \|\delta_x\|_{\dot H^{-p}}
    &=\sup_{x\in[0,1]^d}\left\{\sup_{h\in\dot H^p:\|h\|_{H^p}=1} |h(x)| \right\}\\
    &\le \sup_{h\in\dot H^p:\|h\|_{H^p}=1} \|h\|_{L^\infty}
    \lesssim \sup_{h\in\dot H^p:\|h\|_{H^p}=1} \|h\|_{H^p} = 1.
\end{align*}
This shows that
$\operatorname{E}_\theta^{(1)}[\|\delta_{X_1}\|_{\dot H^{-p}}]$ is uniformly bounded over $\theta\in\Theta$. Next, since, for every $\theta\in\dot H^p([0,1]^d)$, $\inf_{x\in[0,1]^d}f_\theta(x) \geq e^{-2\|\theta\|_{L^\infty}}$, the triangle inequality gives
\begin{equation}
\label{Eq:DensityLowerBound}
    \inf_{x\in[0,1]^d}f_\theta(x)
    \geq
    \exp\left\{
        -2\big(
            \|\theta_0\|_{L^\infty}+c_pR_0
        \big)
    \right\}
    =:c_0
\end{equation}
for every $\theta\in B_{R_0}^\Theta(\theta_0)$, where $c_p>0$ is the multiplicative constant in the continuous embedding 
$\dot H^p([0,1]^d)\subset C([0,1]^d)$. Moreover, since every $h\in \dot H^p([0,1]^d)$ has zero Lebesgue integral, the variational characterisation of the variance yields
\begin{equation*}
\begin{split}
    I(\theta)[h]
    &=
    \inf_{a\in\R}
    \int_{[0,1]^d}(h(x)-a)^2f_\theta(x)\,dx
    \\
    &\geq
    \left(\inf_{x\in[0,1]^d}f_\theta(x)\right)
    \inf_{a\in\R}
    \int_{[0,1]^d}(h(x)-a)^2\,dx
    \ge
    c_0
    \|h\|_{L^2}^2.
\end{split}
\end{equation*}
On the other hand,
\begin{equation*}
    I(\theta_0)[h]
    \leq
    \int_{[0,1]^d}h(x)^2f_{\theta_0}(x)\,dx
    \leq
    \|f_{\theta_0}\|_{L^\infty}\|h\|_{L^2}^2.
\end{equation*}
Recalling that $\dot L^2([0,1]^d) = \dot H^0([0,1]^d)$, this shows that, for any $R_0>0$, Assumption \ref{Ass:Link} holds with $\gamma=0$, $c_0$ as in \eqref{Eq:DensityLowerBound} and $C_0:=\|f_{\theta_0}\|_{L^\infty}$.

Turning to Assumption \ref{Ass:KLTaylor}, in view of the continuous embedding $\dot H^p([0,1]^d)\subseteq C([0,1]^d)$ and the chain rule, $M$ is infinitely Fr\'echet differentiable (in fact, analytic) on $\dot H^p([0,1]^d)$. Fix any $m_K\in(d/2,p]$. The first three derivatives of $M$ in any direction $h\in \dot H^p([0,1]^d)$ satisfy $DM(\theta)[h]=\operatorname{E}_\theta^{(1)}[h(X_1)]$, $D^2M(\theta)[h,h] = \operatorname{Var}_\theta^{(1)}(h(X_1))=I(\theta)[h]$
and
$$
    D^3M(\theta)[h,h,h]
    =
    \operatorname{E}_\theta^{(1)}
    \left[
        \left(
            h(X_1)-\operatorname{E}_\theta^{(1)}[h(X_1)]
        \right)^3
    \right].
$$
In particular, $\left|D^3M(\theta)[h,h,h] \right|\leq 8\|h\|_{L^\infty}^3\le 8c_{m_K}^3\|h\|_{H^{m_K}}^3$, and Taylor's formula applied to $ K(\theta\mid\theta_0) = M(\theta)-M(\theta_0)-T_{\theta_0}(\theta-\theta_0)$ therefore yields
$$
    \left|
        K(\theta\mid\theta_0)
        -
        \frac12 I(\theta_0)[\theta-\theta_0]
    \right|
    \le \frac{4c_{m_K}^3}{3}
    \|\theta-\theta_0\|_{H^{m_K}}^3,
$$
for all $\theta\in\dot H^p([0,1]^d)$, whence Assumption \ref{Ass:KLTaylor} holds with $q_K=1$, $C_K = 4c_{m_K}^3/3$ and any $R_K>0$.

We next verify the global lower bound required by Assumption \ref{Ass:KLLB}. For $\theta\in \dot H^p([0,1]^d)$, write $h:=\theta-\theta_0$ and set
$$
    g
    :=
    h-\operatorname{E}_{\theta_0}^{(1)}[h(X_1)],
    \qquad
    T_{\theta_0}(|g|)
    :=
    \operatorname{E}_{\theta_0}^{(1)}[|g(X_1)|].
$$
Then $\operatorname{E}_{\theta_0}^{(1)}[g(X_1)]=0$ and 
$K(\theta\mid\theta_0)
    =
    \log \operatorname{E}_{\theta_0}^{(1)}
    [
        e^{g(X_1)}
    ].
$
Since $\operatorname{E}_{\theta_0}^{(1)}[g_+(X_1)]=T_{\theta_0}(|g|)/2$, the elementary inequality $e^u\geq1+u+u_+^2/2$, $u\in\R$, gives
$$
    \operatorname{E}_{\theta_0}^{(1)}
    \left[
        e^{g(X_1)}
    \right]
    \geq
    1+\frac{T_{\theta_0}(|g|)^2}{8}.
$$
This yields a quadratic lower bound when $T_{\theta_0}(|g|)$ is bounded. For $T_{\theta_0}(|g|)\geq2$, if
$q=P_{\theta_0}^{(1)}(g(X_1)\geq0)$, the conditional Jensen's inequality gives
\begin{align*}
    \operatorname{E}_{\theta_0}^{(1)}
    \left[
        e^{g(X_1)}
    \right]
    &\geq
    \operatorname{E}_{\theta_0}^{(1)}
    \left[
        e^{g(X_1)}1_{g(X_1)\geq0}
    \right]\\
	&\geq
    P_{\theta_0}^{(1)}(g(X_1)\geq0)\exp\left\{\frac{T_{\theta_0}(|g|)}{2P_{\theta_0}^{(1)}(g(X_1)\geq0)}\right\}
    \geq
    e^{T_{\theta_0}(|g|)/2}.
\end{align*}
Combining the two preceding displays, we find that
$$
    K(\theta\mid\theta_0)
    \gtrsim
    T_{\theta_0}(|g|)^2\wedge T_{\theta_0}(|g|).
$$
Now since both $\theta$ and $\theta_0$ have zero Lebesgue integral, $h=g-\int_{[0,1]^d}g(x)\,dx$, and hence
$$
    \|h\|_{L^1}
    \leq
    2\|g\|_{L^1}
    \leq
    \frac{2}{\inf_{x\in[0,1]^d}f_{\theta_0}(x)}T_{\theta_0}(|g|)
    \le\frac{2}{e^{-2\|\theta_0\|_{L^\infty}}}T_{\theta_0}(|g|).
$$
For every $r_K>d/2$, the continuous embedding $\dot H^{r_K}([0,1]^d)\subset L^\infty ([0,1]^d)$ and duality give $\|h\|_{\dot H^{-r_K}}\lesssim\|h\|_{L^1}$. Hence
$$
    K(\theta\mid\theta_0)
    \geq
    \varphi\left(
        c_\varphi
        \|\theta-\theta_0\|_{\dot H^{-r_K}}
    \right),
$$
for suitable constants $c,c_\varphi>0$, where $\varphi(u):=c(u^2\wedge u)$, $u\geq0$. This proves that Assumption \ref{Ass:KLLB} is valid for every $r_K>d/2$.

We finally verify the concentration condition. For the index $m_K>d/2$ fixed above, $B_{m_K} := \sup_{x\in[0,1]^d}\|\delta_x\|_{\dot H^{-m_K}}<\infty$. The centred random element $Z_1:=\delta_{X_1}-T_{\theta_0}$ satisfies $\|Z_1\|_{\dot H^{-m_K}}\leq 2B_{m_K}$ almost surely, whence
$$
    \operatorname{E}_{\theta_0}^{(1)}
    \exp\left\{
        \|\delta_{X_1}-T_{\theta_0}\|_{\dot H^{-m_K}}/A_T
    \right\}
    \leq
    e^{2B_{m_K}/A_T}
    <\infty,
$$
for every $A_T>0$. This is precisely Assumption \ref{Ass:Concentration}. The indices $p$, $m_K$ and $r_K$ may all be chosen arbitrarily close to $d/2$. Since $\alpha>5d/2$ and $\alpha\leq\beta$, they can be chosen so that condition \eqref{Eq:Floor} holds; the additional oversmoothing condition \eqref{Eq:FloorOversmooth} is not required. Theorem \ref{Theo:Main}, with $\gamma=0$, $q_K=1$ and $\alpha\wedge\beta=\alpha$, now gives the asserted rate.
\end{proof}

%
%
%

\subsection{Proof of Theorem \ref{Theo:IntensityPCR}}
\label{Subsec:ProofIntensity}

\begin{proof}
We verify the assumptions of Theorem \ref{Theo:Main} in the setting of Section \ref{Subsec:IntensityEstim}, working in the wavelet Sobolev scale $\{H^t([0,1]^d)\}_{t\in\R}$ and taking $\Theta=H^p([0,1]^d)$ for some $p>d/2$ to be chosen below. The Sobolev embedding gives $B_p:=\sup_{x\in[0,1]^d}\|\delta_x\|_{H^{-p}} <\infty$.
Consequently,
$$
    \|N\|_{H^{-p}}
    \leq
    \sum_{j=1}^{J}\|\delta_{X_j}\|_{H^{-p}}
    \leq
    B_pJ,
$$
and therefore, recalling $J\sim \text{Po}(\Lambda_\theta)$, $\Lambda_\theta := \int_{[0,1]^d}e^{\theta(x)}dx$, 
$
	\operatorname{E}_\theta^{(1)}\big[\|N\|_{H^{-p}}\big] \leq B_p\Lambda_\theta.
$  
The latter is uniformly bounded over $B_{R_0}^\Theta(\theta_0)$, for every $R_0>0$, because  $\Lambda_\theta\leq e^{\|\theta\|_{L^\infty}} \leq \exp\{\|\theta_0\|_{L^\infty}+c_pR_0\}$, where $c_p>0$ denotes the multiplicative constant in the continuous embedding $H^p([0,1]^d)\subset C([0,1]^d)$. Next, the standard moment identities for Poisson stochastic integrals \citep[][Section~1.1]{kutoyants1998statistical} give, for all $h\in H^p([0,1]^d)$,
$$
    T_\theta(h)
    =
    \operatorname{E}_\theta^{(1)}[N(h)]
    =
    \int_{[0,1]^d} h(x)e^{\theta(x)}\,dx;
$$
$$
	I(\theta)[h]= 
	\operatorname{Var}_\theta^{(1)}[N(h)]
	=
	\int_{[0,1]^d} h(x)^2e^{\theta(x)}\,dx.
$$ 
For every $\theta$ satisfying
$\|\theta-\theta_0\|_{H^p}\leq R_0$, we have $\|\theta\|_{L^\infty}\leq\|\theta_0\|_{L^\infty}+c_pR_0$ and hence
$$
    I(\theta)[h]
    \geq
    \exp\big\{-\|\theta_0\|_{L^\infty}-c_pR_0\big\}
    \|h\|_{L^2}^2,
$$
whereas, at the truth,
\[
    I(\theta_0)[h]
    \leq
    e^{\|\theta_0\|_{L^\infty}}\|h\|_{L^2}^2.
\]
This shows that Assumption \ref{Ass:Link} holds with $\gamma=0$, any $R_0>0$, $c_0 = \exp\{-\|\theta_0\|_{L^\infty}-c_pR_0\}$ and $C_0=e^{\|\theta_0\|_{L^\infty}}$.

We proceed with the verification of Assumption \ref{Ass:KLTaylor} with similar steps to those in the proof of Theorem \ref{Theo:DensityPCR}. The log-partition function $M$ is smooth on $H^p([0,1]^d)$, and its first three derivatives along any direction $h\in H^p([0,1]^d)$ are $DM(\theta)[h] = T_\theta(h)$,
$D^2M(\theta)[h,h] = I(\theta)[h]$ and
\begin{align*}
    D^3M(\theta)[h,h,h]
    &=
    \int_{[0,1]^d} h(x)^3e^{\theta(x)}\,dx.
\end{align*}
Fix any $m_K\in(d/2,p]$ and $R_K>0$, and write $h=\theta-\theta_0$. If $\|h\|_{H^{m_K}}\leq R_K$, then, uniformly over $t\in[0,1]$,
\[
    \|\theta_0+th\|_{L^\infty}
    \leq
    \|\theta_0\|_{L^\infty}+c_{m_K}R_K.
\]
It follows that
\begin{align*}
    \left|
        D^3M(\theta_0+th)[h,h,h]
    \right|
    &\leq
    \exp\big\{\|\theta_0\|_{L^\infty}+c_{m_K}R_K\big\}
    \|h\|_{L^\infty}\|h\|_{L^2}^2\\
    &\lesssim
    \exp\big\{\|\theta_0\|_{L^\infty}+c_{m_K}R_K\big\}
    \|h\|_{H^{m_K}}^3.
\end{align*}
Taylor's theorem therefore gives
$$
    \left|
        M(\theta)-M(\theta_0)
        -DM(\theta_0)[h]
        -\frac12 I(\theta_0)[h]
    \right|
    \leq
    C_K\|h\|_{H^{m_K}}^3
$$
for a constant $C_K$ depending only on $\theta_0$, $m_K$ and $R_K$. This verifies Assumption \ref{Ass:KLTaylor} with $q_K=1$.

Moving to Assumption \ref{Ass:KLLB}, for $h=\theta-\theta_0$ the Kullback--Leibler divergence admits the explicit representation
$$
    K(\theta\mid\theta_0)
    =
    M(\theta)-M(\theta_0)-T_{\theta_0}(h)
 	=
    \int_{[0,1]^d} e^{\theta_0(x)}
    \big(e^{h(x)}-1-h(x)\big)\,dx.
$$
The function $\psi(u)=e^u-1-u$, $u\in\R$, satisfies $\psi(u)\geq c_\psi\big(u^2\wedge |u|\big)$ for a universal constant $c_\psi>0$. Since $\inf_{x\in[0,1]^d}e^{\theta_0(x)}\geq e^{-\|\theta_0\|_{L^\infty}}$, we obtain
$$
    K(\theta\mid\theta_0)
    \geq
    c_\psi e^{-\|\theta_0\|_{L^\infty}}
    \int_{[0,1]^d}\big(h(x)^2\wedge |h(x)|\big)\,dx.
$$
To relate the right-hand side to the $L^1$-norm of $h$, note that
$$
	\|h\|_{L^1}
	=
    	\int_{\{|h|\leq1\}}|h(x)|\,dx
	+
    \int_{\{|h|>1\}}|h(x)|\,dx,
$$
so that, by the Cauchy--Schwarz inequality on the set $\{x:|h(x)|\leq1\}$, whose Lebesgue measure is at most one,
\begin{align*}
    \int_{[0,1]^d}\big(h^2(x)\wedge|h(x)|\big)\,dx
    &=
    \int_{\{x:|h(x)|\leq1\}}h(x)^2\,dx+\int_{\{x:|h(x)|>1\}}|h(x)|\,dx\\
    &\geq
    \left( \int_{\{x:|h(x)|\leq1\}}|h(x)|\,dx\right)^2+\int_{\{x:|h(x)|>1\}}|h(x)|\,dx.
\end{align*}
Considering separately the cases $\|h\|_{L^1}\leq1$ and $\|h\|_{L^1}>1$, one can lower bound the right-hand side in the last display by $\frac{1}{4}(\|h\|_{L^1}^2\wedge \|h\|_{L^1})$. Consequently,
$$
    K(\theta\mid\theta_0)
    \gtrsim
    \|\theta-\theta_0\|_{L^1}^2
    \wedge
    \|\theta-\theta_0\|_{L^1}.
$$
Duality and the embedding $H^{r_K}([0,1]^d)\subset L^\infty([0,1]^d)$ then show that Assumption \ref{Ass:KLLB} is valid, with $\varphi(u)=u^2\wedge u$ and a suitable $c_\varphi>0$, for every $r_K>d/2$.

Finally, for the index $m_K>d/2$ fixed above, let $B_{m_K}:=\sup_{x\in[0,1]^d}\|\delta_x\|_{H^{-m_K}} <\infty$ and write $dT_{\theta_0}(x)=e^{\theta_0(x)}\,dx$, $Z_i=N_i-T_{\theta_0}$. If $J_i=N_i([0,1]^d)$ denotes the total number of points of $N_i$, then
$J_i\sim\operatorname{Poisson}(\Lambda_{\theta_0})$ and
$$
    \|Z_i\|_{H^{-m_K}}
    \leq
    B_{m_K}(J_i+\Lambda_{\theta_0}).
$$
Hence, for every $\lambda>0$,
\begin{align*}
    \operatorname{E}_{\theta_0}^{(1)}
    \left[
        \exp\left\{
            \lambda\|Z_i\|_{H^{-m_K}}
        \right\}
    \right]
    &\leq
    e^{\lambda B_{m_K}\Lambda_{\theta_0}}
    \operatorname{E}\left[e^{\lambda B_{m_K}J_i}\right]\\
    &=
    \exp\left\{
        \lambda B_{m_K}\Lambda_{\theta_0}
        +
        \Lambda_{\theta_0}\big(e^{\lambda B_{m_K}}-1\big)
    \right\}
    <\infty.
\end{align*}
Since $\lambda>0$ is arbitrary, taking $\lambda=1/A_T$ for any $A_T>0$
shows that the centred $H^{-m_K}$-valued random element $Z_1$ satisfies
$$
    \operatorname{E}_{\theta_0}^{(1)}
    \exp\left\{
        \|Z_1\|_{H^{-m_K}}/A_T
    \right\}
    <\infty .
$$
This recovers Assumption \ref{Ass:Concentration}. As in the preceding proof, $p$, $m_K$ and $r_K$ may be chosen arbitrarily close to $d/2$. Since $\alpha>5d/2$ and $\alpha\leq\beta$, they can be chosen so that \eqref{Eq:Floor} holds, while \eqref{Eq:FloorOversmooth} is not required. Applying Theorem \ref{Theo:Main} with $\gamma=0$, $q_K=1$ and $\alpha\wedge\beta=\alpha$ gives the asserted contraction rate.
\end{proof}

%
%
%
%
%

\subsection{Proof of Theorem \ref{Theo:InvProbPCR}}
\label{Subsec:ProofInvProb}

Unlike the preceding applications, the proof does not invoke Theorem \ref{Theo:Main} directly. Instead, we exploit the exact quadratic expression for the Kullback--Leibler divergence and the global validity of the link condition afforded by the underlying linear white noise model to apply its underlying deterministic and stochastic estimates without localisation. Particularly, the proof follows closely the Gaussian covariance and bias estimates established in Lemma \ref{Lem:QuadraticDet}, with an additional bound for the random fluctuation of the posterior mean. This removes the additional smoothness restrictions \eqref{Eq:Floor} and \eqref{Eq:FloorOversmooth} required by the general theorem.

\begin{proof}
Since $\alpha\wedge\beta>(d/2-\gamma)\vee0$ by assumption, we may fix $p\geq0$ such that
$p<\alpha\wedge\beta$ and $p+\gamma>d/2$. Then $\theta_0\in H_p$ and the prior is
supported on $H_p$, as required in the exponential-family formulation of Section \ref{Subsec:InvProb}.

Denote the normal operator by $A:=G^*G$. The link condition \eqref{Eq:InvProbLink} is equivalent to
\begin{equation}
\label{Eq:InvProbAComparison}
	c_G L^{-2\gamma}
	\preceq
	A
	\preceq
	C_G L^{-2\gamma}
\end{equation}
in the sense of quadratic forms on $H_0$. Recall the covariance operator $Q_\Pi=L^{-(2\alpha+d)}$. We carry out the calculation first on the Galerkin space $\Theta_K=\operatorname{span}\{e_1,\ldots,e_K\}$, $K\in\N$, with associated projection operators $P_K$, and then let $K\to\infty$. Set $L_K=L|_{\Theta_K}$, $A_K:=P_KA|_{\Theta_K}$, $
Q_{\Pi,K}:=Q_\Pi|_{\Theta_K}=L_K^{-2\alpha-d}$ and let $\Pi_K:=(P_K)_\#\Pi$, which is a centred Gaussian measure
on $\Theta_K$ with covariance $Q_{\Pi,K}$. Further, define the $\Theta_K$-valued random vectors
$$
    Y_K^{(n)}:=\sum_{k=1}^K Y^{(n)}(Ge_k)e_k,
    \qquad
    \xi_K:=\sum_{k=1}^K W(Ge_k)e_k.
$$
Then, under $\mu_{\theta_0}^{(n)}$, it holds that  $Y_K^{(n)}=P_KA\theta_0+n^{-1/2}\xi_K$, where $\xi_K$ is centred
Gaussian with covariance $A_K$. By the conjugacy property of Gaussian priors and likelihoods, the posterior $\Pi_K(\cdot\mid Y^{(n)})$ on $\Theta_K$ arising from prior $\Pi_K$ and observations $Y_K^{(n)}$ is, in view of \eqref{Eq:InvProbLikelihood},  Gaussian with posterior covariance $\Sigma_{n,K}:=(nA_K+Q_{\Pi,K}^{-1})^{-1}$  and posterior mean
$
    m_{n,K}
    :=
    n\Sigma_{n,K}Y_K^{(n)}
    =
    m_{n,K}^0+\zeta_{n,K},
$
where the deterministic component and the centred noise fluctuation are, respectively,
$$
    m_{n,K}^0:=n\Sigma_{n,K}P_KA\theta_0,
    \qquad
    \zeta_{n,K}:=\sqrt n\,\Sigma_{n,K}\xi_K.
$$

We first bound the posterior covariance. From the lower bound in \eqref{Eq:InvProbAComparison} and the order-reversing property of inversion,
$$
	\Sigma_{n,K}
	\preceq
	\big(nc_GL_K^{-2\gamma}+L_K^{2\alpha+d}\big)^{-1}.
$$
Consequently, for every $0\le s<\alpha$, uniformly in $K$,
\begin{equation}
\label{Eq:InvProbCovariance}
\begin{split}
	\operatorname{Tr}\big(L_K^{2s}\Sigma_{n,K}\big)
	&\le
	\sum_{k=1}^K
	\frac{\lambda_k^{2(s+\gamma)}}
	{nc_G+\lambda_k^{2\alpha+2\gamma+d}}
	\lesssim
	n^{-\frac{2(\alpha-s)}{2\alpha+2\gamma+d}},
\end{split}
\end{equation}
where the final estimate follows as in \eqref{Eq:QuadraticCovarianceGalerkin} from the eigenvalue growth \eqref{Eq:PolyGrowth}. Moreover,
$nA_K\preceq\Sigma_{n,K}^{-1}$ implies
$
	n\Sigma_{n,K}A_K\Sigma_{n,K}
	\preceq
	\Sigma_{n,K}.
$
After conjugating this inequality by $L_K^s$ and taking traces, we obtain
\begin{equation}
\label{Eq:InvProbNoise}
	\operatorname{E}_{\theta_0}^{(n)}
	\|\zeta_{n,K}\|_{H_s}^2
	=
	n\operatorname{Tr}
	\big(L_K^{2s}\Sigma_{n,K}A_K\Sigma_{n,K}\big)
	\lesssim
	n^{-\frac{2(\alpha-s)}{2\alpha+2\gamma+d}}.
\end{equation}

There remains to control the bias $m_{n,K}^0-\theta_0$. The vector $m_{n,K}^0$ is the unique minimiser over $\Theta_K$ of the map $u\in\Theta_K\mapsto n\|G(u-\theta_0)\|_{\mathcal Y}^2+\|u\|_{H_{\alpha+d/2}}^2$. 
Let $K_n$ be such that $\lambda_{K_n}\asymp n^{1/(2\alpha+2\gamma+d)}$ and take $K\ge K_n$. Comparing the preceding objective at $m_{n,K}^0$ and at $P_{K_n}\theta_0$, using \eqref{Eq:InvProbAComparison} and standard spectral projection estimates, gives
\begin{equation}
\label{Eq:InvProbVariational}
\begin{split}
	n\|m_{n,K}^0-\theta_0\|_{H_{-\gamma}}^2
	+\|m_{n,K}^0\|_{H_{\alpha+d/2}}^2
	&\lesssim
	n\|\theta_0-P_{K_n}\theta_0\|_{H_{-\gamma}}^2
	+\|P_{K_n}\theta_0\|_{H_{\alpha+d/2}}^2
	\\
	&\lesssim
	n^{1-\frac{2(\alpha\wedge\beta+\gamma)}{2\alpha+2\gamma+d}}
	\|\theta_0\|_{H_{\alpha\wedge\beta}}^2.
\end{split}
\end{equation}
It follows that
$
\begin{aligned}
	\|m_{n,K}^0-P_{K_n}\theta_0\|_{H_{-\gamma}}
	&\lesssim
	n^{-\frac{\alpha\wedge\beta+\gamma}{2\alpha+2\gamma+d}}\|\theta_0\|_{H_{\alpha\wedge\beta}};
	\\
	\|m_{n,K}^0-P_{K_n}\theta_0\|_{H_{\alpha+d/2}}
	&\lesssim
	n^{\frac{\alpha+d/2-(\alpha\wedge\beta)}{2\alpha+2\gamma+d}}\|\theta_0\|_{H_{\alpha\wedge\beta}}.
\end{aligned}
$
Interpolating between $H_{-\gamma}$ and $H_{\alpha+d/2}$, and using $\|\theta_0-P_{K_n}\theta_0\|_{H_s}
\lesssim n^{-((\alpha\wedge\beta)-s)/(2\alpha+2\gamma+d)}$ $\|\theta_0\|_{H_{\alpha\wedge\beta}}$, then yields, for every $0\le s<\alpha\wedge\beta$,
\begin{equation}
\label{Eq:InvProbBias}
	\|m_{n,K}^0-\theta_0\|_{H_s}
	\lesssim
	n^{-\frac{(\alpha\wedge\beta)-s}{2\alpha+2\gamma+d}}\|\theta_0\|_{H_{\alpha\wedge\beta}},
\end{equation}
uniformly over $K\ge K_n$. This is the same variational interpolation argument used to establish
\eqref{Eq:QuadraticBiasGalerkin}. However, we note that in this context the smoothness conditions \eqref{Eq:Floor} and \eqref{Eq:FloorOversmooth} are not required.

Combining \eqref{Eq:InvProbCovariance}--\eqref{Eq:InvProbBias} and using
$\alpha\wedge\beta\le\alpha$, we conclude that
\begin{equation}
\label{Eq:InvProbPosteriorRiskGalerkin}
	\operatorname{E}_{\theta_0}^{(n)}
	\left[
		\int_{\Theta_K}
		\|\theta-\theta_0\|_{H_s}^2
		d\Pi_K(\theta\mid Y^{(n)})
	\right]
	\lesssim
	n^{-\frac{2((\alpha\wedge\beta)-s)}{2\alpha+2\gamma+d}}
\end{equation}

We pass to the limit $K\to\infty$ by a limiting argument similar to the one at the end of the proof of Lemma \ref{Lem:QuadraticDet}. For every nonnegative measurable function $F$ on $H_s$,
\begin{align*}
    \int_{\Theta_K} F(u)\,d\Pi_K(u\mid Y^{(n)})
    =
    \frac{
        \displaystyle\int_\Theta
        F(P_K\theta)
        \exp\left\{
            nY^{(n)}(GP_K\theta)
            -\frac n2\|GP_K\theta\|_{\mathcal Y}^2
        \right\}
        \,d\Pi(\theta)
    }{
        \displaystyle\int_\Theta
        \exp\left\{
            nY^{(n)}(GP_K\theta)
            -\frac n2\|GP_K\theta\|_{\mathcal Y}^2
        \right\}
        \,d\Pi(\theta)
    }.
\end{align*}
By the choice of $p$ and the second-moment calculation in Section \ref{Subsec:InvProb}, the sufficient statistic
$Y^{(n)}\circ G$ belongs almost surely to $H_{-p}$. Moreover, $P_K\theta\to\theta$ in $H_{s\vee p}$ for $\Pi$-almost every $\theta$. Fernique's theorem therefore provides integrable dominating functions for both integrals in the preceding display, also when $F(u)=\|u-\theta_0\|_{H_s}^2$. Dominated convergence, followed by Fatou's lemma and the uniform bound
\eqref{Eq:InvProbPosteriorRiskGalerkin}, consequently gives
\begin{equation*}
	\operatorname{E}_{\theta_0}^{(n)}
	\left[
		\int_\Theta
		\|\theta-\theta_0\|_{H_s}^2
		d\Pi(\theta\mid Y^{(n)})
	\right]
	\lesssim
	n^{-\frac{2((\alpha\wedge\beta)-s)}{2\alpha+2\gamma+d}},
\end{equation*}
whence the claim follows, for every $M_n\to\infty$, from Markov's inequality
\end{proof}

\section*{Acknowledgments}
S.F.~and M.G.~are also affiliated to Collegio Carlo Alberto, Turin.

\section*{Funding}
S.F.~and M.G.~gratefully acknowledge the financial support from the Italian Ministry of Education, University and Research (MIUR), “Dipartimenti di Eccellenza” grant 2023-2027. S.F.~and M.G.~were also partially supported by Collegio Carlo Alberto, Turin.

\bibliography{References}

@article{brascamp1976extensions,
  author  = {Brascamp, Herm Jan and Lieb, Elliott H.},
  title   = {On extensions of the {Brunn--Minkowski} and
             {Pr{\'e}kopa--Leindler} theorems, including inequalities
             for log-concave functions, and with an application to the
             diffusion equation},
  journal = {Journal of Functional Analysis},
  volume  = {22},
  number  = {4},
  pages   = {366--389},
  year    = {1976},
  doi     = {10.1016/0022-1236(76)90004-5}
}

@book{nickl2023bayesian,
  author    = {Nickl, Richard},
  title     = {Bayesian Non-linear Statistical Inverse Problems},
  series    = {Zurich Lectures in Advanced Mathematics},
  publisher = {EMS Press},
  address   = {Berlin},
  year      = {2023},
  doi       = {10.4171/ZLAM/30}
}

@article{adavani2010fast,
  title={Fast algorithms for source identification problems with elliptic {PDE} constraints},
  author={Adavani, Santi S. and Biros, George},
  journal={SIAM Journal on Imaging Sciences},
  volume={3},
  number={4},
  pages={791--808},
  year={2010},
  doi={10.1137/080738064}
}

@article{katzourakis2019regularization,
  title={An {$L^\infty$} regularization strategy to the inverse source identification problem for elliptic equations},
  author={Katzourakis, Nikos},
  journal={SIAM Journal on Mathematical Analysis},
  volume={51},
  number={2},
  pages={1349--1370},
  year={2019},
  doi={10.1137/18M1226373}
}

@article{hardle1989investigating,
  title={Investigating smooth multiple regression by the method of average derivatives},
  author={H{\"a}rdle, Wolfgang and Stoker, Thomas M.},
  journal={Journal of the American Statistical Association},
  volume={84},
  number={408},
  pages={986--995},
  year={1989},
  doi={10.1080/01621459.1989.10478863}
}

@article{genovese2016non,
  title={Non-parametric inference for density modes},
  author={Genovese, Christopher R and Perone-Pacifico, Marco and Verdinelli, Isabella and Wasserman, Larry},
  journal={Journal of the Royal Statistical Society Series B: Statistical Methodology},
  volume={78},
  number={1},
  pages={99--126},
  year={2016},
}

@article{comaniciu2002mean,
  title={Mean shift: A robust approach toward feature space analysis},
  author={Comaniciu, Dorin and Meer, Peter},
  journal={IEEE Transactions on Pattern Analysis and Machine Intelligence},
  volume={24},
  number={5},
  pages={603--619},
  year={2002},
  doi={10.1109/34.1000236}
}

@inproceedings{song2019generative,
  title={Generative modeling by estimating gradients of the data distribution},
  author={Song, Yang and Ermon, Stefano},
  booktitle={Advances in Neural Information Processing Systems},
  volume={32},
  year={2019}
}

@article{singh1977improvement,
  title={Improvement on some known nonparametric uniformly consistent estimators of derivatives of a density},
  author={Singh, Radhey S},
  journal={The Annals of Statistics},
  volume={5},
  number={2},
  pages={394--399},
  year={1977},
  doi={10.1214/aos/1176343805}
}

@article{bhattacharya1967estimation,
  title={Estimation of a probability density function and its derivatives},
  author={Bhattacharya, P. K.},
  journal={Sankhy\={a}: The Indian Journal of Statistics, Series A},
  volume={29},
  number={4},
  pages={373--382},
  year={1967}
}

@article{stone1982optimal,
  title={Optimal global rates of convergence for nonparametric regression},
  author={Stone, Charles J.},
  journal={The Annals of Statistics},
  volume={10},
  number={4},
  pages={1040--1053},
  year={1982},
  doi={10.1214/aos/1176345969}
}

@article{rousseau2011asymptotic,
  title={Asymptotic behaviour of the posterior distribution in overfitted mixture models},
  author={Rousseau, Judith and Mengersen, Kerrie},
  journal={Journal of the Royal Statistical Society Series B: Statistical Methodology},
  volume={73},
  number={5},
  pages={689--710},
  year={2011},
}

@article{gine2011rates,
  title={Rates of contraction for posterior distributions in ${L}^r$-metrics, $1\le r\le\infty$},
  author={Gin{\'e}, Evarist and Nickl, Richard},
  journal={The Annals of Statistics},
  volume={39},
  number={6},
  pages={2883--2911},
  year={2011},
}

@article{doob1949application,
  title={Application of the theory of martingales},
  author={Doob, Joseph L.},
  journal={Le calcul des probabilites et ses applications},
  pages={23--27},
  year={1949},
}

@book{vaart2000asymptotic,
  title={Asymptotic statistics},
  author={Van der Vaart, Aad W},
  volume={3},
  year={2000},
  publisher={Cambridge university press}
}

@article{pinelis1994optimum,
  author  = {Pinelis, Iosif},
  title   = {Optimum Bounds for the Distributions of Martingales
             in {B}anach Spaces},
  journal = {The Annals of Probability},
  year    = {1994},
  volume  = {22},
  number  = {4},
  pages   = {1679--1706},
  doi     = {10.1214/aop/1176988477}
}

@book{bogachev1998gaussian,
  author    = {Bogachev, Vladimir I.},
  title     = {Gaussian Measures},
  publisher = {American Mathematical Society},
  series    = {Mathematical Surveys and Monographs},
  volume    = {62},
  year      = {1998}
}

@article{benamou2000computational,
  title={A computational fluid mechanics solution to the Monge-Kantorovich mass transfer problem},
  author={Benamou, Jean-David and Brenier, Yann},
  journal={Numerische Mathematik},
  volume={84},
  number={3},
  pages={375--393},
  year={2000},
}

@article{kolesnikov2016riemannian,
  title={Riemannian metrics on convex sets with applications to {P}oincar\'e and log-{S}obolev inequalities},
  author={Kolesnikov, Alexander V. and Milman, Emanuel},
  journal={Calculus of Variations and Partial Differential Equations},
  volume={55},
  number={4},
  pages={77},
  year={2016},
  doi={10.1007/s00526-016-1018-3}
}

@article{bakry2008simple,
  title={A simple proof of the {P}oincar\'e inequality for a large class of probability measures including the log-concave case},
  author={Bakry, Dominique and Barthe, Franck and Cattiaux, Patrick and Guillin, Arnaud},
  journal={Electronic Communications in Probability},
  volume={13},
  pages={60--66},
  year={2008},
  doi={10.1214/ECP.v13-1352}
}

@book{bakry2014analysis,
  title={Analysis and Geometry of Markov Diffusion Operators},
  author={Bakry, Dominique and Gentil, Ivan and Ledoux, Michel},
  series={Grundlehren der mathematischen Wissenschaften},
  volume={348},
  publisher={Springer},
  year={2014},
  doi={10.1007/978-3-319-00227-9}
}

@book{ledoux1991probability,
  title={Probability in {B}anach Spaces: isoperimetry and processes},
  author={Ledoux, Michel and Talagrand, Michel},
  volume={23},
  year={1991},
  publisher={Springer}
}

@book{ambrosio2008gradient,
  title={Gradient flows: in metric spaces and in the space of probability measures},
  author={Ambrosio, Luigi and Gigli, Nicola and Savar{\'e}, Giuseppe},
  publisher={Springer Science Business \& Media},
  year={2008}
}

@article{dolera2024besov,
  title={On strong posterior contraction rates for {B}esov--{L}aplace priors in the white noise model},
  author={Dolera, Emanuele and Favaro, Stefano and Giordano, Matteo},
  journal={arXiv preprint arXiv:2411.06981},
  year={2024},
  eprint={2411.06981},
  archivePrefix={arXiv}
}

@article{dolera2020uniform,
  title={On uniform continuity of posterior distributions},
  author={Dolera, Emanuele and Mainini, Edoardo},
  journal={Statistics \& Probability Letters},
  volume={157},
  pages={108627},
  year={2020}
}

@article{dolera2023lipschitz,
  title={{L}ipschitz continuity of probability kernels in the optimal transport framework},
  author={Dolera, Emanuele and Mainini, Edoardo},
  journal={Annales de l'Institut Henri Poincar{\'e}, Probabilit{\'e}s et Statistiques},
  volume={59},
  number={4},
  pages={1778--1812},
  year={2023},
  doi={10.1214/23-AIHP1389}
}

@book{taylor2011partial2,
  title={Partial differential equations II. Qualitative Studies of Linear Equations},
  author={Taylor, Michael E},
  year={2011},
  publisher={Springer, New York}
}

@article{gretton2012kernel,
  title={A kernel two-sample test},
  author={Gretton, Arthur and Borgwardt, Karsten M and Rasch, Malte J and Sch{\"o}lkopf, Bernhard and Smola, Alexander},
  journal={Journal of Machine Learning Research},
  volume={13},
  number={1},
  pages={723--773},
  year={2012},
}

@article{castillo2013nonparametric,
  title={Nonparametric {B}ernstein--von {M}ises theorem in {G}aussian white noise},
  author={Castillo, I. and Nickl, R.},
  journal={The Annals of Statistics},
  volume={41},
  number={4},
  pages={1999--2028},
  year={2013}
}

@article{vaart2008rates,
  title={Rates of contraction of posterior distributions based on {G}aussian process priors},
  author={van der Vaart, A.~W. and van Zanten, J.~H.},
  journal={The Annals of Statistics},
  volume={36},
  number={3},
  pages={1435--1463},
  year={2008},
}

@article{cavalier2008nonparametric,
  title={Nonparametric statistical inverse problems},
  author={Cavalier, Laurent},
  journal={Inverse Problems},
  volume={24},
  number={3},
  pages={034004},
  year={2008},
  doi={10.1088/0266-5611/24/3/034004}
}

@article{natterer1984error,
  title={Error bounds for {T}ikhonov regularization in {H}ilbert scales},
  author={Natterer, Frank},
  journal={Applicable Analysis},
  volume={18},
  number={1--2},
  pages={29--37},
  year={1984},
}

@book{horn2012matrix,
  title={Matrix analysis},
  author={Horn, Roger A and Johnson, Charles R},
  year={2012},
  publisher={Cambridge university press}
}

@book{ghosal2017fundamentals,
  title={Fundamentals of nonparametric {B}ayesian inference},
  author={Ghosal, Subhashis and Van der Vaart, Aad W},
  volume={44},
  year={2017},
  publisher={Cambridge University Press}
}

@article{agapiou2024heavy,
  title={Heavy-tailed {B}ayesian nonparametric adaptation},
  author={Agapiou, Sergios and Castillo, Isma{\"e}l},
  journal={The Annals of Statistics},
  volume={52},
  number={4},
  pages={1433--1459},
  year={2024},
}

@book{wong2001asymptotic,
  title={Asymptotic approximations of integrals},
  author={Wong, Roderick},
  year={2001},
  publisher={SIAM}
}

@article{dolera2024strong,
  title={Strong posterior contraction rates via {W}asserstein dynamics},
  author={Dolera, Emanuele and Favaro, Stefano and Mainini, Edoardo},
  journal={Probability Theory and Related Fields},
  volume={189},
  number={1},
  pages={659--720},
  year={2024},
}

@article{shang2018gaussian,
  title={{G}aussian approximation of general non-parametric posterior distributions},
  author={Shang, Zuofeng and Cheng, Guang},
  journal={Information and Inference: A Journal of the IMA},
  volume={7},
  number={3},
  pages={509--529},
  year={2018},
  doi={10.1093/imaiai/iax017}
}

@book{evans1998partial,
  title={Partial Differential Equations},
  author={Evans, L.~C.},
  isbn={9780821807729},
  lccn={97041033},
  series={Graduate studies in mathematics},
  url={https://books.google.it/books?id=YHJfvwEACAAJ},
  year={1998},
  publisher={American Mathematical Society}
}

@article{giordano2020bernstein,
  title={Bernstein--von Mises Theorems and Uncertainty Quantification for Linear Inverse Problems},
  author={Giordano, Matteo and Kekkonen, Hanne},
  journal={SIAM/ASA Journal on Uncertainty Quantification},
  volume={8},
  number={1},
  pages={342--373},
  year={2020},
  publisher={SIAM}
}

@article{sriperumbudur2017density,
  title={Density estimation in infinite dimensional exponential families},
  author={Sriperumbudur, Bharath and Fukumizu, Kenji and Gretton, Arthur and Hyv{\"a}rinen, Aapo and Kumar, Revant},
  journal={Journal of Machine Learning Research},
  volume={18},
  number={57},
  pages={1--59},
  year={2017}
}

@article{knapik2011bayesian,
  title={{B}ayesian inverse problems with {G}aussian priors},
  author={Knapik, Bartek T. and van der Vaart, Aad W. and van Zanten, J. Harry},
  journal={The Annals of Statistics},
  volume={39},
  number={5},
  pages={2626--2657},
  year={2011}
}

@article{knapik2016bayes,
  title={Bayes procedures for adaptive inference in inverse problems for the white noise model},
  author={Knapik, Bartek T. and Szab{\'o}, Botond T. and van der Vaart, Aad W. and van Zanten, J. Harry},
  journal={Probability Theory and Related Fields},
  volume={164},
  number={3--4},
  pages={771--813},
  year={2016},
  doi={10.1007/s00440-015-0619-7}
}

@article{gugushvili2020bayesian,
  title={{B}ayesian linear inverse problems in regularity scales},
  author={Gugushvili, Shota and van der Vaart, Aad W. and Yan, Dong},
  journal={Annales de l'Institut Henri Poincar{\'e}, Probabilit{\'e}s et Statistiques},
  volume={56},
  number={3},
  pages={2081--2107},
  year={2020}
}

@book{brown1986fundamentals,
  title={Fundamentals of Statistical Exponential Families with Applications in Statistical Decision Theory},
  author={Brown, Lawrence D.},
  series={IMS Lecture Notes--Monograph Series},
  volume={9},
  publisher={Institute of Mathematical Statistics},
  address={Hayward, CA},
  year={1986}
}

@article{pistone1995infinite,
  title={An infinite-dimensional geometric structure on the space of all the probability measures equivalent to a given one},
  author={Pistone, Giovanni and Sempi, Carlo},
  journal={The Annals of Statistics},
  volume={23},
  number={5},
  pages={1543--1561},
  year={1995}
}

@incollection{fukumizu2009exponential,
  title={Exponential Manifold by Reproducing Kernel {H}ilbert Spaces},
  author={Fukumizu, Kenji},
  booktitle={Algebraic and Geometric Methods in Statistics},
  editor={Gibilisco, Paolo and Riccomagno, Eva and Rogantin, Maria Piera and Wynn, Henry P.},
  pages={291--306},
  publisher={Cambridge University Press},
  address={Cambridge},
  year={2009},
  doi={10.1017/CBO9780511642401.019}
}

@article{canu2006kernel,
  title={Kernel methods and the exponential family},
  author={Canu, St{\'e}phane and Smola, Alex},
  journal={Neurocomputing},
  volume={69},
  number={7--9},
  pages={714--720},
  year={2006}
}

@article{scricciolo2006convergence,
  title={Convergence rates for {B}ayesian density estimation of infinite-dimensional exponential families},
  author={Scricciolo, Catia},
  journal={The Annals of Statistics},
  volume={34},
  number={6},
  pages={2897--2920},
  year={2006}
}

@article{rivoirard2012posterior,
  title={Posterior concentration rates for infinite dimensional exponential families},
  author={Rivoirard, Vincent and Rousseau, Judith},
  journal={Bayesian Analysis},
  volume={7},
  number={2},
  pages={311--334},
  year={2012},
  doi={10.1214/12-BA710}
}

@article{krein1966scales,
  title={Scales of {B}anach spaces},
  author={Krein, Selim G. and Petunin, Ju. I.},
  journal={Russian Mathematical Surveys},
  volume={21},
  number={2},
  pages={85--159},
  year={1966}
}

@book{engl1996regularization,
  title={Regularization of Inverse Problems},
  author={Engl, Heinz W. and Hanke, Martin and Neubauer, Andreas},
  series={Mathematics and its Applications},
  volume={375},
  publisher={Kluwer Academic Publishers},
  address={Dordrecht},
  year={1996}
}

@article{mathe2003geometry,
  title={Geometry of linear ill-posed problems in variable {H}ilbert scales},
  author={Math{\'e}, Peter and Pereverzev, Sergei V.},
  journal={Inverse Problems},
  volume={19},
  number={3},
  pages={789--803},
  year={2003}
}

@article{fukumizu2024estimation,
  title={Estimation with infinite-dimensional exponential family and {F}isher divergence},
  author={Fukumizu, Kenji},
  journal={Information Geometry},
  volume={7},
  pages={609--622},
  year={2024},
  doi={10.1007/s41884-023-00122-z}
}

@book{gine2016mathematical,
  title={Mathematical Foundations of Infinite-Dimensional Statistical Models},
  author={Gin{\'e}, Evarist and Nickl, Richard},
  series={Cambridge Series in Statistical and Probabilistic Mathematics},
  publisher={Cambridge University Press},
  address={New York},
  year={2016}
}

@article{kirichenko2015optimality,
  title={Optimality of {P}oisson processes intensity learning with {G}aussian processes},
  author={Kirichenko, Alisa and van Zanten, Harry},
  journal={Journal of Machine Learning Research},
  volume={16},
  pages={2909--2919},
  year={2015}
}

@article{belitser2015rate,
  title={Rate-optimal {B}ayesian intensity smoothing for inhomogeneous {P}oisson processes},
  author={Belitser, Eduard N. and Serra, Paulo and van Zanten, Harry},
  journal={Journal of Statistical Planning and Inference},
  volume={166},
  pages={24--35},
  year={2015},
  doi={10.1016/j.jspi.2014.03.009}
}

@article{donnet2017posterior,
  title={Posterior concentration rates for counting processes with {A}alen multiplicative intensities},
  author={Donnet, Sophie and Rivoirard, Vincent and Rousseau, Judith and Scricciolo, Catia},
  journal={Bayesian Analysis},
  volume={12},
  number={1},
  pages={53--87},
  year={2017}
}

@article{belitser2003adaptive,
  title={Adaptive {B}ayesian inference on the mean of an infinite-dimensional normal distribution},
  author={Belitser, Eduard and Ghosal, Subhashis},
  journal={The Annals of Statistics},
  volume={31},
  number={2},
  pages={536--559},
  year={2003},
  doi={10.1214/aos/1051027880}
}

@book{kutoyants1998statistical,
  title={Statistical Inference for Spatial {P}oisson Processes},
  author={Kutoyants, Yury A.},
  series={Lecture Notes in Statistics},
  volume={134},
  publisher={Springer-Verlag},
  address={New York},
  year={1998}
}

@book{hjort2010bayesian,
  title={Bayesian Nonparametrics},
  editor={Hjort, Nils Lid and Holmes, Chris and M{\"u}ller, Peter and Walker, Stephen G.},
  series={Cambridge Series in Statistical and Probabilistic Mathematics},
  publisher={Cambridge University Press},
  address={Cambridge},
  year={2010}
}

@book{muller2015bayesian,
  title={Bayesian Nonparametric Data Analysis},
  author={M{\"u}ller, Peter and Quintana, Fernando Andr{\'e}s and Jara, Alejandro and Hanson, Tim},
  series={Springer Series in Statistics},
  publisher={Springer},
  address={Cham},
  year={2015},
  doi={10.1007/978-3-319-18968-0}
}

@book{castillo2024bayesian,
  title={Bayesian Nonparametric Statistics: {\'E}cole d'{\'E}t{\'e} de Probabilit{\'e}s de Saint-Flour {LI} -- 2023},
  author={Castillo, Isma{\"e}l},
  series={Lecture Notes in Mathematics},
  volume={2358},
  publisher={Springer},
  address={Cham},
  year={2024},
  doi={10.1007/978-3-031-74035-0}
}

@article{diaconis1986consistency,
  title={On the consistency of {B}ayes estimates},
  author={Diaconis, Persi and Freedman, David},
  journal={The Annals of Statistics},
  volume={14},
  number={1},
  pages={1--26},
  year={1986}
}

@article{cox1993analysis,
  title={An analysis of {B}ayesian inference for nonparametric regression},
  author={Cox, Dennis D.},
  journal={The Annals of Statistics},
  volume={21},
  number={2},
  pages={903--923},
  year={1993}
}

@article{freedman1999bernstein,
  title={On the {B}ernstein--von {M}ises theorem with infinite-dimensional parameters},
  author={Freedman, David},
  journal={The Annals of Statistics},
  volume={27},
  number={4},
  pages={1119--1140},
  year={1999}
}

@article{schwartz1965bayes,
  title={On {B}ayes procedures},
  author={Schwartz, Lorraine},
  journal={Zeitschrift f{\"u}r Wahrscheinlichkeitstheorie und Verwandte Gebiete},
  volume={4},
  number={1},
  pages={10--26},
  year={1965}
}

@article{ghosal2000convergence,
  title={Convergence rates of posterior distributions},
  author={Ghosal, Subhashis and Ghosh, Jayanta K. and van der Vaart, Aad W.},
  journal={The Annals of Statistics},
  volume={28},
  number={2},
  pages={500--531},
  year={2000}
}

@article{shen2001rates,
  title={Rates of convergence of posterior distributions},
  author={Shen, Xiaotong and Wasserman, Larry},
  journal={The Annals of Statistics},
  volume={29},
  number={3},
  pages={687--714},
  year={2001}
}

@article{ghosal2007convergence,
  title={Convergence rates of posterior distributions for noniid observations},
  author={Ghosal, Subhashis and van der Vaart, Aad},
  journal={The Annals of Statistics},
  volume={35},
  number={1},
  pages={192--223},
  year={2007}
}

@article{walker2007rates,
  title={On rates of convergence for posterior distributions in infinite-dimensional models},
  author={Walker, Stephen G. and Lijoi, Antonio and Pr{\"u}nster, Igor},
  journal={The Annals of Statistics},
  volume={35},
  number={2},
  pages={738--746},
  year={2007},
  doi={10.1214/009053606000001361}
}

@article{castillo2014supremum,
  title={On {B}ayesian supremum norm contraction rates},
  author={Castillo, Isma{\"e}l},
  journal={The Annals of Statistics},
  volume={42},
  number={5},
  pages={2058--2091},
  year={2014},
  doi={10.1214/14-AOS1253}
}

@article{yoo2016supremum,
  title={Supremum norm posterior contraction and credible sets for nonparametric multivariate regression},
  author={Yoo, William Weimin and Ghosal, Subhashis},
  journal={The Annals of Statistics},
  volume={44},
  number={3},
  pages={1069--1102},
  year={2016},
  doi={10.1214/15-AOS1398}
}

@article{yoo2018adaptive,
  title={Adaptive supremum norm posterior contraction: Wavelet spike-and-slab and anisotropic {B}esov spaces},
  author={Yoo, William Weimin and Rivoirard, Vincent and Rousseau, Judith},
  journal={arXiv preprint arXiv:1708.01909},
  year={2018},
  eprint={1708.01909},
  archivePrefix={arXiv}
}

@article{liu2026optimal,
  title={Optimal plug-in {G}aussian processes for modeling derivatives},
  author={Liu, Zejian and Li, Meng},
  journal={Journal of the American Statistical Association},
  year={2026},
  doi={10.1080/01621459.2026.2612778},
  note={Published online}
}

@article{shen2017posterior,
  title={Posterior contraction rates of density derivative estimation},
  author={Shen, Weining and Ghosal, Subhashis},
  journal={Sankhy\={a}: The Indian Journal of Statistics, Series A},
  volume={79},
  number={2},
  pages={336--354},
  year={2017},
  doi={10.1007/s13171-017-0105-7}
}

@article{comte2020optimal,
  title={Optimal adaptive estimation on $\mathbb{R}$ or $\mathbb{R}^+$ of the derivatives of a density},
  author={Comte, Fabienne and Duval, Celine and Sacko, Ousmane},
  journal={Mathematical Methods of Statistics},
  volume={29},
  number={1},
  pages={1--31},
  year={2020}
}

@article{fischer2020sobolev,
  title={Sobolev norm learning rates for regularized least-squares algorithms},
  author={Fischer, Simon and Steinwart, Ingo},
  journal={Journal of Machine Learning Research},
  volume={21},
  number={205},
  pages={1--38},
  year={2020},
  url={https://jmlr.org/papers/v21/19-734.html}
}

@inproceedings{wibisono2024optimal,
  title={Optimal score estimation via empirical {B}ayes smoothing},
  author={Wibisono, Andre and Wu, Yihong and Yang, Kaylee Yingxi},
  booktitle={Proceedings of the Thirty-Seventh Conference on Learning Theory},
  series={Proceedings of Machine Learning Research},
  volume={247},
  pages={4958--4991},
  publisher={PMLR},
  year={2024},
  url={https://proceedings.mlr.press/v247/wibisono24a.html}
}

@article{bonis2026optimal,
  title={Optimal score function estimation via derivatives constraints},
  author={Bonis, Thomas and Pham Ngoc, Thanh Mai and Tran, Viet Chi},
  journal={arXiv preprint arXiv:2606.19084},
  year={2026},
  doi={10.48550/arXiv.2606.19084}
}

@article{ghosh2025stein,
  title={Stein's unbiased risk estimate and {Hyv\"arinen}'s score matching},
  author={Ghosh, Sulagna and Ignatiadis, Nikolaos and Koehler, Frederic and Lee, Amber},
  journal={arXiv preprint arXiv:2502.20123},
  year={2025},
  doi={10.48550/arXiv.2502.20123}
}

\end{document}